%% file: Equality_new.tex
\documentclass[11pt]{article}
\usepackage[letterpaper,margin=1.00in]{geometry}
\usepackage{float}
\usepackage{calc}

\input{macros.tex}

\newcommand{\realized}[5]{$#1 \ \frac{#3}{#4,#5} \ #2$}
\newcommand{\classs}[1]{[#1]}

\newcommand{\sizeofarray}{(2\ell+1) \times (\Delta-2)}
\title{\vspace{-1cm} Deterministic Distributed Algorithms and Measurable Combinatorics on $\Delta$-Regular Forests\footnote{This paper is an extension of some parts of the conference paper ``Local Problems on forests from the Perspectives of Distributed Algorithms, Finitary Factors, and Descriptive Combinatorics (arXiv:2106.02066)'' presented at the 13th Innovations in Theoretical Computer Science Conference (ITCS 2022).}}

\begin{document}
	
	\newcommand*\samethanks[1][\value{footnote}]{\footnotemark[#1]}
	
	\author{
		Sebastian Brandt \\
		\small CISPA Helmholtz Center for Information Security \\
		\small \texttt{brandt@cispa.de}\\
		\and 
		Yi-Jun Chang
		\\
		\small National University of Singapore \\
		\small \texttt{cyijun@nus.edu.sg} \\
		\and 
		Jan Greb\'{i}k
		\\
		\small University of Leipzig\\
		\small \texttt{grebikj@gmail.com }\\
		\and 
		Christoph Grunau
		\\
		\small ETH Z\"{u}rich \\
		\small \texttt{cgrunau@inf.ethz.ch}\\
		\and 
		V\'{a}clav Rozho\v{n}
		\\
		\small Charles University \\
		\small \texttt{vaclavrozhon@gmail.com}\\
		\and 
		Zolt\'{a}n Vidny\'{a}nszky
		\\
		\small E\"otv\"os Lor\'and University\\
		\small \texttt{zoltan.vidnyanszky@ttk.elte.hu}\\}

	\date{}
	
	\maketitle 
	\pagenumbering{gobble}

	\begin{abstract}

        We investigate the connections between the fields of distributed computing and measurable combinatorics by considering complexity classes of locally checkable labeling problems on regular forests.
        We show that the most important deterministic complexity classes from the LOCAL model of distributed computing exactly coincide with well-studied classes in measurable combinatorics.
        Namely, first we show that a locally checkable labeling problem admits a continuous solution if and only if it can be solved by a deterministic local algorithm with complexity $O(\log^* n)$.
        Second, our main result states that, surprisingly, a locally checkable labeling problem admits a Baire measurable solution if and only if it can be solved by a local algorithm with complexity $O(\log n)$. These theorems suggest the existence of deeper connections between the two frameworks.
        Furthermore, the latter result relies on a complete combinatorial characterization of the classes in question, and as a by-product, it shows that membership in these classes is decidable.

	\end{abstract}

	\pagenumbering{arabic}   
	
	\section{Introduction}
	
	This work is part of a research program that investigates the connections between the area of distributed computing and measurable combinatorics. In the past couple of years, it has been revealed that there is a surprising connection between these two fields. In fact, oftentimes the two communities have been working on the same type of questions and have been using the same type of methods, independently of each other. The first connections of this sort have been formalized in the seminal papers of Bernshteyn \cite{Bernshteyn2021LLL} (see also \cite{elek2018qualitative}). Since then, many exciting results confirmed the importance of this approach \cite{OlegLLL,Bernshteyn2021local=cont,grebik_rozhon2021toasts_and_tails,grebik_rozhon2021LCL_on_paths,brandt2021homomorphism,brandt_chang_grebik_grunau_rozhon_vidnyaszky2021LCLs_on_trees_descriptive,BernshteynVizing}, some of which also build a connection to the area of random processes.
	Combined with the work in each field such as \cite{linial92LOCAL,HolroydSchrammWilson2017FinitaryColoring,DetMarks,GJKS,GJKS2,brandt_etal2016LLL,brandt_grids}, a rich picture of a common theory of locality starts to emerge. Let us give a brief informal description of the two fields.

    \medskip

	The definition of the \emph{$\local$ model of distributed computing} by Linial~\cite{linial92LOCAL} was motivated by the desire to understand distributed algorithms in large networks. 
	Informally, we have the following setup.
	Let $G$ be a finite graph, where each vertex is imagined to be a computer, which knows the size of the graph $n$, and perhaps some other parameters like the maximum degree $\Delta$. Every computer runs the same algorithm. In order to break the symmetries, each computer possesses some unique information, such as a natural number in the case of deterministic algorithms or 
	a random bit string in the case of randomized ones. In one round, each vertex can exchange any message with its neighbors and can perform an arbitrary computation. The goal is to find a solution to a given graph theoretic problem (more precisely, LCL problem, see below) in as few communication rounds as possible. The theory of distributed algorithms and the $\local$ model is extremely rich and deep, see, e.g., \cite{barenboim2016locality,chang_kopelowitz_pettie2019exp_separation,chang_pettie2019time_hierarchy_trees_rand_speedup,ghaffari_kuhn_maus2017slocal,kuhn16_jacm,RozhonG19}. 
	
	\textit{Descriptive} or \textit{measurable combinatorics} is an emerging field on the border of combinatorics, logic, group theory, and ergodic theory.
	The highlights of the field include the recent results on the \emph{circle squaring problem} of Tarski \cite{measurablesquare,marksunger,JordanCircle}, or on the \emph{Banach-Tarski paradox} \cite{doughertyforeman,marksungerbaire}, see also \cite{laczk, marksunger,measurablesquare,doughertyforeman,marksungerbaire,gaboriau,KST,DetMarks,millerreducibility,conley2020borel,OlegLLL,Bernshteyn2021LLL}, and  \cite{kechris_marks2016descriptive_comb_survey,pikhurko2021descriptive_comb_survey} for surveys. The typical setup in measurable combinatorics is that we have a graph with uncountably many connected components, each being a countable graph of bounded degree (we remark that a considerable amount of work has been also done in the context when the degrees can be infinite, see, e.g., \cite{KST,benen,lecomte2009dichotomy}, but in this paper we will not consider such graphs), and we want to find a solution to a given graph theoretic problem with some additional regularity (measurability) properties. 
    We refer the reader who is less familiar with distributed computing (respectively, with measurable combinatorics) to \cref{subsec:LOCAL} and to the first part of \cref{subsec:descriptivecombinatorics}.

	The role of the common link is played by so-called \emph{locally checkable labeling (LCL) problems} on graphs, which is a class of problems, where the correctness of a solution can be checked locally, that is, in a fixed-size neighborhood of vertices (we provide the exact definition in \cref{subsec:local_problems}, and use this intuitive description for the rest of the introduction).
	Examples include classical problems from combinatorics such as proper vertex coloring, proper edge coloring, perfect matching, and finding a maximal independent set.

	One can order the collection of LCL problems into \emph{complexity classes,} based on the asymptotic number of rounds required to solve them in the $\local$ model, and based on the measurability properties of the solution in measurable combinatorics. Hence the natural question of comparing these classes arises. In the insightful papers \cite{Bernshteyn2021LLL,Bernshteyn2021local=cont}, Bernshteyn showed the equality of some of these classes: roughly speaking, fast local algorithms solving LCL problems yield measurable solutions, and conversely, in the context of Cayley graphs of countable groups continuous solutions give back fast local algorithms.

    \medskip
	
	In this paper, we focus on $\Delta$-regular acyclic graphs and their finite analogues (see \cref{def:Deltareg} for a precise definition).
	The motivation for studying regular forests stems from different sources:
	\begin{itemize}
		\item The case of trees (and forests) is one of the highlights of the theory of the LOCAL model. For example, several lower bounds in distributed computing are proven for trees \cite{balliu_et_al:binary_lcl,balliu2019LB,balliu21dominatingset,balliu22seek,balliu20rulingset,brandt_etal2016LLL,brandt19automatic,Brandt20tightLBmatching,ChangHLPU20,goeoes14_PODC}; a classification of all potential complexity classes of LCL problems is also known in this context (see \cref{thm:basicLOCAL}).
		
		\item In many of the aforementioned results of measurable combinatorics, including the Banach-Tarski paradox, graphs where each component is an infinite $\Delta$-regular tree appear naturally. Moreover, such forests are also studied in the area of ergodic theory \cite{Bowen1,Bowen2} and random processes \cite{backhausz,backhausz2,lyons2011perfect}.
		
		\item Even in the case of relatively simple non-acyclic graphs (e.g., grids) one encounters undecidability barriers, making complete characterization results impossible, see e.g., \cite{brandt_grids, GJKS}. Thus, the class of acyclic graphs is often the largest natural collection, which is non-trivial, but still has the possibility of reaching a complete understanding.
		
		\item Comparing the techniques coming from these perspectives in this simple setting reveals already deep and unexpected connections.
	
	\end{itemize}

	On a high-level, our main results are the following. First, we extend the correspondence of Bernshteyn \cite{Bernshteyn2021local=cont} to regular forests, that is, we show that the class of LCL problems that admit a continuous solution and a solution by a fast deterministic algorithm are the same. Second, we show that, surprisingly, there is an exact equality ``high up'' in the hierarchy, namely, the class of LCL problems on regular forests that admit a local algorithm of local deterministic complexity $O(\log n)$ coincides with the class of LCL problems that admit so called \emph{Baire} measurable solutions, see \cref{fig:new_picture}.

	\medskip
	
	Before we discuss our results in more details, we remark that \cite{brandt2021homomorphism,brandt2021classes} and this paper extend the conference paper \cite{brandt_chang_grebik_grunau_rozhon_vidnyaszky2021LCLs_on_trees_descriptive}.
	In addition to the two perspectives studied here, we study LCL problems from the perspective of random processes in \cite{brandt2021classes}.
	We refer the reader to \cite{brandt2021classes} or \cite{brandt_chang_grebik_grunau_rozhon_vidnyaszky2021LCLs_on_trees_descriptive} for more details of the ``big picture.''

	\begin{figure}
		\centering
		\includegraphics[width=0.8\textwidth]{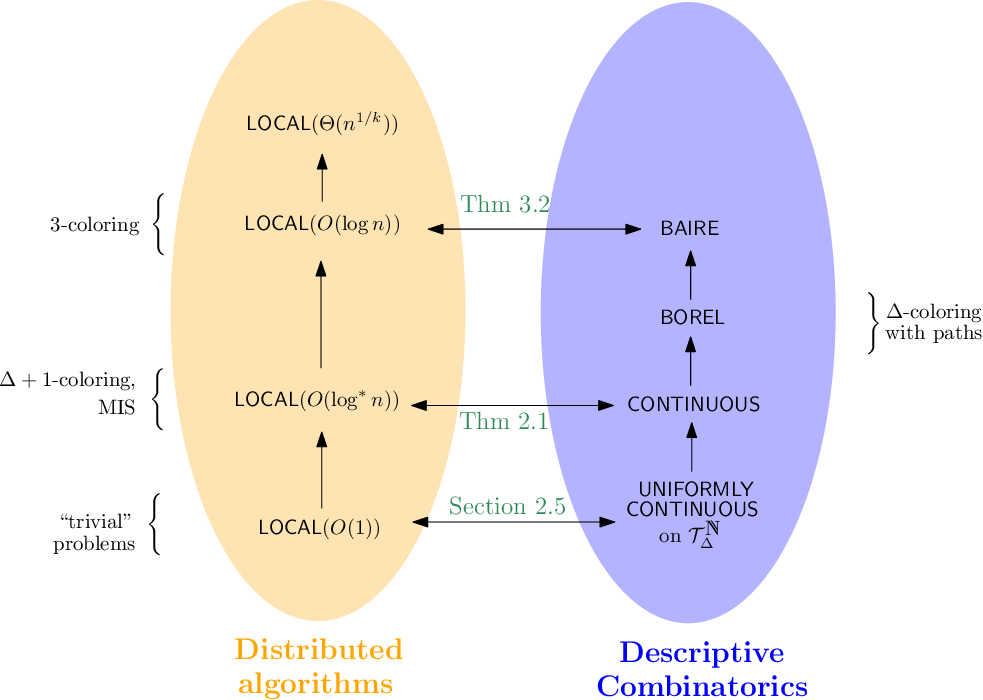}
		\caption{Complexity classes on $\Delta$-regular forests considered in the areas of distributed computing and measurable combinatorics. The left part shows deterministic complexity classes of distributed computing.
			The black arrows with two endpoints represent equalities between the complexity classes. 
			The black arrows with just one endpoint stand for proper inclusions between the classes.
			We highlight two phenomena. First, we have $\local(O(\log^* n))  = \continuous$, that is, there is a natural class of basic symmetry breaking problems in both areas. Second, we have $\local(O(\log n)) = \baire$, that is, the most powerful class in measurable combinatorics exactly matches the power of finite local construction with local complexity $O(\log n)$.
		}
		\label{fig:new_picture}
	\end{figure}

	\subsection*{The results}
	
	$\local(f(n))$ denotes the class of LCL problems of deterministic local complexity $f(n)$, and for a class of functions $\mathcal{F}$, we set $\local(\mathcal{F})=\bigcup_{f \in \mathcal{F}} \local(f)$ (see \cref{def:local_complexity}, and the discussion thereafter). The complexity classes of local problems that always admit a solution with the respective regularity properties that we consider in the paper are $\cont,\borel,\baire$; see  \cref{subsec:descriptivecombinatorics}.

	Once the complexity classes are defined, we can state the connections in a compact form, e.g., $\local(O(\log^* n))=\cont$.
	Indeed, as mentioned above, this equality was proven in \cite{Bernshteyn2021local=cont} for Cayley graphs  finitely generated countable groups.
	Our first result extends his theorem to the context of $\Delta$-regular forests.

	\begin{theorem}\label{main1}
		We have 
		$$\local(O(\log^* n))=\cont$$
		for LCL problems on $\Delta$-regular forests.
	\end{theorem}

    We note that \cite{Bernshteyn2021local=cont} includes the case of the $\Delta$-regular forest endowed with a proper $\Delta$-edge coloring.
	However, as the proper $\Delta$-edge coloring problem can be solved neither by a continuous nor by a Borel function \cite{DetMarks}, we see that these models are different.
	Importantly, the automorphism group of the infinite $\Delta$-regular tree without edge coloring is bigger than the one that preserves the coloring structure, hence the above theorem establishes an (a priori) significantly stronger result.
	
	Note also that the direction from left to right follows already from \cite{Bernshteyn2021LLL}, see also \cite{elek2018qualitative}.
	On a high level, it uses the speedup result of \cite{chang_kopelowitz_pettie2019exp_separation} that implies that every LCL problem of complexity $O(\log^* n)$ can be solved by first solving the \emph{$k$-distance coloring problem}\footnote{Recall that a vertex coloring is a solution to the $k$-distance coloring problem if no two vertices of graph distance at most $k$ get the same color. As we fix the maximum vertex degree to be at most $\Delta$, the number of colors that we are allowed to use is $O(\Delta^k)$, a quantity that does not depend on the size of the graph.} on the underlying graph, for some $k\in \mathbb{N}$, and then applying a constant local rule.
	These operations can be performed using continuous functions by \cite{KST}, see also \cite[Section~2]{Bernshteyn2021local=cont}.

	The other direction  should be interpreted as an equivalence of several models for continuous solutions, see \cref{sec:preliminaries} and \cref{t:maincontinuous}.
	Most notably it captures the $\{0,1\}$-model that has been studied extensively in the context of countable groups \cite{ColorinSeward,ST,GJKS,Bernshteyn2021local=cont}.
	Intuitively, a given LCL problem can be solved in the $\{0,1\}$-model if there is a map (local rule) $\fA$ with a subset of $\{0,1\}$-labeled neighborhoods of vertices of the infinite $\Delta$-regular tree $T_\Delta$ as a domain that assigns to each neighborhood from the domain some labeling such that the following holds.
	Whenever we are given a $\{0,1\}$-labeling of $T_\Delta$ that breaks all the symmetries of $T_\Delta$, then we can run $\fA$ at every vertex and the output given by $\fA$ is a solution to the given LCL problem.
	Here by ``running $\fA$'' we mean that each vertex explores its neighborhood until it encounters a $\{0,1\}$-labeled neighborhood that is in the domain of $\fA$; the output at the vertex is then the output of $\fA$ applied on this neighborhood of the vertex.
	The overall strategy to show that LCL problems solvable in this model have local deterministic complexity exactly $O(\log^* n)$ follows \cite{Bernshteyn2021local=cont}, in particular we use Bernshteyn's continuous version of the Lov\'asz Local Lemma, and the framework from \cite{GJKS}.
	However, as our underlying graph is not induced by a group action some additional non-trivial arguments are needed.
	It is an interesting open problem to study the $\{0,1\}$-labeling model on other classes of (structured) graphs.

	Next we describe the main result of the paper, namely the equality $\LOCAL(O(\log n))=\baire$. The key underlying concept is a combinatorial condition called \emph{$\ell$-fullness}. Roughly speaking, it is a property of the LCL problem which guarantees that we may pass to a subset of labels, say $S$, such that sufficiently (farther than $\ell$) spaced partial solutions on acyclic graphs can be extended to full ones (see \cref{sec:baire} for the definition). For example, $3$-coloring is $3$-full, while $2$-coloring is not $\ell$-full for any $\ell>0$.

    Bernshteyn has pointed \cite{Bernshteyn_work_in_progress} out that $\ell$-fullness completely characterizes the LCL problems in the class $\baire$. We complement this result as follows. 
	First, we prove that any $\ell$-full problem has local complexity $O(\log n)$.
	Interestingly, this implies that all complexity classes considered in \cite{brandt2021classes}, most notably LCL problems that admit factor of iid solution, are contained in $\local(O(\log n))$. 
	In particular, this yields that the existence of \emph{any} uniform algorithm implies a local distributed algorithm for the same problem of local complexity $O(\log n)$, see \cite{brandt2021classes} for details.
	We obtain this result via the well-known rake-and-compress decomposition~\cite{MillerR89}.  
	
	On the other hand, we prove that any problem in the class $\local(O(\log n))$ satisfies the $\ell$-full condition. 
	The proof combines a machinery  developed by Chang and Pettie~\cite{chang_pettie2019time_hierarchy_trees_rand_speedup}
	with additional ideas.
	In this proof we construct recursively a sequence of sets of rooted, layered, and partially labeled trees, where the partial labeling is computed by simulating any given $O(\log n)$-round distributed algorithm, 
	and then the set  $S$ meeting the $\ell$-full condition is constructed by considering all possible extensions of the partial labeling to complete a correct labeling of these trees. In summary, we have the following rather surprising equality.
	
	\begin{theorem} We have
		$$\local(O(\log n)) = \baire$$
		for LCL problems on $\Delta$-regular forests.
	\end{theorem}

    Note that the formal connections mentioned above utilize the fact that a very efficient local algorithm 
    sees only a small proportion of the vertices, in particular, we can run it on an infinite graph, pretending that it is finite.
    In contrast, an $O(\log n)$ algorithm will see vertices of degree $<\Delta$, or even leaves, hence running it on a $\Delta$-regular infinite forest seems to be an impossible task on the first sight. Thus, the above equality provides a deeper structural link between the two contexts.

	The combinatorial characterization of the local complexity class $\local(O(\log n))$ on $\Delta$-regular forests is interesting from the perspective of distributed computing alone. 
	This result can be seen as a part of a large research program aiming at the combinatorial classification of possible local complexities on various graph classes~\cite{balliu2020almost_global_problems,brandt_etal2016LLL,chang_kopelowitz_pettie2019exp_separation,chang_pettie2019time_hierarchy_trees_rand_speedup,chang2020n1k_speedups,Chang_paths_cycles,balliu2021_rooted_trees,balliu_et_al:binary_lcl,balliu2018new_classes-loglog*-log*}. 
	That is, we wish not only to understand the possible complexity classes (see the left part of \cref{fig:new_picture} for the possible local complexity classes on regular forests), but also to find combinatorial characterizations of problems in those classes that allow us to  efficiently decide for a given problem which class it belongs to.  
	As mentioned above, even for grids with input labels, it is \emph{undecidable} whether a given local problem can be solved in $O(1)$ rounds~\cite{naorstockmeyer,brandt_grids}, since
	local problems on grids can be used to simulate a Turing machine.
	This undecidability result does not apply to paths and forests, hence for these graph classes it is still hopeful that we can find simple and useful characterizations for different classes of distributed problems. 

    \medskip

	The paper is organized as follows.
	In \cref{sec:preliminaries}, we recall all the notation, definitions, and standard results.
	In \cref{sec:Continuous} we demonstrate the equality between the classes $\local(O(\log^* n))$ and $\cont$. 
	In \cref{sec:baire} and \cref{sec:implieslfull}, we prove the equality of the classes $\local(O(\log n))$ and $\baire$.

\section{Preliminaries}
\label{sec:preliminaries}

In this section, we explain the setup we work with, the main definitions, techniques and results. The class of graphs that we consider in this work are either infinite $\Delta$-regular forests, or their finite analogues that we define formally in \cref{subsec:local_problems}.
	We always assume that $\Delta > 2$, since
	the case $\Delta = 2$, that is, studying paths, behaves differently and is easier to understand, see \cite{grebik_rozhon2021LCL_on_paths}. 
	Unless stated otherwise, we do not consider any additional structure on the graphs.
	
	A graph $G$ is a pair $(V(G),E(G))$, where $V(G)$ is a vertex set and $E(G)\subseteq \binom{V(G)}{2}$.
	We write $B_G(v,t)$ for the $t$-neighborhood of a vertex $v\in V(G)$.
	We denote by $N_G(A)$ the set of vertices in $G$ that are connected by an edge with some element from $A\subseteq V(G)$.
	The graph distance $\distance_G$ is defined in the standard way.
	For $L\in \mathbb{N}$, we define the $L$-power graph of $G$ as the graph on the same vertex set where $v,w\in V(G)$ form an edge if and only if $0<\distance_G(v,w)\le L$.
	If we talk about forests, that is, acyclic graphs, we use $T$ instead of $G$.
	To emphasize the difference between Borel (or continuous) graphs and discrete (at most countable) graphs, we use $\fG$ and $\fT$ in the former case.
	Vertices of discrete graphs are denoted by $v,w,\dots$ and of Borel graphs by $x,y,\dots$ or $\alpha,\beta,\dots$.
	
	We write $T_\Delta$ for the $\Delta$-regular tree and $\mathfrak{r}$ for a fixed vertex of $T_\Delta$ that we call the \emph{root} of $T_\Delta$.
	As $T_\Delta$ is the main object of our study, we omit the subscript in the notation above whenever we talk about $T_\Delta$ (that is, in the case of $B(\cdot,\cdot)$ and $\distance(\cdot,\cdot)$).
	Similarly, we use $v\in T_\Delta$ instead of $v\in V(T_\Delta)$.
	Recall also that for $k\in \mathbb{N}$ and $v\in T_\Delta$, we have $|B(v,k)|=1+\Delta\sum_{i=0}^{k-1}(\Delta-1)^i$; we mostly use the upper bound $2\Delta^k$ which holds for every $k\in \mathbb{N}$.
	
	A \emph{$k$-distance coloring} of $G$ is a coloring of the $k$th power of $G$ with $2\Delta^k$-colors, or equivalently, a coloring of $G$ where vertices of distance at most $k$ have different colors. Observe that the degree of the $k$th power of $G$ is bounded by $2\Delta^k-1$, hence using $2\Delta^k$-colors corresponds to the greedy coloring.

	For sets $X$ and $Y$, we use the notation $c:X \rightharpoonup Y$ for a partial function $c$ from $X$ to $Y$.
	For functions $f,g:\N \to \N$ we use the notation $f(n)=\Omega(g(n))$ if $\liminf_{n \to \infty} \frac{f(n)}{g(n)}>0$ and $f(n)=\Theta(g(n))$ if $f(n)=O(g(n))$ and $f(n)=\Omega(g(n)).$

\subsection{Local Problems on $\Delta$-regular forests}
\label{subsec:local_problems}

The problems we study in this work are locally checkable labeling (LCL) problems, which, roughly speaking, are problems that can be described via local constraints that have to be satisfied in a neighborhood of each vertex.
In the context of distributed algorithms, these problems were introduced in the seminal work by Naor and Stockmeyer~\cite{naorstockmeyer}, and have been studied extensively since.
In the modern formulation introduced in~\cite{brandt19automatic}, instead of labeling vertices or edges, LCL problems are described by labeling half-edges, i.e., pairs of a vertex and an adjacent edge.
This formulation is very general in that it not only captures vertex and edge labeling problems, but also others such as orientation problems, or combinations of all of these types.
Before we can provide this general definition of an LCL, we need to introduce some further definitions.

As mentioned above, one of the main motivations behind our work is to develop the unified theory of distributed computing and measurable combinatorics. In the context of our current investigation, it turns out that one has to modify the notion of finite graphs of degree $\leq \Delta$ to obtain a perfect analogue of $\Delta$-regular infinite forests in the finite setup. 
Namely, to handle vertices with degrees less than $\Delta$, we add ``virtual" half-edges to these. 

Let us start by formalizing the notion of a half-edge.

\begin{definition}[Half-edge]
	\label{def:half-edge}
	Given a graph $G=(V(G),E(G))$, a \emph{half-edge} is a pair $(v, e)$ where $v \in V(G)$ is a vertex, and either $e$ is an edge incident to $v$, or $e \not \in \binom{V(G)}{2}$.  We say that a half-edge $(v, e)$ is \emph{adjacent} to a vertex $w\in V(G)$ if $w = v$. 
	Equivalently, we say that $w$ is \emph{contained} in $(v,e)$ if $(v,e)$ is adjacent to $w$. 
	We say that $(v, e)$ \emph{belongs} to an edge $e'\in E(G)$ if $e' = e$.
	
	A half-edge is called \emph{real}, if it belongs to some edge, and it is \emph{virtual}, otherwise.		
\end{definition}

In a $\Delta$-regular forest, we will require that every vertex is contained in exactly $\Delta$-many half-edges, however, of course, not every such half-edge will be real.  See also \cref{fig:Delta_reg_tree}.

\begin{definition}[$\Delta$-regular forest]
	\label{def:Deltareg}	
	A \emph{$\Delta$-regular forest} $T$ is a triplet $(V(T), E(T), H(T))$, where 
	\begin{itemize}
		\item $(V(T),E(T))$ is an acyclic graph of maximum degree $\Delta$,
		\item  $H(T)=\Hr(T) \cup \Hv(T)$ consists of half-edges, where $\Hr(T)$ and $\Hv(T)$ stand for the collection of real and virtual half-edges, such that
		each vertex $v\in V(T)$ is contained in exactly $\Delta$-many half-edges, and if $(v,e),(v',e) \in \Hv(T)$, then $v=v'$,
		\item if $V(T)$ is infinite, then $\Hv(T)=\emptyset$, that is, $T$ is a disjoint union of copies of $T_\Delta$.
	\end{itemize}
	A \emph{half-edge labeling} is a function $c \colon H(T) \to \Sigma$ that assigns to each half-edge an element from some label set $\Sigma$. 
	
\end{definition}

\begin{figure}
	\centering
	\includegraphics{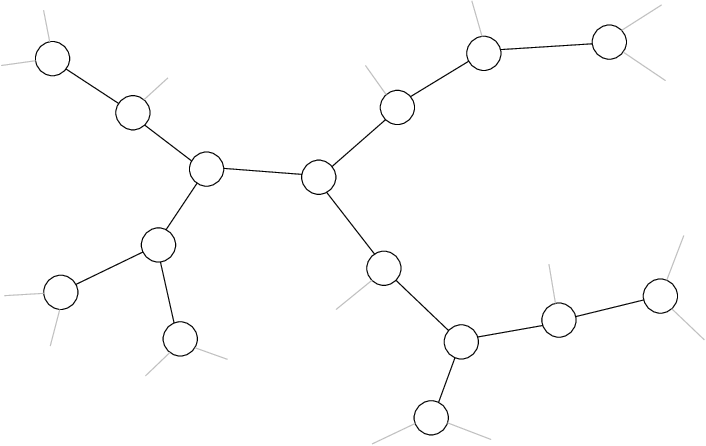}
	\caption{A $3$-regular tree on 15 vertices. Every vertex has $3$ half-edges but not all half-edges lead to a different vertex. }
	\label{fig:Delta_reg_tree}
\end{figure}

As we are considering forests in this work, each LCL problem can be described in a specific form that provides two lists, one describing all label combinations that are allowed on the half-edges adjacent to a vertex, and the other describing all label combinations that are allowed on the two half-edges belonging to an edge.\footnote{Every problem that can be described in the form given by Naor and Stockmeyer~\cite{naorstockmeyer} can be equivalently described as an LCL problem in this list form, by simply requiring each output label on some half-edge $h$ to encode all output labels in a suitably large (constant) neighborhood of $h$ in the form given in~\cite{naorstockmeyer}.}
We arrive at the following definition for LCL problems on $\Delta$-regular forests.\footnote{Note that the defined LCL problems do not allow so-called input labels.}

\begin{definition}[LCL problems on $\Delta$-regular forests]\label{def:LCL_trees}
	A \emph{locally checkable labeling problem}, or \emph{LCL} for short, is a triple $\Pi=(\Sigma,\fV,\fE)$, where $\Sigma$ is a finite set of labels, $\fV$ is a subset of unordered cardinality-$\Delta$ multisets\footnote{Recall that a multiset is a modification of the concept of sets, where repetition is allowed.} of labels from $\Sigma$, and $\fE$ is a subset of unordered cardinality-$2$ multisets of labels from $\Sigma$.
	
	We call $\fV$ and $\fE$ the \emph{vertex constraint} and \emph{edge constraint} of $\Pi$, respectively.
	Moreover, we call each multiset contained in $\fV$ a \emph{vertex configuration} of $\Pi$, and each multiset contained in $\fE$ an \emph{edge configuration} of $\Pi$.
	
	Let $T$ be a $\Delta$-regular forest and $c:H(T)\to \Sigma$ a half-edge labeling of $T$ with labels from $\Sigma$.
	We say that $c$ is a \emph{$\Pi$-coloring}, or, a \emph{solution} of $\Pi$, if, for each vertex $v$ of $T$, the multiset of labels assigned to the half-edges adjacent to $v$ is contained in $\fV$, and, for each edge $e$ of $T$, the cardinality-$2$ multiset of labels assigned to the half-edges belonging to $e$ is an element of $\fE$.
\end{definition}

    In the main body of the paper, we will refer to LCL problems on $\Delta$-regular forests simply as LCL problems. Let us illustrate the difference between our setting and the ``standard setting'' without virtual half-edges on the perfect matching problem.
	A standard definition of the perfect matching problem asks to find a subset of edges in such a way that each vertex is covered by exactly one edge. Notice that there is no local algorithm to solve this problem on the class of finite trees (without virtual half-edges and even restricting to graphs of even size). Indeed, already on the class of oriented paths with an even number of vertices, the unique perfect matching yields a 2-coloring: every vertex is incident to a matched edge, and we may color a vertex according to whether this matched edge is oriented towards it or away from it. Along the path the colors then alternate, so any local algorithm for perfect matching would immediately give a local algorithm for 2-coloring on oriented paths, contradicting the folklore fact that there is no efficient distributed algorithm for 2-coloring (see, e.g., \cite[Section 2.2.2]{suomela2014distributed}). One can also modify these examples to have higher degrees by adding a bounded amount of new vertices to each vertex on the path. 
	
	However, the natural version of the perfect matching problem in our setting is that every vertex needs to pick exactly one half-edge (real or virtual) in such a way that both endpoints of each edge are either picked or not picked.
	It follows from our main result (and also not hard to see directly) that this problem can be solved with an $O(\log n)$ local algorithm, if $\Delta>2$. 
	
	\begin{remark}\label{rem:equivLCL}
		The formalism with half-edges is convenient for subsequent analysis.
		However, there are at least two other ways to define finite $\Delta$-regular forests and LCL problems that would serve the same purpose
		(see also \cite{balliu2019hardness_homogeneous} for a closely related definition of homogeneous problems):
		
		{\bf (a)}
		Consider the class of graphs with vertices of degree $\Delta$ and vertices of degree $1$, and define the class of LCL problems so that we leave the degree $1$ vertices unconstrained with respect to the vertex constraint.
		To see the equivalence between these models, consider our definition of a finite $\Delta$-regular forest and append a leaf to the ``other side'' of each virtual edge.
		
		{\bf (b)}
		Consider the class of arbitrary finite forests of maximum degree $\Delta$, and define the class of LCL problems so that for vertices of degree smaller than $\Delta$, we require the multiset of labels assigned to the half-edges to be extendable to some multiset in $\fV$ of size $\Delta$.
		This model is also clearly equivalent to our model from the perspective of local complexity of LCL problems.
		Its main advantage is that it directly implies that the set of LCL problems on finite $\Delta$-regular forests is a \emph{subset} of the set of LCL problems on finite forest of degree at most $\Delta$, cf. \cref{thm:basicLOCAL}.
	\end{remark}

\subsection{The $\local$ model}\label{subsec:LOCAL}

In this section, we define local algorithms and local complexity.
Recall that our setup is informally the following. We are given a graph $G$, where each vertex is a computer, which knows the size of the graph $n$, and perhaps some other parameters like the maximum degree $\Delta$. In the case of randomized algorithms, each vertex has access to a private random bit string, while in the case of deterministic algorithms, each vertex is equipped with a unique identifier from a range polynomial in the size $n$ of the network. Every vertex must run the same algorithm. In one round, each vertex can exchange any message with its neighbors and can perform an arbitrary computation. The goal is to find a solution to a given LCL in as few communication rounds as possible. An algorithm is correct if and only if the collection of outputs at all vertices constitutes a correct solution to the problem. 

In this work we focus exclusively on \emph{deterministic} local algorithms. As the allowed message size is unbounded, a $t$-round $\local$ algorithm can be equivalently described as a function that maps $t$-neighborhoods (with identifiers) to outputs---the output of a vertex is then simply the output of this function applied to the $t$-neighborhood of this vertex. Hence, we arrive at the following definition.

\begin{definition}[Local algorithm]
	\label{def:local_algorithm} Let $\Sigma$ be a set of labels.
	A \emph{distributed local algorithm} $\fA$ of local complexity $t(n)$ is a function that takes as an input the (isomorphism type as a rooted graph) $t(n)$-neighborhood of a vertex $u$ together with the identifiers of each vertex in the neighborhood and the size of the graph $n$, and outputs a label from $\Sigma$ for each half-edge adjacent to $u$.
	
	Applying an algorithm $\fA$ on an input graph $G$ of size $n$ means that the function is applied to a $t(n)$-neighborhood $B_G(u, t(n))$ of each vertex $u$ of $G$. The output of the function is a labeling of the half-edges around the vertices. 
\end{definition}

\begin{definition}[Local complexity]
	\label{def:local_complexity}
	We say that an LCL problem $\Pi$ has a \emph{deterministic local complexity} $t(n)$ if there is a local algorithm $\fA$ of local complexity $t(n)$ such that when we run $\mathcal{A}$ on an input graph $G$, with each of its vertices having a unique identifier from $[n^{C}]$ for some absolute constant $C\geq 1$, $\fA$ always returns a correct solution to $\Pi$.  
	We write $\Pi \in \local( O(t(n)) )$. 
\end{definition}

We note for completeness that \emph{randomized local complexity} is defined similarly, the difference is that the unique identifiers at each vertex are replaced with sequences of independent random strings, and we allow the algorithm to fail with probability at most $1/n$.

Recall that $\log^*n$ stands for the iterated logarithm of $n$, that is, $\log^*n$ is the number of times the logarithm function must be iteratively applied to $n$ before the result is $\leq 1$. The threshold $\log^*n$ plays crucial role in distributed computing, as it is a classical result that graphs with degrees bounded by $\Delta$ can be $(\Delta+1)$-colored in $O(\log^*n)$-many rounds, or in our terminology, $(\Delta+1)$-coloring is in $\local(O(\log^*n))$. Note that in these bounds $\Delta$ is assumed to be constant, that is, the constants in the $O(\cdot)$ notation depend on $\Delta$.

\paragraph{Classification of Local Problems on $\Delta$-regular Forests}\footnote{Traditionally the case of acyclic graphs is referred to in the distributed computing literature as trees.}
There has been a lot of work done aiming to classify all possible local complexities of LCL problems on bounded degree finite forests. This classification was recently finished (see the citations in \cref{thm:basicLOCAL} below).
Even though the results in the literature are stated for the classical notion of forests with a degree bounded by constant, we note that it follows directly from the equivalent reformulation in \cref{rem:equivLCL}~(b) that the same picture emerges if we restrict ourselves even further to $\Delta$-regular forests, see also \cite{brandt2021classes}.
For our purposes, we only need to recall the classification in the deterministic setting.

\begin{theorem}[Classification of deterministic local complexities of LCL problems on $\Delta$-regular forests \cite{naorstockmeyer,chang_kopelowitz_pettie2019exp_separation,chang_pettie2019time_hierarchy_trees_rand_speedup,chang2020n1k_speedups,balliu2020almost_global_problems,balliu2019hardness_homogeneous,GrunauRozhonBrandt2022complexity_o(logstar(n))}]
	\label{thm:basicLOCAL}
	Let $\Pi$ be an LCL problem on $\Delta$-regular forests. Then the deterministic local complexity of $\Pi$ on $\Delta$-regular forests is one of the following:
	\begin{enumerate}
		\item \label{o(1)} $O(1)$,
		\item \label{logstarn}$\Theta(\log^*n)$,
		\item \label{logn} $\Theta(\log n)$,
		\item \label{theta} $\Theta(n^{1/k})$ for some $k \in \mathbb{N}$.
	\end{enumerate}
	Moreover, all classes are nonempty.
\end{theorem}

	The ``Moreover" part of \cref{thm:basicLOCAL} for classes \eqref{o(1)}-\eqref{logn} are witnessed by any trivial LCL problem, $\Delta+1$ coloring problem, and $\Delta$ coloring problem, respectively, see for instance \cite{brandt2021classes}.
	The fact that the classes in \eqref{theta} are non-empty as well is demonstrated in a subsequent work \cite{brandt2021classes}.
Let us mention that a classification is known also for randomized local complexities, but, as we do not need it here, we refer the reader to e.g.~\cite{brandt2021classes} for more details.

Deciding the  local complexity of a given LCL  on  bounded-degree forests is $\mathsf{EXPTIME}$-hard~\cite{chang2020n1k_speedups}. Therefore, to obtain simple and useful characterizations for each complexity class of LCL problems on bounded-degree forests, it is necessary that we restrict ourselves to special classes of LCL problems. As many natural LCL problems (e.g., the proper $k$-coloring problem for any $k$) on bounded-degree forests are also LCL problems on $\Delta$-regular forests, this further motivates the study of LCL problems on $\Delta$-regular forests. 

Let us also mention that in a subsequent work it was shown that there is a polynomial-time algorithm to decide the local complexity of a given LCL  on  $\Delta$-regular forests once the problem has complexity $\Omega(\log n)$, i.e., it falls in some of the cases \eqref{logn} or \eqref{theta} above ~\cite{balliu2022efficient}.

\subsection{Measurable combinatorics}
\label{subsec:descriptivecombinatorics}

Before we define formally the measurable combinatorics complexity classes, we give a high-level overview on their connection to distributed computing for the readers more familiar with the latter.

The complexity class that partially captures deterministic local complexity classes and hints at the additional computational power that one can get in the infinite setting is called $\borel$.
First, by a result of Kechris, Solecki and Todor\v{c}evi\'c \cite{KST} the \emph{maximal independent set problem} is in this class for any bounded degree Borel graph.\footnote{That is, it is possible to find a Borel maximal independent set; see the precise definitions later in this section.}
In particular, this yields that $\borel$ contains the class $\local(O(\log^* n))$ by the characterization of \cite{chang_kopelowitz_pettie2019exp_separation}, see \cite{Bernshteyn2021LLL}. 
Moreover, (see the definition below) $\borel$ is closed under countably many iterations of the operations of finding a maximal independent set in some power of the graph, and of applying a constant local rule that takes into account what has been constructed previously.
Therefore, the constructions in the class $\borel$ can use some unbounded, global, information. 

The ``truly'' local class in measurable combinatorics is the class $\CONT$.
As above, the inclusion $\local(O(\log^* n))\subseteq \CONT$ holds in most reasonable classes of bounded degree graphs, since it is possible to solve the maximal independent set problem in a continuous manner, see \cite[Lemma~2.2]{Bernshteyn2021local=cont}.
In the paper \cite{Bernshteyn2021local=cont}, Bernshteyn showed that, in the context of \emph{Cayley graphs} of countable groups, this class admits various equivalent definitions and the inclusion is in fact equality.

Another complexity class that we consider in this paper is the class $\baire$.
This is a topological relaxation of the class $\borel$, namely, an LCL is in the class $\baire$ if there is a Borel solution that can be incorrect on a topologically negligible set.
The main advantage of the class $\baire$ is that every bounded degree Borel graph admits a hierarchical decomposition--called \emph{toast}--in the class $\baire$.
The independence of colorings on a tree together with this structure allows for a combinatorial characterization of the class $\baire$, which was proven by Bernshteyn, see also \cref{sec:baire}.

\paragraph{Basic definitions}
A \emph{Polish space} is a separable, completely metrizable topological space. A prime example of a Polish space is the real numbers. A Polish space is called \emph{zero-dimensional} if it admits a basis consisting of clopen sets. The two most important examples of zero-dimensional Polish spaces are the collection of irrational numbers (as a subspace of $\mathbb{R
}$) and the Cantor set. In fact, the former captures the structure of general zero-dimensional Polish spaces, while the latter captures the structure of compact ones (see, \cite[7.B, 7.C]{kechrisclassical} for precise statements). Recall that the collection of \emph{Borel subsets} of a topological space is the minimal family of sets which contains all open sets and closed under taking complements and countable unions. The family of Borel sets has sufficient closure properties, while it still behaves intuitively, for example, such sets do not exhibit the paradoxes arising from the Axiom of Choice (e.g., Banach--Tarski).
Since the edge set of a graph on a space $X$ can be identified with a subset of $X^2$, which carries naturally a canonical topology or product Borel structure, we can talk about a graph being closed or Borel etc.

Recall that in a topological space $(X,\tau)$, a set $C\subseteq X$ is \emph{$\tau$-meager} (or meager if the topology is clear) if it can be covered by countably many $\tau$-closed sets with empty interior.
In Polish spaces, the collection of meager sets forms a $\sigma$-ideal, \cite[Section~8]{kechrisclassical}.
This provides the notion of small sets in the same way as does the $\sigma$-ideal of null sets in measure spaces.
A set $C\subseteq X$ is \emph{comeager} if $X\setminus C$ is meager.

A \emph{standard Borel space} is a pair $(X,\mathcal{B}(X))$ where $X$ is a set and $\mathcal{B}(X)$ is a $\sigma$-algebra of subsets of $X$, and $(X,\mathcal{B}(X))$ is isomorphic to $(Y,\mathcal{B}(Y))$, where $Y$ is Polish, and $\mathcal{B}(Y)$ is the collection of Borel subsets of $Y$. Let $(X,\mathcal{B}(X))$ be a standard Borel space. We say that a Polish topology $\tau$ on $X$ is \emph{compatible} if the fixed original $\sigma$-algebra on $X$ and the $\tau$-Borel $\sigma$-algebra are the same.

We refer the reader to \cite{pikhurko2021descriptive_comb_survey,kechris_marks2016descriptive_comb_survey,Bernshteyn2021LLL,kechrisclassical}, or to \cite[Introduction,~Section~4.1]{grebik_rozhon2021LCL_on_paths} and \cite[Section~7.1,~7.2]{grebik_rozhon2021toasts_and_tails} for intuitive examples and standard facts of descriptive set theory.

\subsubsection{$\cont$}\label{subsub:continuous}

We follow Bernshteyn \cite[Section 2]{Bernshteyn2021local=cont} and say that a graph $\mathcal{G}=(V(\mathcal{G}),E(\mathcal{G}))$, is a \emph{continuous graph}, if $V(\mathcal{G})$ is a zero-dimensional Polish space and the neighborhood $N_\fG(U)$ of any clopen set $U$ is clopen.
An easy argument shows that if $\mathcal{G}$ is locally finite and continuous, then the set $\{(x,y):\{x,y\} \in E(\mathcal{G})\}$ is closed in $V(\mathcal{G})^2\setminus \{(x,x):x \in V(\mathcal{G})\}$. In particular, if $V(\mathcal{G})$ is zero-dimensional, using the (open) factor map $V(\mathcal{G})^2\setminus \{(x,x):x \in V(\mathcal{G})\} \to \binom{V(G)}{2}$ we obtain a zero-dimensional Polish topology on $E(\mathcal{G})$.

\begin{definition}[Continuous forest]
	A \emph{continuous ($\Delta$-regular) forest} $\fT$ is an acyclic $\Delta$-regular continuous graph.
	We define $H(\fT)$ to be the set of all half-edges of $\fT$, i.e., the collection of pairs $(v,e)$, where $v \in V(\mathcal{T}), e \in E(\mathcal{T})$ and $e$ is incident to $v$. 
\end{definition}

Note that $H(\mathcal{T})$ can be identified with the ordered pairs coming from $E(\mathcal{T})$, hence carries a canonical zero-dimensional topology.

\begin{definition}[$\cont$]
	Let $\Pi=(\Sigma,\fV,\fE)$ be an LCL.
	We say that $\Pi$ is in the class $\cont$ if for every continuous forest $\fT$ there is a \emph{continuous function} $f:H(\fT)\to \Sigma$ that is a $\Pi$-coloring of $\fT$.
\end{definition}

We remark that this definition of the class $\cont$ is different than the one given in the introduction (which is the intuitive description of the $\{0,1\}$-model below).
In \cref{sec:Continuous}, we show that these definitions coincide.
In the next sections, we describe particular instances of continuous forests.

\paragraph{The space of labelings}
Recall that $T_\Delta$ is the $\Delta$-regular tree and we pick a distinguished vertex $\mathfrak{r}$ which we call \emph{the root}.
We write $\AutT$ for the group of all automorphisms of $T_\Delta$ and $\AutR$ for the subgroup that fixes $\mathfrak{r}$.
Endowed with pointwise convergence, or equivalently, subspace topology from ${T_\Delta}^{T_\Delta}$, $\AutT$ is a Polish group and $\AutR$ is a compact subgroup with countable index (see, \cite[Theorem 2.4.4]{gao2008invariant}).
We write $\id_{T_\Delta}$ for the identity in $\AutT$.

For a set $b$, we define $b^{T_\Delta}$ to be the space of all $b$-labelings of vertices of $T_\Delta$.
We endow $b^{T_\Delta}$ with the product topology, where $b$ is viewed as a discrete space.
The shift action $\AutT\curvearrowright b^{T_\Delta}$ is defined as 
$$g\cdot x(v)=x(g^{-1}(v)),$$
where $g\in \AutT$, $x\in b^{T_\Delta}$ and $v\in T_\Delta$.

\begin{definition}
	Write $b^{T_\Delta}/\sim$ for the quotient space of the action $\AutR\curvearrowright b^{T_\Delta}$.
	That is, $x\sim y$ if and only if there is $g\in \AutR$ so that $g\cdot x=y$.
	Given $x\in b^{T_\Delta}$, we let $[x]\in b^{T_\Delta}/\sim$ to be the $\sim$-equivalence class of $x$.
	
	The graph $\fG_{b,\Delta}$ on $b^{T_\Delta}/\sim$ is defined as follows: $[x]$ and $[y]$ form an edge in $\fG_{b,\Delta}$ if and only if there is a $g\in \AutT$ that moves the root to one of its neighbors so that $[g\cdot x]= [y]$.
\end{definition}

It is easy to see that $\fG_{b,\Delta}$ is not a continuous forest, as it contains cycles and even loops.
For convenience, we extend the notion of a continuous graph to this setting. The following statement along with the construction outlined above has first appeared in \cite{conley2013ultraproducts}.

\begin{proposition}\label{pr:ContTree1}
	Let $b$ be finite. The quotient space $b^{T_\Delta}/\sim$ is a compact zero-dimensional Polish space and the graph $\fG_{b,\Delta}$ is a continuous graph of degree at most $\Delta$.
\end{proposition}
\begin{proof}
	The space $b^{T_\Delta}$ is a compact metric space.
	Write $d$ for the standard compatible metric, that is, $d(x,y)=2^{-k-1}$ if the first vertex where $x$ and $y$ differ is at a distance $k$ from the root $\mathfrak{r}$.
	It is a standard fact that the action $\AutR\curvearrowright b^{T_\Delta}$ is continuous.
	Define the metric $d/\sim$ on $b^{T_\Delta}/\sim$ as 
	$$d/\sim([x],[y])=\inf\{d(x',y'):x'\in [x], \ y'\in [y]\}$$
	where $x,y\in b^{T_\Delta}$.
	Equivalently, $d/\sim([x],[y])=2^{-k-1}$ is the minimal $k\in \mathbb{N}$ so that $x$ and $y$ do not have isomorphic labeled $k$-neighborhood of $\mathfrak{r}$.
	As $\AutR$ is compact, we have that $d/\sim([x],[y])=0$ if and only if $[x]=[y]$.
	It is routine to verify that $d/\sim$ is a compatible metric for the quotient topology. Moreover, sets of the form $\{[x]: x \restriction B_{T_{\Delta}}(\mathfrak{r},k)=l\}$ where $l$ is some labeling $l:B_{T_\Delta}(\mathfrak{r},k) \to b$ and $k \in \N$, form a clopen basis.
	This shows that $b^{T_\Delta}/\sim$ is a compact zero-dimensional Polish space.
	
	The root $\mathfrak{r}$ of $T_\Delta$ has exactly $\Delta$-many neighbours.
	Let $g_1,\dots,g_\Delta\in \AutT$ be such that $g_i$ moves $\mathfrak{r}$ to its $i$-th neighbour.
	Then for every other $h\in \AutT$ that moves $\mathfrak{r}$ to one of the neighbors, there is an $h'\in \AutR$ so that $h'\cdot h=g_i$ for some $i\le \Delta$.
	Consequently, the degree of $\fG_{b,\Delta}$ is at most $\Delta$.
	
	It is routine to verify that if $U$ is a basic clopen set, that is, a set defined by some $b$-labeling of a $k$-neighborhood of $\mathfrak{r}$ for some $k\in \mathbb{N}$, then $N_{\fG_{b,\Delta}}(U)$, the neighborhood of $U$ in $\fG_{b,\Delta}$, is a finite union of basic clopen sets, thus a clopen set. Now, by compactness it follows that the neighborhood of every clopen set is clopen as well. 
\end{proof}

\paragraph{The compact continuous forest $\dist_\Delta$}
First, we restrict our attention to labelings that satisfy local constraints.
This way, we define a continuous forest that provides a bridge between local algorithms and the class $\cont$, see \cite[Section 5.A.]{Bernshteyn2021local=cont}. Note that a similar construction has been used by Elek \cite{elek2018qualitative}. Let $k\in \mathbb{N}$ and write $\ell_k$ for the size of the $k$-neighborhood of the root $\mathfrak{r}$ in $T_\Delta$, that is, $\ell_k=|B(\mathfrak{r},k)|=1+\Delta\sum_{i=0}^{k-1}(\Delta-1)^i$.
By a simple greedy algorithm, we see that there exist $k$-distance colorings of $T_\Delta$ with $\ell_k$ many colors for each $k\in \mathbb{N}$.
Write $T^{[k]}_\Delta$ for the space of all $k$-distance colorings, we have $T^{[k]}_\Delta\subseteq \ell_k^{T_\Delta}$. 

\begin{proposition}\label{pr:ContTree2}
	The space $T^{[k]}_\Delta$ is $\AutT$-invariant and closed in $\ell_k^{T_\Delta}$, thus, the space $T^{[k]}_\Delta/\sim$ is a nonempty compact zero-dimensional Polish space. Let $\fG^{[k]}_\Delta$ be the restriction of $\fG_{\ell_k,\Delta}$ to this space. Then the graph $\fG^{[k]}_\Delta$ is a $\Delta$-regular continuous graph of girth at least $k+1$, that is, every non-trivial cycle in $\fG^{[k]}_\Delta$ consists of at least $k+1$ vertices.
\end{proposition}
\begin{proof}
	
	The first assertion is obvious from the definition. Since $T^{[k]}_\Delta/\sim$ is a continuous image of a compact space, it must be compact itself, and it is zero-dimensional, by being a subset of a zero-dimensional Polish space. 
	
	Using \cref{pr:ContTree1} and the compactness of $T^{[k]}_\Delta/\sim$, it is easy to see that $\fG^{[k]}_\Delta$ is a continuous graph. Suppose that there is a nontrivial cycle $[x_1],\dots,[x_\ell]$ in $\fG^{[k]}_\Delta$ for some $\ell \le k+1$. Namely, $[x_1]=[x_{\ell}]$ and $([x_i],[x_{i+1}])\in \fG^{[k]}_\Delta$ for every $i<\ell$. By switching to smaller cycles, we may assume that $[x_1],\dots,[x_{\ell}]$ does not contain repeated vertices apart from the first and last one, and of course $\ell \ge 4$. Let $g_1,\dots,g_{\ell-1} \in \AutT$ witness the edge relations. Since $T_\Delta$ is a tree, if the distance of $\mathfrak{r}$ and $g_{\ell-1}\circ \cdots \circ g_1(\mathfrak{r})$ was less than $\ell-1$, there would have been an $i <\ell-1$ with $[x_i]=[x_{i+2}]$, contradicting our assumption. 		
	
	Thus, the distance of $\mathfrak{r}$ and $g_{\ell-1}\circ \cdots \circ g_1(\mathfrak{r})$ is exactly $\ell-1\le k$, so these vertices have different labels, showing $[x_1]\not=[x_\ell]$.
	A similar argument shows that $\fG^{[k]}_\Delta$ is $\Delta$-regular.
\end{proof}
Observe that $\AutT$ acts on $\prod_{k\in \mathbb{N}} T^{[k]}_\Delta$ by the product shift action, that is, $(g\cdot x)_k(v)=x_k(g^{-1}( v))$ for every $x=(x_k)_{k\in \mathbb{N}}\in \prod_{k\in \mathbb{N}} T^{[k]}_\Delta$ and $g\in \AutT$.
\begin{definition}
	
	Define $X=\left(\prod_{k\in \mathbb{N}} T^{[k]}_\Delta\right)/\sim$, where $\sim$ is the equivalence relation induced by a restriction of the product shift action to $\AutR$. Similarly, we abuse the notation and write $[{-}]$ for the corresponding equivalence classes. Let the graph  $\dist_\Delta$ be defined as follows: $[x], [y]\in X$ form an edge in $\dist_\Delta$ if and only if there is $g\in \AutT$ that moves $\mathfrak{r}$ to one of its neighbors and $g\cdot x=y$. 
\end{definition}

A similar reasoning as in \cref{pr:ContTree1} yields the following.

\begin{claim}
	The sets of the form $\{[(x_k)_{k \in \N}]\in X: \forall k \leq N \ x_k \restriction B(\mathfrak{r},N)=s_k\}$ form a basis in $X$, where $N \in \N$ and $s_k:B(\mathfrak{r},N) \to \ell_k$ for each $k \leq N$. 
\end{claim}

\begin{proposition}
	\label{pr:compactness}
	The space $X$ is a nonempty compact zero-dimensional Polish space and the graph $\dist_\Delta$ is a continuous forest. 
\end{proposition}
\begin{proof}
		
		An argument analogous to those in the proofs of \cref{pr:ContTree1} and \cref{pr:ContTree2} shows that $X$ is a nonempty zero-dimensional Polish space and $\dist_\Delta$ is a $\Delta$-regular continuous graph.
		Showing that $\dist_\Delta$ is acyclic is the same as showing that its girth is at least $k+1$ for every $k\in \mathbb{N}$.
		This latter claim follows by repeating the proof of \cref{pr:ContTree2} separately in each coordinate. \footnote{Another way of seeing the above claim this is as follows: Given $K \in \N$ we define the space $X_K=\left(\prod_{k\in K} T^{[k]}_\Delta\right)/\sim$ and graph $\dist_K$ analogously to $X$ and $\dist$. Now if $K_0<K_1\in \mathbb{N}$, there is a canonical continuous and surjective map $q_{K_0,K_1}:X_{K_1}\to X_{K_0}$ that is a homomorphism from $\dist_{K_1}$ to $\dist_{K_0}$.
			It is routine to check that $X$ and $\dist_\Delta$ are inverse limits of the systems $(X_K)_{K\in \mathbb{N}}$ and $(\dist_K)_{K\in\mathbb{N}}$, yielding all the desired properties.}	
\end{proof}

\paragraph{The universal continuous forest $\lab_\Delta$}
In the next example, we consider the space of $\{0,1\}$-labelings of $T_\Delta$ that breaks all symmetries. In the case of countable groups and their Cayley graphs, this model has been studied extensively in the literature, see \cite{ColorinSeward,ST,GJKS,Bernshteyn2021local=cont}.
The main difference from the previous example is that these continuous graphs are not defined on compact spaces.

\newcommand{\freeT}{\free\left(\{0,1\}^{T_\Delta}\right)}

\begin{definition}
	Let $\freeT$ be the set of all elements of $\{0,1\}^{T_\Delta}$ that break all the symmetries of $T_\Delta$.
	That is, $x\in \freeT$ if and only if $g\cdot x\not=x$ whenever $\id_{T_\Delta}\not= g\in \AutT$.
\end{definition}

It is easy to see that the space $\freeT$ is $\AutT$-invariant, however, it is not obvious that it is nonempty.
On the other hand, if it is nonempty, then it is again not difficult to see that it is not compact.

\begin{definition}
	Let $\lab_\Delta$ be the restriction of $\fG_{\{0,1\},\Delta}$ to the set $\free(\{0,1\}^{T_\Delta})/\sim$.
\end{definition}

\begin{proposition}
	The space $\freeT/\sim$ is a nonempty zero-dimensional Polish space and the graph $\lab_\Delta$ is a continuous forest.
\end{proposition}
\begin{proof}
	By \cite[Theorem 3.11]{kechrisclassical}, it is clearly enough to show that $\freeT/\sim$ is a nonempty $G_\delta$-subset of $\{0,1\}^{T_\Delta}/\sim$.
		It is not hard to verify that $[x]\in \freeT/\sim$ if and only if for every $v\not= w\in T_\Delta$ there is $k\in \mathbb{N}$ so that $x\upharpoonright B(v,k)$ and $x\upharpoonright B(w,k)$ are not isomorphic.
	The latter condition is clearly $G_\delta$.
	To show that $\freeT$ is nonempty, we can simply consider a random element of $\{0,1\}^{T_\Delta}$ with respect to the product coin-flip measure.
	The fact that $\lab_\Delta$ is acyclic follows directly from the definition.

    Finally, it is not hard to see that $\lab_\Delta$ is a continuous graph, using the fact that a basis for the topology is given by the clopen sets of the form $\{x:x \restriction B(\mathfrak{r},N)=s\}/\sim$, for $N \in \N$ and $s: B(\mathfrak{r},N) \to \{0,1\}$.
\end{proof}

\subsubsection{$\borel$}\label{subsubsec:borel}
A Borel graph $\mathcal{G}$ is a graph such that $V(\mathcal{G})$ is a standard Borel space and $E(\mathcal{G}) \subset \binom{X}{2}$ is a Borel set. 

\begin{definition}[Borel forest]
	A \emph{Borel ($\Delta$-regular) forest} $\fT$ is an acyclic $\Delta$-regular Borel graph.
	We define $H(\fT)$ to be the set of all half-edges of $\fT$.
	
\end{definition}

Note that the space $H(\fT)$ is naturally endowed with a standard Borel structure. 

\begin{definition}[$\borel$]
	Let $\Pi=(\Sigma,\fV,\fE)$ be an LCL.
	We say that $\Pi$ is in the class $\borel$ if for every Borel forest $\fT$ there is a \emph{Borel function} $f:H(\fT)\to \Sigma$ that is a $\Pi$-coloring of $\fT$.
\end{definition}

We have the following observation.

\begin{theorem}
	Let $\Pi$ be an LCL.
	If $\Pi\in \cont$, then $\Pi\in \borel$.
\end{theorem}

\begin{proof}
	Let $\mathcal{T}$ be a $\Delta$-regular Borel forest. By \cite{KST}, $\mathcal{T}$ admits a Borel edge-coloring with $2\Delta-1$-many colors. This yields that there is a finite collection $(\varphi_i)_{i \leq 2\Delta-1}$ of Borel involutions with the property that $E(\mathcal{G})=\bigcup_i graph(\varphi_i)$. 
	Now, it is a standard fact, see \cite[Section~13]{kechrisclassical}, that given a countable collection of partial Borel maps on $X$, one can find a compatible zero-dimensional Polish topology that turns them into partial homeomorphisms with clopen domains. Then, in this topology, $\mathcal{T}$ becomes a continuous graph, and a continuous solution to $\Pi$ is clearly Borel in the original topology.  
\end{proof}

\subsubsection{$\baire$}\label{subsubsec:baire}

Now, we are ready to define the class $\baire$ as the relaxation of the class $\borel$.

\begin{definition}[$\baire$]
	Let $\Pi=(\Sigma,\fV,\fE)$ be an LCL.
	We say that $\Pi$ is in the class $\baire$ if for every Borel forest $\fT$ and every compatible Polish topology $\tau$ on $V(\mathcal{T})$, there is a $\tau$-comeager set $C$, and Borel function $f:H(\fT) \cap (C \times E(\mathcal{T}))\to \Sigma$ that is a $\Pi$-coloring.
\end{definition}

A direct consequence of the definitions is the following.

\begin{theorem}
	Let $\Pi$ be an LCL.
	If $\Pi\in \borel$, then $\Pi\in \baire$.
\end{theorem}

\paragraph{Toast decomposition}
The main technical advantage, that the class $\baire$ offers, is the possibility of building a hierarchical decomposition.
Finding a hierarchical decomposition in the context of descriptive combinatorics is tightly connected with the notion of \emph{Borel hyperfiniteness}.
Understanding which Borel graphs are Borel hyperfinite is a major theme in descriptive set theory \cite{doughertyjacksonkechris,gaojackson,conley2020borel}.
It is known that grids, and generally polynomial growth graphs are hyperfinite, while, e.g., acyclic graphs are not in general hyperfinite \cite{jackson2002countable, bernshteyn2025large}.
A strengthening of hyperfiniteness that is of interest to us is called \emph{toast} \cite{gao2015forcing,conleymillerbound}. Note that a similar concept has also appeared under the name ``path decomposition" in \cite{conley2020measurable}. 

\begin{definition}[CToast]
	Let $G$ be a graph and $q\in \mathbb{N}$.
	A \emph{$q$-ctoast} of $G$ is a collection $\fD$ of finite connected subsets of $G$ with the property that 
	\begin{enumerate}
		\item every pair of adjacent  vertices is covered by some element of $\fD$ and 
		\item the vertex boundaries of every $D\not=E\in \fD$ are at least $q$ apart in the graph distance.
	\end{enumerate}

	Let $\fG$ be a Borel graph.
	A $q$-ctoast $\fD$ of $\fG$ is \emph{Borel}, if it is a Borel subset of the standard Borel space of all finite subsets of $V(\mathcal{G})$.
\end{definition}

Note that while the general definition of toast does not require the pieces to be connected, we follow \cite[Definition~4.1]{grebik_rozhon2021toasts_and_tails} who considered the case connected toasts.
To emphasize the difference from the general definition, we call a connected toast \emph{ctoast}.

In the case of forests there is no way of constructing a Borel toast in general, however, it is a result of Hjorth and Kechris \cite{hjorth1996borel} that every Borel graph is hyperfinite on a comeager set for every compatible Polish topology.
A direct consequence of \cite[Lemma~3.1]{marksungerbaire} together with a standard construction of a toast gives the following.

\begin{proposition}
	\label{pr:BaireToast}
	Let $\fG$ be a Borel graph that has degree bounded by $\Delta\in \mathbb{N}$ and $\tau$ be a compatible Polish topology on $V(\mathcal{G})$.
	Then for every $q>0$ there is a Borel $\fG$-invariant (that is, containing full connected components) $\tau$-comeager set $C$ on which $\fG$ admits a Borel $q$-ctoast.
\end{proposition}
For the sake of completeness, we include the proof of this statement. 

\begin{proof}	
	Fix $q\in \mathbb{N}$ and a sufficiently fast growing function $f(n)$, e.g., $f(n)=(4q)^{n^2}$.
	Then a standard Baire category argument (see \cite[Lemma~3.1]{marksungerbaire}) gives a sequence $\{A_n\}_{n\in \mathbb{N}}$ of Borel subsets of $X$ such that $C=\bigcup_{n\in \mathbb{N}}A_n$ is a Borel $\tau$-comeager set that is $\fG$-invariant, and $\distance_\fG(x,y)\ge 2f(n)$ for every distinct $x,y\in A_n$.
	
		Let $R_n(x)=B_\fG(x,f(n)/4)$, where $n\in \N$ and $x\in A_n$.

	\begin{lemma} There exists a Borel collection of subsets $(\mathcal{D}_n)_{n \in \N}=(\{C_n(x)\}_{x \in A_n})_{n \in \N}$ with the following properties:
		\begin{enumerate}
			\item \label{c:lfirst} for every $n$ and $x \in A_n$ we have $B_\fG(x,f(n)/3) \supseteq C_n(x) \supseteq R_n(x)$,
			\item \label{c:lsec} $C_n(x)$ is connected,
			
			\item \label{c:lthird} if $A \in \fD_i$ and $B\in \fD_j$ with $i \leq j$ then the distance between the boundaries of $A$ and $B$ is $>q$.
		\end{enumerate}
		
	\end{lemma}

	\begin{proof}
        For $A\subseteq V(\fG)$, we write
        $$\mathcal{B}_\mathcal{G}(A,m)=\{v\in V(\mathcal{G}):\exists a\in A \ \operatorname{dist}_\fG(a,v)\le m\}$$
        for the set of vertices of $\mathcal{G}$ of graph distance at most $m$ from $A$.
    
		Set $C_1(x)=R_1(x)$.
		Suppose that $\fD_n$ has been defined and set
		\begin{itemize}
			\item  $H^{n+1}(x,n+1):=R_{n+1}(x)$ for every $x\in A_{n+1}$,
			\item  if $1\le i\le n$ and $\left\{H^{n+1}(x,i+1)\right\}_{x\in A_{n+1}}$ has been defined, then we put
			$$H^{n+1}(x,i)=H^{n+1}(x,i+1)\cup \bigcup \left\{B_\mathcal{G}\left(C_i(y),\frac{f(i)}{4}\right):y \in A_i \land B_\mathcal{G}\left(C_i(y),\frac{f(i)}{4}\right) \cap H^{n+1}(x,i+1) \neq \emptyset\right\}$$
			for every $x\in A_{n+1}$,
			\item  set $C_{n+1}(x):=H^{n+1}(x,1)$ for every $x\in A_{n+1}$, this defines $\mathcal{D}_{n+1}$.
		\end{itemize}
	
		We show by induction on $n$ that $\mathcal{D}_n$ satisfies the desired properties, more precisely, it satisfies the first two conditions, and that if $A \neq B \in \bigcup_{i \leq n} \mathcal{D}_i$ then the third condition of the lemma holds as well. This is clearly enough. 
		
		The case $n=1$ is clear, so assume that we have this for $n$. To see \eqref{c:lfirst} holds, observe that by induction, for all $x \in A_{n+1}$ and $i \leq n+1$ we have \begin{equation}  \tag{*} H^{n+1}(x,i) \subseteq B_{\mathcal{G}}\left(H^{n+1}(x,i+1), \frac{2f(i)}{3}+\frac{2f(i)}{4}+1\right),\label{e:contains}\end{equation}
		thus 
		\[C_{n+1}(x) \subseteq B_{\mathcal{G}}\left(x,\frac{f(n+1)}{4}+\sum_{i \leq n} \left(\frac{2f(i)}{3}+\frac{2f(i)}{4}+1\right)\right) \subseteq  B_{\mathcal{G}}\left(x,\frac{f(n+1)}{3}\right),\]
		by the assumption on the growth rate of $f$. 
		
		\eqref{c:lsec} is clear from the definition, so let us check \eqref{c:lthird}.  If $A \neq B \in \bigcup_{i \leq n} \mathcal{D}_i$ then the statement follows from the inductive hypothesis, while if they are from $\mathcal{D}_{n+1}$ then it follows from \eqref{c:lfirst} and $f(n+1)/3>q$. So assume that $C_{n+1}(x)=A \in \mathcal{D}_{n+1}$ and $B \in \mathcal{D}_i$ for some $i<n+1$ and the condition fails. If $B_{\mathcal{G}}(B,q) \cap H^{n+1}(x,j) \neq \emptyset$ for some $j>i$, then $B_{\mathcal{G}}(B,\frac{f(i)}{4})$ would have been added to $H^{n+1}(x,i)$, obtaining that $B_{\mathcal{G}}(B,q) \subseteq C_{n+1}(x)$. Consequently, it must be the case that $B_\mathcal{G}(B,\frac{f(i)}{4}) \cap H^{n+1}(x,i+1)=\emptyset$. 
		
		Moreover, by \eqref{c:lfirst} and the fact that elements of $A_i$ are further than $2
        f(i)$ separated, it follows that $B_\mathcal{G}(B,\frac{f(i)}{4}) \cap H^{n+1}(x,i)=\emptyset$ holds as well. On the other hand, by the containment in \eqref{e:contains} and \eqref{c:lfirst}, it follows that 
		\[C_{n+1}(x) \subseteq B_{\mathcal{G}}\left(H^{n+1}(x,i),\sum_{j<i}\left(\frac{2f(j)}{3}+\frac{2f(j)}{4}+1\right)\right),\] consequently, $B_\mathcal{G}(B,q) \cap C_{n+1}(x) \neq \emptyset$ would imply $B_\mathcal{G}(B,\frac{f(i)}{4}) \cap H^{n+1}(x,i)\neq \emptyset$, using again the assumption on the growth rate of $f$. This contradiction shows the desired statement.
\end{proof}

	It remains to check that $\fD=\bigcup_{n\in\mathbb{N}} \fD_n$ is a $q$-ctoast of $\fG$ restricted to $C$. But this is clear: since every element of $C$ is covered by $A_n$ for some $n\in \mathbb{N}$, it follows from the definition of $\fD$ and $q>0$ that every adjacent pair of vertices is contained in some element of $\mathcal{D}$, and the rest of the properties are satisfied by the lemma.
\end{proof}

\paragraph{$\toast$ algorithm}
The idea to use a toast structure to solve LCL problems appears in \cite{conleymillerbound} and has many applications since then \cite{gao2015forcing,marksunger}.
This approach has been formalized in \cite{grebik_rozhon2021toasts_and_tails}, where the authors introduce $\toast$ algorithms (see also \cite{qian2022descriptive} for a related notion of ASI algorithms).
As we use this notion in \cref{sec:baire}, we discuss here briefly the definition and refer the reader to \cite[Section~4]{grebik_rozhon2021toasts_and_tails} for more details.

An LCL $\Pi$ admits a $\toast$ algorithm if there is $q\in \mathbb{N}$ and a partial extending function that has the property that whenever it is applied inductively to a $q$-ctoast, then it produces a $\Pi$-coloring.
A \emph{partial extending function} gets a finite subset of some underlying set that is partially colored as an input and outputs an extension of this coloring on the whole finite subset.
An advantage of this approach is that once we know that a given Borel graph admits a Borel or Baire ctoast structure, and a given LCL $\Pi$ admits a $\toast$ algorithm, then we may conclude that $\Pi$ is in the class $\borel$ or $\baire$, respectively.

\subsection{The remaining implications on \cref{fig:new_picture}}
\label{subsec:implications}

Let us discuss now the arrows on \cref{fig:new_picture} that are not the main two results of this paper. Note that assuming the two main results and using \cref{thm:basicLOCAL}, we only have to show two things. First, that there is a class from measurable combinatorics corresponding to $\local(O(1))$, second, that $\CONT \subsetneq \borel \subsetneq \baire$.

{\bf (a)}
In order to show that problems in $\local(O(1))$ can be defined purely in measurable combinatorics terminology, we define the continuous forest $\fT^\mathbb{N}_\Delta$.
The vertex set are injective maps from $T_\Delta\to \mathbb{N}$, i.e., a closed and $\AutT$-invariant subset of $\mathbb{N}^{T_\Delta}$, and $\fT^\mathbb{N}_\Delta$ is the restriction of $\fG_{\mathbb{N},\Delta}$ to this set.
We claim that the class $\local(O(1))$ is the same as the class of problems solvable by uniformly continuous (for the canonical metric) function. Indeed, by Naor-Stockmeyer \cite[Theorem~3.3]{naorstockmeyer}, if $\Pi \in \local(O(1))$, then there is an \emph{order-invariant} constant round solution, that is, one that only depends on the order of the IDs. Applying this algorithm to the labeling obtained from $\mathcal{T}^\N_\Delta$, we solve $\Pi$ in a uniformly continuous manner. For the converse, note that a uniformly continuous map with finite range depends only on the labels in a fixed sized neighborhood.

It is interesting to note that from the perspective of $\borel$ is the graph $\fT^{\mathbb{N}}_{\Delta}$ trivial, i.e, there is a Borel set that intersects every connected component of $\fT^{\mathbb{N}}_{\Delta}$ in exactly one point, namely, the collection of elements so that the label of the root is the minimal natural number appearing on the tree. In the Borel context this suffices to solve any coloring problem (see, e.g., \cite[Theorem 5.23]{pikhurko2021descriptive_comb_survey}).
On the other hand we do not know if $\cont$ is equal to $\local(O(\log^* n))$ on this graph.

{\bf (b)}
It is proved in \cite{brandt2021classes} that the $\Delta$-coloring problem with paths distinguishes the classes $\cont$ and $\borel$.
Recall that to solve the problem we have to produce a $\Delta$-coloring but we are allowed not to color some vertices.
The constraint is that the induced subgraph on the uncolored vertices consists of doubly infinite paths.
The intuitive reason why this problem is not in $\cont$ is that it is not possible to produce doubly infinite lines in a continuous way.
This is proved formally in \cite{brandt2021classes}.
On the other hand, after inductively discarding $\Delta-2$ Borel maximal independent sets, the induced subgraph on the remaining vertices consists of finite, one-ended, or doubly infinite paths.
This is clearly enough to construct the desired solution in a Borel way.

Finally, \cite{conley2016brooks} and \cite{DetMarks} show that $\Delta$-coloring distinguishes between $\borel$ and $\baire$.

	\section{Continuous colorings}
	\label{sec:Continuous}

	Now we turn to the investigation of the existence of continuous solutions of LCL problems on regular forests. Our main contribution is that, unlike in the main theorem of \cite{Bernshteyn2021LLL}, we do not need to assume that our graph is given by a group action, that is, the edges do not come with a proper edge $\Delta$-coloring. As seen below, this requires a significant change of Bernshteyn's argument.

	Recall the definition of $\dist_\Delta$ and $\lab_\Delta$ from \cref{subsub:continuous}.
	
	\begin{theorem}
		\label{t:maincontinuous} Let $\Pi$ be an LCL problem on $\Delta$-regular forests. The following are equivalent:
		\begin{enumerate}
			\item \label{c:localcont} $\Pi \in \local(O(\log^* n))$
			\item \label{c:incont} $\Pi \in \cont$
			\item \label{c:ondist} $\Pi$ admits a continuous solution on $\dist_\Delta$ 
			\item \label{c:on01} $\Pi$ admits a continuous solution on $\fT^{\{0,1\}}_\Delta$ 
		\end{enumerate} 
		
	\end{theorem}
	
	It is clear that \eqref{c:incont} implies \eqref{c:ondist} and \eqref{c:on01}.
	Before we turn our attention to the most challenging part \eqref{c:on01} implies \eqref{c:incont} (\cref{t:from01}), we demonstrate that \eqref{c:localcont} implies \eqref{c:incont} (\cref{t:localcontinuous}), and \eqref{c:ondist} implies \eqref{c:localcont} (\cref{t:distcont}).
	These implications are fairly standard and almost follow from \cite{Bernshteyn2021LLL,elek2018qualitative,Bernshteyn2021local=cont}.
	The reason why we cannot use these results directly is that we consider LCL problems defined on half-edges.
	
	\begin{theorem}[\cite{Bernshteyn2021local=cont}]\label{t:localcontinuous}
		Let $\Pi$ be an LCL problem such that $\Pi\in \local(O(\log^* n))$.
		Then $\Pi\in \cont$.
	\end{theorem}
	\begin{proof}

			By the results of \cite{chang_pettie2019time_hierarchy_trees_rand_speedup,chang_kopelowitz_pettie2019exp_separation}, we have that $\Pi$ can be solved by an algorithm that first solves the $k$-distance coloring problem, for some $k\in \mathbb{N}$, and then a local algorithm $\fA$ with complexity $O(1)$ is applied.
			Of course, we may assume that $k>2$.
			
			Let $\fT$ be a continuous forest on $X$.
			By \cite[Lemma~2.3]{Bernshteyn2021local=cont}, there is a continuous $k$-distance coloring $c$ of $\fT$.
			Define a partition
			$$\bigsqcup X_{s,i,j}=H(\fT),$$
			where $X_{s,i,j}$ is the set of half-edges $(v,e)\in H(\fT)$ such that $c\upharpoonright B_\fT(v,k)=s$, $c(v)=i$ and $c(w)=j$, where $e=\{v,w\}$.
			It follows from the definitions of $H(\fT)$ and $c$ that $\bigsqcup X_{s,i,j}$ is a finite clopen partition.
			
			For any $(v,e)\in X_{s,i,j}$, we define $\varphi(v,e)$ to be the value of $\fA(s)$ on the half-edge that belongs to the edge $\{v,w\}$ adjacent to the root such that $c(v)=i$ and $c(w)=j$.
			Note that this is well-defined as we assume that $k>2$.
			Now, it is easy to see that $\varphi$ is a $\Pi$-coloring of $\fT$.
	\end{proof}
	
	\begin{theorem}[\cite{elek2018qualitative,Bernshteyn2021local=cont}]
		\label{t:distcont}
		Let $\Pi$ be an LCL problem such that $\dist_\Delta$ admits a continuous $\Pi$-coloring.
		Then $\Pi\in \local(O(\log^* n))$. 
	\end{theorem}
	\begin{proof}
		Suppose that $\varphi$ is a continuous $\Pi$-coloring of $\dist_\Delta$. By \cref{pr:compactness} and the definition of the topology on $H(\fT)$, there is some $N\in \mathbb{N}$ such that the value of $\varphi$ on the half-edges around $[(x_k)_{k\in \mathbb{N}}]$ depends only on the isomorphism type of the restriction of $(x_k)_{k \leq N}$ to $B_{\dist_\Delta}(\mathfrak{r},N)$, where we recall that $\mathfrak{r}$ is the root of $T_\Delta$.
		
		We define a local algorithm $\mathcal{A}$ of complexity $N$ from $\varphi$ as follows: $\mathcal{A}$ is defined on the isomorphism types of a neigborhood $B(\mathfrak{r},N)$ with a sequence of vertex labelings $(c_k)_{k\leq N}$, where $c_k$ is a solution to the $k$-distance coloring problem on $B(\mathfrak{r},N)$ for every $k\leq N$. Note that this information determines a basic open neighborhood in $\dist_\Delta$, on which $\varphi$ is constant. Let $\mathcal{A}$ evaluate to this constant value.
		
		Given a finite regular forest $T$, we describe how to solve $\Pi$ in $O(\log^* n)$ rounds with deterministic local algorithm.
		The main technical difficulty is that $\fA$ assumes as an input an isomorphic copy of the rooted graph $B(\mathfrak{r},N)$, therefore we need to trick the algorithm whenever it is applied to a vertex that is close to a virtual half-edge.
		
		First, we produce a sequence $(c_k)_{k\leq N}$ of vertex labelings that solve the $k$-distance coloring problem on $T$ for every $k\leq N$.
		As $N\in \mathbb{N}$ is fixed, this requires $O(\log^* n)$ rounds.
		
		Second, we construct a solution to the $2N$-distance coloring problem, let $C_1,\dots, C_{\ell_{2N}}$ be the obtained color classes. Now, by induction on $i \leq \ell_{2N}$
		we create an increasing sequence of forests $(T_i)_{i \leq \ell_{2N}}$ together with maps $c^{i}_k:{V}(T_i) \to \ell_k$ so that 
		\begin{itemize}
			\item $c^0_k=c_k$, $T_0=T$, for every $k \leq N$, $c^0_k \subseteq \dots \subseteq c^{i}_k$, and each $c^{i}_k$ is a solution to the $k$-distance coloring problem on $T_{j}$,
			\item if $v \in C_i$ then $B_{T_i}(v,N)$ is isomorphic to $B_{T_{\Delta}}(\mathfrak{r},N)$. 
		\end{itemize}
		
		Formally, this can be encoded by each vertex $v \in V(T)$ adjacent to a virtual half-edge computing the labelings $c^i_k$ and computing the isomorphism type of $B_{T_i}(v,N)$, and finally communicating this information to every vertex in their $N$-neighborhood. 
		
		Observe that both conditions can be checked in an $N$-neighborhood of a vertex. Moreover, given all these data for some $i$, we can find the appropriate objects for $i+1$ by a greedy LOCAL algorithm of complexity $N$, as the number of colors that we can use to solve the $k$-distance coloring problem is $\ell_k=|B(\mathfrak{r},k)|$, and vertices in $C_{i+1}$ can define these objects independently, since their distance is bigger than $2N$.
		This requires $O(\log^* n)+O(1)$ rounds.

		Finally, we apply $\fA$ to the isomorphism type of $(c^{\ell_{2N}}_k\upharpoonright B_{T_{\ell_{2N}}}(v,N))_{k\leq N}$ at every vertex $v\in V(T)$.
		It is routine to verify that this produces a $\Pi$-coloring of $T$ in $O(\log^* n)$ rounds.  
	\end{proof}
	
	We remark that a similar trick to the second half of the above proof has been utilized in \cite[Section~5.3]{Bernshteyn2021local=cont}.

	\subsection{The forest $\tzeroone$}
	
	In order to finish the proof of \cref{t:maincontinuous}, we have to show the following.
	
	\begin{theorem}
		\label{t:from01}
		Let $\Pi$ be an LCL problem on $\Delta$-regular forests that admits a continuous solution on $\fT^{\{0,1\}}_\Delta$ and $\mathcal{T}$ be a $\Delta$-regular continuous forest.
		Then $\Pi$ admits a continuous solution on $\mathcal{T}$. 
	\end{theorem}
	
	Our strategy follows closely Bernshteyn's argument \cite{Bernshteyn2021local=cont}.
	Namely, given a continuous forest $\fT$ on $X$, we construct a continuous $\{0,1\}$-labeling of $X$ in such a way that the $\{0,1\}$-labeled tree rooted at any given vertex $x\in X$, i.e., the vertex labeled connected component of $x$ in $\fT$, is ``close'' to $\freeT$.
	Here ``close'' means that the continuous function on $\freeT/\sim$ that solves $\Pi$ can be extended to these elements. 
	Before we give a formal proof, we need to introduce some technical notation and recall the continuous Lov\'{a}sz Local Lemma of Bernshteyn \cite{Bernshteyn2021local=cont}.

	\subsubsection{Hyperaperiodic elements}
	An element $[\alpha]\in \freeT/\sim$ is called \emph{hyperaperiodic} if the closure of the connectivity component of $[\alpha]$ in $\tzeroone$ in the space $\{0,1\}^{T_\Delta}/\sim$ is a subset of $\freeT/\sim$.
	Equivalently, the closure is a compact subset of $\freeT/\sim$, or every element that can be approximated by a sequence from $\AutT\cdot \alpha$ is in the space $\freeT$.
	We also abuse the notation and say that $\alpha\in \freeT$ is hyperaperiodic if $[\alpha]$ is such.
	
	This notion is well studied in the context of countable groups \cite{ColorinSeward,ST,GJKS,Bernshteyn2021local=cont}.
	It is a highly non-trivial result for example that for every countable group $\Gamma$ there exists some $x \in \{0,1\}^\Gamma$, such that the closure of the orbit of $x$, with respect to the left-shift action of $\Gamma$ on $\{0,1\}^\Gamma$ is contained in the free part of this action. Now, one can define the spaces analogous to $\{0,1\}^{T_\Delta}$ and $\freeT/\sim$ for any decorated graph in the place of $T_\Delta$. The aforementioned results imply that for any finitely generated group $\Gamma$, if $C$ is Cayley graph of $\Gamma$ (with edges decorated by the generators), then $\{0,1\}^\Gamma/\sim$ admits hyperaperiodic element.  
	
	As a byproduct of our main result, we show that these elements indeed exist even in our setting, where the automorphism group is much larger.
	It is an exciting open problem to understand for which other classes of (structured) graphs hyperaperiodic elements exist.
	
	\begin{theorem}\label{t:hyperaperiodic}
		There are hyperaperiodic elements in $\freeT/\sim$.    
	\end{theorem}
	The statement follows from the last part of the proof of \cref{t:from01}. 
	The following is the main definition that we use to approximate hyperaperiodic elements.
	
	\begin{definition} Let $g\in \AutT$, $\alpha\in {\{0,1\}}^{T_\Delta}$, $v\in T_\Delta$, $n\in \mathbb{N}$ and $K\in \mathbb{N}$.
		We say that \emph{$(\alpha,K)$ $n$-breaks $g$ at $v$} if either $\distance(g\cdot v,v)> n$ or $g\cdot v=v$, or $g\cdot \alpha'\not= \alpha'$ for every $\alpha'\in\{0,1\}^{T_\Delta}$ such that
		$$\alpha\upharpoonright B(v,K)=\alpha'\upharpoonright B(v,K).$$
		Define
		$$X_{K,n}=\left\{\alpha\in \{0,1\}^{T_\Delta}:\forall v \in T_\Delta,\forall g\in \AutT \  (\alpha,K) \ n\operatorname{-breaks} \ g \ \operatorname{at} \  v\right\}$$
		for every $K,n \in \mathbb{N}$.
	\end{definition}
	
	\begin{proposition}\label{pr:CompactInv}
		Let $n>1$, $K\in \mathbb{N}$ and suppose that $X_{K,n}\subseteq \{0,1\}^{T_\Delta}$ is non-empty.
		The sets $X_{K,n}$ and $X_{K,n}/\sim$ are compact, $X_{K,n}$ is $\AutT$-invariant and we have $\distance(g\cdot v,v)>n$ for every $v\in T_\Delta$ and $\id_{T_\Delta}\not =g\in \AutT$ that satisfies $g\cdot \alpha=\alpha$ for some $\alpha\in X_{K,n}$.
		In particular, $\fG_{\{0,1\},\Delta}\upharpoonright (X_{K,n}/\sim)$ has girth at least $n$.	
	\end{proposition}
	\begin{proof}
		It is easy to see from the definition that $X_{K,n}$ is $\AutT$-invariant.
		Suppose that $\alpha_\ell\to \alpha$ in $\{0,1\}^{T_\Delta}$, where $\alpha_\ell\in X_{K,n}$ for every $\ell\in \mathbb{N}$.
		Let $v\in T_\Delta$ and $g\in \AutT$.
		We need to show that $(\alpha,K)$ $n$-breaks $g$ at $v$.
		If not, then $0<\distance(g\cdot v,v)\le n$ and we find $\alpha'\in \{0,1\}^{T_\Delta}$ such that $g\cdot \alpha'=\alpha'$ and $\alpha\upharpoonright B(v,K)=\alpha'\upharpoonright B(v,K)$.
		Taking $\ell\in \mathbb{N}$ large so that $\alpha\upharpoonright B(v,K)=\alpha_\ell\upharpoonright B(v,K)$, we see that $(\alpha_\ell,K)$ does not $n$-break $g$ at $v$ as well, a contradiction.
		This shows that $X_{K,n}$, and consequently $X_{K,n}/\sim$, is closed and hence compact.
		
		Let $g\in \AutT$ be such that $g\cdot \alpha=\alpha$ for some $\alpha\in X_{K,n}$.
		It is easy to see that if there is $w\in T_\Delta$ such that $g\cdot w=w$, then either $g=\id_{T_\Delta}$, or we can find $v\in T_\Delta$ such that $\distance(g\cdot v,v)= 2$.
		Altogether, if $g\not=\id_{T_\Delta}$ and there is a $v$ with  $\distance(g\cdot v,v) \leq n$, then there is a $v\in T_\Delta$ so that $0<\distance(g\cdot v,v)\le n$.
		From the definition of $X_{K,n}$, we have that $(\alpha,K)$ $n$-breaks $g$ at $v$.
		But that contradicts the fact that $g\cdot \alpha=\alpha$, hence, $g=\id_{T_\Delta}$ or $\distance(g\cdot v,v)>n$ for every $v\in T_\Delta$.
		
		The additional part follows by the same argument as in \cref{pr:ContTree2}, namely, if there were a non-trivial cycle in $\fG_{\{0,1\},\Delta}\upharpoonright (X_{K,n}/\sim)$ of length at most $n$, then the composition of the elements of $\AutT$ that witness the edge relations would produce a $g\in \AutT$ that fixes some $\alpha\in X_{K,n}$ and satisfies $0<\distance(g\cdot \mathfrak{r},\mathfrak{r})\le n$, a contradiction.
	\end{proof}

	\subsubsection{Continuous Lov\'{a}sz Local Lemma}
	In order to construct/approximate hyperaperiodic elements, Bernshteyn \cite{Bernshteyn2021local=cont} developed a continuous version of the Lov\'{a}sz Local Lemma (LLL).
	Namely, he found a sufficient LLL-style condition for the existence of a continuous solution to a continuous constraint satisfaction problem (CSP) defined on a zero-dimensional Polish space, see the definitions below.
	Rather surprisingly, a similar but incomparable sufficient condition for the existence of $O(\log^*n)$ deterministic algorithms was studied in the distributed computing literature \cite{brandt_grunau_rozhon2020tightLLL}.
	We note that both conditions are tight, see the discussion in \cite[Section~1.A.2]{Bernshteyn2021local=cont}.
	
	Let $X$ be a set and $k \in \mathbb{N}$. For a finite set $D \subseteq X$, an $(X,k)$-\emph{bad event with domain $D$} is an element of $k^D$. A \emph{constraint $\mathscr{B}$} with domain $D$ is a collection of $(X,k)$-bad events with domain $D$; we will use the notation $\dom(\mathscr{B})=D$, and omit $X$ and $k$ if it is clear from the context. A \emph{CSP} $\mathfrak{B}$ (on $X$ with range $k$)
	is a set of $(X,k)$-constraints.
	Define the \emph{maximal probability of $\mathfrak{B}$} as
	\[p(\mathfrak{B})=\sup_{\mathscr{B} \in \mathfrak{B}}\frac{|\mathscr{B}|}{k^{|dom(\mathscr{B})|}},\]
	the \emph{vertex degree of $\mathfrak{B}$} as 
	\[vdeg(\mathfrak{B})=\sup_{x \in X}|\{\mathscr{B}\in \mathfrak{B}:x \in dom(\mathscr{B})\}|,\]
	and the \emph{order of $\mathfrak{B}$} by
	\[ord(\mathfrak{B})=\sup_{\mathscr{B} \in \mathfrak{B}} |\dom(\mathscr{B})|.\]
	If $X$ is a zero-dimensional Polish space, a CSP $\mathfrak{B}$ on $X$ is \emph{continuous} if for every
	set $S$ of functions $\{1,\dots, n\}\to k$ and $U_2,\dots,U_n \subseteq X$ clopen the set
	\[\{x_1: \text{there are $x_2 \in U_2,\dots, x_n \in U_n$ such that } x_1,\dots, x_n \text{ are distinct and } S(x_1,\dots, x_n) \in \mathfrak{B}\}\]
	is clopen, where $S(x_1,\dots,x_n)$ stands for the collection of functions determined by the elements of $S$, i.e., $f \in S(x_1,\dots,x_n)$ if for some $\varphi \in S$ we have $f(x_i)=j \iff \varphi(i)=j$. 
	
	A $k$-coloring $f:X\to k$ \emph{satisfies (or solves) the CSP $\mathfrak{B}$} if $f$ has no restriction that is a bad event. The following theorem has been proven in \cite{Bernshteyn2021local=cont}:
	
	\begin{theorem}[\cite{Bernshteyn2021local=cont}]
		\label{t:antonLLL}
		Let $X$ be a zero-dimensional Polish space, $k \in \mathbb{N}$, and $\mathfrak{B}$ be a continuous CSP on $X$ with range $k$ such that 
		\[p(\mathfrak{B})\cdot \vdeg(\mathfrak{B})^{\ord(\mathfrak{B})}<1.\]
		Then $\mathfrak{B}$ admits a continuous solution.
	\end{theorem}

	\subsubsection{The continuous forest $\fT$}
	The aim of this section is to show that every continuous $\{0,1\}$-labeling of vertices of a continuous forest $\fT$ naturally induces a continuous homomorphism from $\fT$ to $\fG_{\{0,1\},\Delta}$.
	The main technical difficulty is that there is no unified way how vertices of $\fT$ ``view'' their connected component and its automorphism group.
	To overcome this difficulty we first break local symmetry by a proper edge coloring. Proper edge coloring in our formulation is a map $c:H(\mathcal{T}) \to k$ so that the multiset of colors appearing around any vertex $v$ is injective and if $g,h$ are half-edges forming a real edge $e$, then $c(g)=c(h)$. Note that, since it contains no virtual half-edges, for a continuous forest it is equivalent to the usual formulation of the problem. 
	
	\begin{proposition}\label{pr:edgebreaksymm}
		Let $\fT$ be a continuous forest.
		There is a continuous map $\Theta_\fT:H(\fT)\to (2\Delta-1)$ that solves the proper edge coloring problem.
	\end{proposition}
	\begin{proof}
		This is a consequence of Bernshteyn's transfer \cref{t:localcontinuous} combined with \cite[Theorem 4.1 and Section 5]{linial92LOCAL}, applied to the line graph of $\fT$.
		Recall that a line graph of $\fT$ is a graph with a vertex set equal to the edge set of $\fT$, where two edges of $\fT$ form an edge in the line graph if they share a vertex in $\fT$. It is not hard to check that this is a continuous graph as well, and a continuous $2\Delta-1$-coloring yields a solution of the proper edge coloring problem for $\mathcal{T}$.  
	\end{proof}
	
	Let $\fT$ be a continuous forest, $\Theta_\fT$ as above and $x\in V(\mathcal{T})$.
	Moreover, fix the canonical linear ordering on the set $(2\Delta-1)$ and some ordering of the vertices of $T_\Delta$.
	Write $I_x$ for the isomorphism between the connectivity component of $x$ in $\fT$ and $T_\Delta$ that is defined inductively along $k=\distance_\fT(x,{-})$ as follows:
	\begin{enumerate}
		\item $I_x$ sends $x$ to the root $\mathfrak{r}$ of $T_\Delta$,
		\item Suppose that $I_x$ is defined on $B_\fT(x,k)$ and $\distance_\fT(x,y)=k$ for some $y\in B_\fT(x,k)$.
		Consider the linear ordering on $S_y=\{z\in V(\mathcal{T}):\distance_\fT(x,z)=k+1, \ \distance_\fT(y,z)=1\}$ induced by $\Theta_\fT$, i.e., $z_0<z_1$ if $\Theta_\fT((y,e_0))<\Theta_\fT((y,e_1))$, where $e_i=\{y,z_i\}$ for $i<2$.
		Extend $I_x$ to each $S_y$ in such a way that for all $z_0,z_1 \in S_y$ we have $z_0 <z_1 \iff I(z_0)<_{T_\Delta} I(z_1)$, where $<_{T_\Delta}$ is the linear ordering fixed on $V(T_\Delta)$. 
	\end{enumerate}

	Similarly, we define the action of $\AutT$ on the connectivity component of $x$ in $\fT$ from the perspective of $x$: 
	\begin{definition}[$\odot_x$]
		Define $g\odot_x y=I^{-1}_x(g\cdot I_x(y))$, where $x,y\in V(\mathcal{T})$, $g\in \AutT$ and $y$ is in the connectivity component of $x$.
	\end{definition}

	\begin{proposition}\label{pr:InducedMap}
		Let $\fT$ be a continuous forest and $\psi:V(\mathcal{T})\to \{0,1\}$ be a continuous function.
		Then the map $\widetilde{h}_\psi:V(\mathcal{T})\to \{0,1\}^{T_\Delta}$ defined as
		\[\widetilde{h}_\psi(x)(v)= \psi(I^{-1}_x(v))\]
		is continuous.
		Moreover, the induced map $h_\psi:V(\mathcal{T})\to \{0,1\}^{T_\Delta}/\sim$ is a continuous homomorphism from $\fT$ to $\fG_{\{0,1\},\Delta}$.
	\end{proposition}
	\begin{proof}
		It follows directly from the definitions that the map is well-defined.
		As $\Theta_\fT$ and $\psi$ are continuous, it follows that $\widetilde{h}_\psi$ is continuous as well.
		To see this, note that the values of $\Theta_\fT$ and $\psi$ on $B_{\fT}(x,\ell)$ fully determine $\widetilde{h}_\psi$ on $B_{\fT}(x,\ell)$.
		The continuity in the moreover part easily follows. Finally, let us check that $h_\psi$ is indeed a homomorphism. Let $(x,y) \in E(\mathcal{T})$. We have to construct an automorphism $g \in \AutT$ moving $\mathfrak{r}$ to one of its neighbors so that $g \cdot \tilde{h}_\psi(x)=\tilde{h}_\psi(y)$. Let $g(v)=I_y(I^{-1}_x(v))$, then $g \in \AutT$ as $I_x$ and $I_y$ are isomorphisms. It is also clear by the fact that $x$ and $y$ are adjacent that $g$ maps $\mathfrak{r}$ to one of its neighbors. Now,	    
		\[\tilde{h}_\psi(x)(g^{-1} \cdot v)=\psi(I^{-1}_x(I_x(I^{-1}_y(v)))= \psi(I^{-1}_y(v))=\tilde{h}_\psi(y)(v),\] which is what we wanted. 
	\end{proof}
	Let us remark that $h_\psi$ remains the same for any continuous assignment $x\mapsto I_x$, where $I_x$ is an enumeration of the connected component of $x$.

	\subsubsection{Key Lemma}
	
	The following technical lemma is the key tool, which allows us to iteratively define a homomorphism. Following the terminology from \cite{Bernshteyn2021local=cont}, recall that if $\mathcal{T}$ is a continuous forest and $s,M \in \mathbb{N}$, a set $C \subseteq V(\mathcal{T})$ is called \emph{s-syndetic} if for every $x \in V(\mathcal{T})$ there exists a $y \in C$ with $\distance_{\mathcal{T}}(x,y) \leq s$; a set $U\subseteq V(\mathcal{T})$ is called \emph{$M$-separated}, if $\distance_{\mathcal{T}}(x,y)>M$ for every distinct $x,y \in U$.
	
	\begin{lemma}
		\label{l:key}
		Let $\mathcal{T}$ be a continuous $\Delta$-regular forest. For every $n\in \mathbb{N}$ and $s\ge 2$ there are $K,M\in \mathbb{N}$ such that for every decomposition of $V(\mathcal{T})=C_0\sqcup C\sqcup U$ into clopen sets such that $C$ is $s$-syndetic and $U$ is $M$-separated, and every $\varphi_0:C_0\to \{0,1\}$ that is continuous, there is a continuous extension map $\varphi:C_0 \cup C\to \{0,1\}$ so that for every continuous map $\psi:V(\mathcal{T})\to \{0,1\}$ that extends $\varphi$ we have that $h_\psi(V(\mathcal{T}))\subseteq X_{K,n}/\sim$. 
	\end{lemma}
	\begin{proof}
		Let $L\in \mathbb{N}$ be divisible by $48 \cdot 8s$ such that
		$$\left(\frac{\sqrt{2s+1}}{(\Delta-1)^{L/16}}\right)^{L/{24s}} \cdot (\Delta^{17L+2n})^2<1,$$
		such an $L$ exists by $\Delta \ge 3$. 
		We show that that taking $M=L$ and $K=6L+n$ satisfies the conclusion of the lemma.
		
		We start by taking a maximal $L$-separated clopen subset $\fZ$ in $\fT$: 
	 	such a set exists, since the $L$-power graph of $\fT$ is a continuous graph as well, and then we can apply \cite[Lemma~2.2]{Bernshteyn2021local=cont}.

		\begin{claim}\label{cl1}
			Let $x\in V(\mathcal{T})$ and $g \in \AutT$ be such that $0<\distance_\fT(g\odot_x x, x) \leq n$.
			Then there is a $z_{x,g}\in \fZ$ that satisfies the following:
			\begin{itemize}
				\item $\distance_\fT(x,z_{x,g})<4L$,
				\item $\distance_\fT(g\odot_x z_{x,g},z_{x,g})>L$.
			\end{itemize}
			Moreover, we may assume that the assignment $(x,g)\mapsto z_{x,g}$, that is defined on a clopen subset of $V(\mathcal{T})\times \AutT$, is continuous.
		\end{claim}
		Note that the proof below relies heavily on the fact that we are working with forests. The important observation is that a non-trivial automorphism of $T_\Delta$, for $\Delta \ge 3$, must move points arbitrarily far from themselves. 
		\begin{proof}
			As $g\not=\id_{T_\Delta}$, there exists an $x'\in V(\mathcal{T})$ with $\distance_\fT(x,x')=2L$ and $\distance_\fT(x',g \odot_x x')=4L+\distance_\fT(g\odot_x x, x)$: indeed, any $x'\in V(\mathcal{T})$ with $\distance_\fT(x,x')=2L$ and the property that the shortest path between $g \odot_x x'$ and $x'$ contains $x$ and $g \odot_x x$ is suitable.
			To find such a point $x'$, we use that $\Delta>2$.
			Take $x_0,x_1,x_2$ in distance $2L$ from $x$ such that the shortest paths from $x_i$ to $x$ are pairwise edge disjoint. Then there is some $i<3$ such that the shortest path between $x_i$ and $x$ and between $g \odot_x x_i$ and $g \odot_x x$ are both edge disjoint from the one between $x$ and $g \odot_x x$. Letting $x'=x_i$ works, since the shortest path between $x_i$ and $g \odot_x x_i$ goes through $x$ and $g \odot_x x$. Now, by the maximality of $\mathcal{Z}$ we can pick a point $z_{x,g}\in \fZ$ with $\distance_\fT(x',z_{x,g}) \leq L$, which is clearly suitable.
			
			To see the additional part, first note that $(x,g)\mapsto z_{x,g}$ is defined on $V(\mathcal{G})\times \bigcup_{g\in I}g\cdot \AutR$, where $I\subseteq \AutT$ is some finite set, since $\mathcal{B}_\fT(x,n)$ is finite.
			The argument above depends on a finite neighborhood of $x$ in $\fT$, namely, the neighborhood of radius $6L+n$.
			Similarly, we only need to know the 'action' $\odot_{{-}}$, and consequently the proper edge coloring $\Theta_\fT$ from \cref{pr:edgebreaksymm}, in a finite radius around $x$ in $\fT$.
			As all the involved objects are either continuous or clopen, it follows that the assignment  can be chosen to be continuous as required.
		\end{proof}
		From now on, we fix a continuous assignment $(x,g)\mapsto z_{x,g}$ with the properties specified in \cref{cl1}. 
		
		\begin{claim}
			\label{cl:numberofchoices}
			Let $z\in \fZ$ and define the set
			$$C_z=\{r\in V(\mathcal{T}):\exists g\in \AutT \ \exists x\in V(\mathcal{T}) \ z=z_{x,g} \ \text{ and } \ g\odot_x z=r\}.$$
			Then we have $|C_z|\le 2\Delta^{8L+n}$.
		\end{claim}
		\begin{proof}
			Let $r \in C_z$ and $x,g$ be from the definition of $C_z$.
			Then we have
			\begin{equation}\label{eq:numberfchoices}
			\begin{split}
			\distance_\fT(z,r)= \distance_\fT(z_{x,g},g\odot_x z_{x,g})\leq  & \ \distance_\fT(z_{x,g},x)+\distance_\fT(x,g \odot_x x)+\distance_\fT(g \odot_x x,g \odot_x z_{x,g}) \\
			\leq & \ 4L+n+4L
			\end{split}
			\end{equation}
			by \cref{cl1}.
			We are done by using the upper bound $|B_\fT(z,8L+n)|\le 2\Delta^{8L+n}$.
		\end{proof}

		In order to utilize \cref{t:antonLLL}, we define a CSP that is suitable in our situation and yields the desired extension $\varphi$.
		For $z \in V(\mathcal{T})$, let $B_z=B_{\mathcal{T}}(z,L/4)$, note that $B_z \cap B_{z'}=\emptyset$ for every $z \neq z' \in \fZ$. We apply the \cref{t:antonLLL} with $X=\mathcal{Z}$ and $k=2^{|B_z|}$ for some, or equivalently any, $z\in \fZ$.
		We use the proper edge coloring $\Theta_\fT$ from \cref{pr:edgebreaksymm} to identify $2^{B_z}$ with $k$ for each $z \in \fZ$ in such a way that every continuous function $f:\fZ \to k$ naturally encodes a continuous $\{0,1\}$-labelling of $\bigcup_{z\in \fZ} B_z$.

		Define $\varphi'$ on the set $C_0 \cup (C \setminus \bigcup_{z \in \fZ} B_z)$ to be $\varphi_0$ on $C_0$ and $0$ otherwise.
		Given a continuous map $f:\fZ\to k$, we define $\varphi_f:C_0\cup C\to \{0,1\}$ to be the union of $\varphi'$ and the $\{0,1\}$-labelling encoded by $f$ restricted to $C\cap  \bigcup_{z \in \fZ} B_z$.
		It is routine to verify that $\varphi'$ is a continuous map defined on a clopen subset of $V(\mathcal{T})$.
		Similarly, $\varphi_f$ is continuous whenever $f$ is continuous.
		Next, we use \cref{t:antonLLL} to find a suitable $f$ so that $\varphi:=\varphi_f$ has all the desired properties.
		
		We define our CSP $\mathfrak{B}$.
		Given $z\in \fZ$ and $r\in C_z$, we let the constraint $\mathscr{B}_{z,r}\in \mathfrak{B}$ with domain
		$$\dom(\mathscr{B}_{z,r})=\left\{z'\in \fZ:B_{z'}\cap \left(B_z\cup B_r\right)\neq \emptyset\right\}$$
		as follows:
		\begin{itemize}
			\item $b:\dom(\mathscr{B}_{z,r})\to k$ is in $\mathscr{B}_{z,r}$ if and only if there is a $g\in \AutT$ such that
			\begin{itemize}
				\item $g\odot_z z=r$,
				\item there is a labeling $\phi$ of $B_z\cup B_r$, which agrees with the labeling defined on $(B_z\cup B_r)\setminus U$ by $\varphi'$ and the $\{0,1\}$-labeling encoded by $b$ restricted to $C\cap (B_z\cup B_r)$,   such that for every $y\in B_z$ we have $\phi(y)=\phi(g\odot_z y)$.
			\end{itemize} 
		\end{itemize}
		It is not hard to see that $\mathfrak{B}$ is a continuous CSP.
		Let us argue first that it is sufficient to solve $\mathfrak{B}$ to show the lemma.
		
		\begin{claim}\label{cl:enough}
			Assume that $f:\fZ \to k$ is a continuous solution to $\mathfrak{B}$.
			Then $\varphi_f$ satisfies the conclusion of \cref{l:key}. 
		\end{claim}
		\begin{proof}
			Let $\psi$ be a continuous extension of $\varphi_f$ to $V(\mathcal{T})$ and $x\in V(\mathcal{T})$.
			We show that $\widetilde{h}_\psi(x) \in X_{K,n}$.
			If not, then there are  $v\in T_\Delta$ and $g\in \AutT$ such that $(\widetilde{h}_\psi(x),K)$ does not $n$-break $g$ at $v$.
			By the definition, this means that $0<\distance(g\cdot v,v)\le n$ and there is an $\alpha\in \{0,1\}^{T_\Delta}$ such that $g\cdot \alpha=\alpha$ and  $\widetilde{h}_\psi(x) \restriction B(v,K)=\alpha \restriction B(v,K)$.
			
			We may assume that $v=\mathfrak{r}$, otherwise we replace $x$ with $I^{-1}_x(v)$.
			Let $z_{x,g}$ be from \cref{cl1} and set $r=g\odot_x z_{x,g}$.
			Recall that $K=6L+n$. Then we have 
			\begin{equation}\label{eq1}
			B_r\cup \bigcup_{z'\in \dom(\mathscr{B}_{z,r})}B_{z'}\subseteq B_\fT(x,6L+n)
			\end{equation}
			by \cref{cl1}.
			
			Set $b=f\upharpoonright \dom(\mathscr{B}_{z,r})$.
			We show that $b\in \mathscr{B}_{z,r}$ contradicting the assumption on $f$.
			For $u\in T_\Delta$ define $g_0 \in \AutT$ by 
			$$g_0\cdot u=I_z\circ I^{-1}_x (g\cdot (I_x \circ I^{-1}_z)(u))=I_z(g\odot_x (I^{-1}_z(u))).$$
			We have,
			\begin{itemize}
				\item by the definition of $g_0$, that $g_0\odot_z z=I^{-1}_z(g_0 \cdot I_z(z))=g\odot_x z=r$,
				\item by $g\cdot \alpha=\alpha$ and the definition of $g_0$, that
				\[		        \begin{split}
				\alpha(I_x(y))= & \ \alpha(g\cdot I_x(y))=\alpha(g\cdot I_x(I_z^{-1}(I_z(y)))) \\
				= & \ \alpha(I_x\circ I_z^{-1}(g_0\cdot I_z(y))) \\
				= & \ \alpha(I_x(g_0\odot_z y))
				\end{split}
				\]
				for $y\in B_z$,
				\item by \eqref{eq1}, that $\phi=\alpha\circ I_x$ agrees with the labelling defined on $(B_z\cup B_r)\setminus U$ by $\varphi'$ and the $\{0,1\}$-labeling encoded by $b$ restricted to $C\cap (B_z\cup B_r)$. 
			\end{itemize}
			Consequently, $b\in \mathscr{B}_{z,r}$ and we are done.
		\end{proof}
		
		\begin{claim}
			\label{cl:ordvdeg}
			We have $ord(\mathfrak{B})\le 2$ and $vdeg(\mathfrak{B}) \leq \Delta^{17L+2n}$.
		\end{claim}
		\begin{proof}
			Let $z\in \fZ$ and $r\in C_z$.
			Clearly, $z\in \dom(\mathscr{B}_{z,r})$.
			Suppose that $z_0,z_1\in \dom(\mathscr{B}_{z,r})\setminus \{z\}$.
			By the definition of $\fZ$, we have $B_{z_i}\cap B_z=\emptyset$, and therefore $B_{z_i}\cap B_r\not=\emptyset$ for $i\in\{0,1\}$.
			This implies that $\distance_\fT(z_0,z_1)<L$.
			Consequently, $z_0=z_1$ as $\dom(\mathscr{B}_{z,r})\subseteq \fZ$.
			
			Let $z\in \fZ$.
			There are at most $2\Delta^{8L+n}$ constraints of the form $\mathscr{B}_{z,r}$ by \cref{cl:numberofchoices}.
			On the other hand, if $z\in \dom(\mathscr{B}_{z',r'})$ for some $z\not=z'\in \fZ$ and $r'\in C_{z'}$, then
			$$\distance_\fT(z,z')\leq \distance_{\mathcal{T}}(z,r)+\distance_{\mathcal{T}}(r,z') \leq 8L+n+L/2$$
			by \eqref{eq:numberfchoices}.
			This shows that there are at most $2\Delta^{8L+n+L/2}$ choices for $z'\in \fZ$, each having $2\Delta^{8L+n}$ choices for $r'\in C_{z'}$.
			Altogether we have $vdeg(\mathfrak{B}) \leq \Delta^{17L+2n}$ as desired.
		\end{proof}	
		
		It remains to estimate $p(\mathfrak{B})$.
		Let $z\in \fZ$ and $r\in C_z$.
		Set $\mathbb{P}(\mathscr{B}_{z,r})=\frac{|\mathscr{B}_{z,r}|}{k^{|\dom(\mathscr{B}_{z,r})|}}$, i.e., the probability that $b\in k^{\dom(\mathscr{B}_{z,r})}$ chosen uniformly at random falls into $\mathscr{B}_{z,r}$.
		With this notation we have
		$$p(\mathfrak{B})=\sup_{z,r} \mathbb{P}(\mathscr{B}_{z,r}).$$
		
		We start with a simple observation about the $s$-syndetic set $C$.
		For $x\in V(\mathcal{T})$ and $i\in \mathbb{N}$, we define $S_\fT(x,i)=\{y\in V(\mathcal{T}):\distance_\fT(x,y)=i\}$.
		We show that there are many indices $i\le L/4$ such that $S_\fT(z,i)\cap C$ is large, where $z\in \fZ$.
		This allows to bound the probability of $\mathscr{B}_{z,r}$ as follows: if $b\in \mathscr{B}_{z,r}$, then it encodes isomorphic labelings of $B_z$ and $B_r$, in particular, this means that the number of $0$s and $1$s on $S_\fT(z,i)$ and $S_\fT(r,i)$ have to agree for every $i\le L/4$.
		We use Stirling's formula to bound the probability on each $S_\fT({-},i)$, and then the independence of these events for different indices, see \cref{l:cnrz}.
		
		\begin{claim}
			\label{cl:inC}  
			Let $x\in V(\mathcal{T})$ and $N\in \mathbb{N}$ be a multiple of $48s$.
			Then we have
			\[\left|\left\{i \leq  2N:|C\cap S_{\mathcal{T}}(x,i)| > \frac{(\Delta-1)^{N}}{2s+1}\right\}\right| \geq \frac{N}{3s}+2.\] 
		\end{claim}
		\begin{proof}

			Let $A_{\mathcal{T}}(x,j,\ell)=\bigcup_{0 \leq i  < \ell} S_{\mathcal{T}}(x,j+i)$, where $j>0$.
			Note that each connected component of the graph $\mathcal{T}$ restricted to $A_{\mathcal{T}}(x,j,2s+1)$ contains a ball of radius $s$ around some point. There are $\Delta(\Delta-1)^{j-1}$ such connected components, hence, $|A_{\mathcal{T}}(x,j,2s+1) \cap C| \geq \Delta(\Delta-1)^{j-1}$.
			In particular, there is some $i \in [j,j+2s+1)$ with $|S_{\mathcal{T}}(x,i) \cap C| > \Delta(\Delta-1)^{j-1}/(2s+1)$.
			Observe that the set $A_{\mathcal{T}}(x,N,N)$ contains at least $\left\lfloor\frac{N}{2s+1}\right\rfloor$ disjoint sets of the form $A_{\mathcal{T}}(x,j,2s)$ with $j \geq N$, which yields the claim as $\left\lfloor\frac{N}{2s+1}\right\rfloor \geq \frac{N}{3s}+2$, which follows from the assumptions that $s\ge 2$ and $N$ is divisible by $48s$.
		\end{proof}
		
		\begin{lemma}
			\label{l:cnrz}
			Let $z\in \fZ$ and $r\in C_z$.
			Then $\mathbb{P}(\mathscr{B}_{z,r})\le \left(\frac{\sqrt{2s+1}}{(\Delta-1)^{L/16}}\right)^{L/{24s}}$.
			Consequently, $p(\mathfrak{B})\le \left(\frac{\sqrt{2s+1}}{(\Delta-1)^{L/16}}\right)^{L/{24s}}$.
		\end{lemma}
		\begin{proof}
			
			Apply \cref{cl:inC} to $N= L/8$ and $x=z$.
			Since $U$ is $M=L$-separated, there are at most two $i,j \leq L/4$ such that $(S_{\mathcal{T}}(r,i) \cup S_{\mathcal{T}}(z,j)) \cap U \neq \emptyset$.
			Hence we can fix a set $S$ of indices $i \leq L/4$ of size at least $L/{24s}$ such that $S_{\mathcal{T}}(r,i)$ and $S_{\mathcal{T}}(z,i)$ are disjoint from $U$ for every $i\in S$, and $S_{\mathcal{T}}(z,i)$ contains at least $\frac{(\Delta-1)^{L/8}}{2s+1}$ elements of $C$ for every $i\in S$.
			
			By \cref{cl:ordvdeg}, $|\dom(\mathscr{B}_{z,r})|\le 2$.
			Let $b:\dom(\mathscr{B}_{z,r})\to k$.
			Suppose that we know the value of $b$ on $\dom(\mathscr{B}_{z,r})\setminus \{z\}$ (this set can be empty).
			This, together with $\varphi'$, fully determines a $\{0,1\}$-labeling of $B_r\setminus U$.
			Intuitively, if $b\in \mathscr{B}_{z,r}$, then the value $b(z)$ needs to satisfy a lot of constraints.
			This is formalized in the next claim \eqref{eq:A}, that is clearly enough to finish the proof.
			
			\begin{equation}
			\tag{A}\label{eq:A}
			\parbox{\dimexpr\linewidth-4em}{%
				\strut
				\emph{Conditioned on any $\{0,1\}$-labeling $\alpha_0$ of $B_r$, the probability that the union of $\varphi'$ with a $\{0,1\}$-labeling of vertices in $B_z\cap C$ encoded by a random map $z\mapsto \gamma\in k$ can be extended to a $\{0,1\}$-labeling of $B_z$ isomorphic to $\alpha_0$ is smaller than $\left(\frac{ \sqrt{2s+1}}{(\Delta-1)^{L/16}}\right)^{L/{24s}}$.}
				\strut
			}
			\end{equation} 
			
			In order to show that, we let $R(z,i)$, for $i\in S$, be the random variable that counts the difference between the number of $0$s and $1$s in the labeling of $S_\fT(z,i)$ that is the restriction of the union of $\varphi'$ with a $\{0,1\}$-labeling of vertices in $B_z\cap C$ encoded by a random map $z\mapsto \gamma\in k$.
			Similarly, we define a deterministic function $R(r,i)$. Observe that $B_r$ and $B_z$ are disjoint, since $r \in C_z$, which gives  that the distance of $r$ and $z$ is $>L$, by \cref{cl1}. Hence, $R(z,i)$ does not depend on the fixed $\{0,1\}$-labeling $\alpha_0$. 
			
			\begin{claim}
				Let $i\in S$.
				Then $\mathbb{P}(R(z,i)=R(r,i))\le \frac{ \sqrt{2s+1}}{(\Delta-1)^{L/16}}$.
			\end{claim}
			\begin{proof}
			
					First, note that as $R(r,i)$ is fixed and $R(z,i)$ depends on $\varphi'$, it can happen that the condition $R(z,i)=R(r,i)$ cannot be satisfied at all.
					In this case the claim trivially holds.
					
					On the other hand, if $R(z,i)=R(r,i)$ can be satisfied, then the difference between $0$s and $1$s that we need to see in the restriction of the random $\{0,1\}$-labeling to the set $C \cap S_{\mathcal{T}}(z,i)$ is equal to a unique number $-t\le R\le t$, where $t=|C \cap S_{\mathcal{T}}(z,i)|$.
					The probability of this event is the same as the probability that the difference between the number of $0$s and $1$s in $t$ bits chosen uniformly at random is exactly $R$, which can be upper bounded by
					\begin{equation}\label{eq5}
					\frac{1}{2^{t}}\binom{t}{\frac{t-R}{2}}\le \frac{1}{2^{t}}\binom{t}{t/2}\leq \frac{1}{\sqrt{\pi t/2}} \leq \frac{1}{\sqrt{t}},
					\end{equation}
					for every $t/2\ge 1$ assuming that $t$ is even, using Stirling's or Wallis' formulas.
					
							By the choice of $i\in S$, we have $t \geq \frac{(\Delta-1)^{L/8}}{2s+1}$.
				Consequently,
				\[\mathbb{P}(R(z,i)=R(r,i))\le \frac{1}{\sqrt t} \leq ((\Delta-1)^{L/8}/(2s+1))^{-1/2}\leq \frac{\sqrt{2s+1}}{(\Delta-1)^{L/16}},\]
				as desired.
			\end{proof}
			
			Note that $R(z,i)$ and $R(z,i')$ are independent for every $i\not=i'$.
			Write $\fE$ for the event from \eqref{eq:A}.
			Then we have
			$$\mathbb{P}(\fE)\le \prod_{i\in S}\mathbb{P}(R(z,i)=R(r,i))\le \left(\frac{\sqrt{2s+1}}{(\Delta-1)^{L/16}}\right)^{L/{24s}}$$
			and the proof is finished.
		\end{proof}

		Combining \cref{l:cnrz} with \cref{cl:ordvdeg} we get 
		\[p(\mathfrak{B}) \cdot vdeg(\mathfrak{B})^{ord(\mathfrak{B})} \leq \left(\frac{\sqrt{2s+1}}{(\Delta-1)^{L/16}}\right)^{L/{24s}} \cdot (\Delta^{17L+2n})^2<1,\]
		by the choice of $L$.
		Therefore, \cref{t:antonLLL} applies and we obtain a continuous coloring $f:\mathcal{Z} \to k$, which avoids all $\mathscr{B} \in \mathfrak{B}$.
		By \cref{cl:enough} this produces the desired map $\varphi=\varphi_f$.
	\end{proof}

	\subsubsection{Proof of \cref{t:from01}}
	Using \cref{l:key} we will inductively define sequences of naturals $(K_i,M_i)_{i \in \mathbb{N}}$, clopen sets $C_n, U_n \subseteq V(\mathcal{T})$ and labelings $\varphi_n:C_n \to \{0,1\}$ such that
	
	\begin{itemize}
		\item $C_0\subseteq C_1 \subseteq C_2 \dots$ and $U_0 \supseteq U_1 \supseteq U_2 \dots$,
		\item $C_n \cup U_n=V(\mathcal{T})$,
		\item $\varphi_{n} \subseteq \varphi_{n+1}$,
		\item $U_n$ is a $(2M_n+1)$-syndetic and $M_n$-separated set,
		\item for every extension of $\varphi_n$ to a continuous map $\psi:V(\mathcal{T}) \to \{0,1\}$, $h_\psi(V(\mathcal{T})) \subset X_{K_{n},n}/\sim$.
	\end{itemize}

	Let $C_0=\emptyset$, $U_{0}=V(\mathcal{T})$, $M_{0},K_{0}=0$ and $\varphi_0=\emptyset$.
	Given  $C_{n},U_{n},M_{n},K_{n}$, apply \cref{l:key} to $s=4M_{n}+3$ and $n+1$ to obtain constants $K_{n+1},M_{n+1}\in \mathbb{N}$.
	It follows from the first paragraph in the proof of \cref{l:key} that we may assume that $K_{n+1},M_{n+1}\geq \max\{K_{n},8M_{n}+6\}$.
	Using \cite[Lemma~2.2]{Bernshteyn2021local=cont}, let $U_{n+1}$ be a maximal $M_{n+1}$-separated clopen subset of $U_n$, i.e., $U_{n+1}$ is $M_{n+1}$-separated and adding any point $x\in U_n\setminus U_{n+1}$ violates this property.
	
	Observe that $U_{n+1}$ is $(2M_{n+1}+1)$-syndetic:
	indeed, for every $x\in V(\mathcal{T})$ there is $z\in U_n$ such that $\distance(x,z)\le 2M_n+1$ by the inductive assumption and, by the maximality of $U_{n+1}$, there is $y\in U_{n+1}$ such that $\distance(z,y)\le M_{n+1}$.
	Altogether, we have that $\distance(x,y)\le 2M_n+1+M_{n+1}\le 2M_{n+1}+1$ by the choice of $M_{n+1}$.
	
	Moreover, we claim that $U_n \setminus U_{n+1}$ is $(4M_{n}+3)$-syndetic.
	In order to see this, take an arbitrary $x\in V(\mathcal{T})$ and $z,z'\in V(\mathcal{T})$ such that $\distance(x,z),\distance(x,z')=2M_n+2$ and $\distance(z,z')=4M_n+4$.
	This is possible as $\mathcal{T}$ is $\Delta$-regular and acyclic.
	As $U_n$ is $(2M_n+1)$-syndetic, there are $y\not=y'\in U_n$ such that $\distance(x,y),\distance(x,y')\le 4M_n+3$.
	Finally, we cannot have that both $y,y'\in U_{n+1}$ because $\distance(y,y')\le 8M_n+6\le M_{n+1}$.
	
	Let $C_{n+1}=C_n \cup (U_{n} \setminus U_{n+1})$, and apply \cref{l:key} to the decomposition $V(\mathcal{T}) =C_n \cup (U_{n} \setminus U_{n+1})\cup U_{n+1}$ and $\varphi_{n}$ to obtain an extension $\varphi_{n+1}$ to $C_{n+1}$.
	
	Let us point out that $\bigcup_n \varphi_n$ might not be continuous. Nevertheless, the properties of the above sequence turn out to be sufficient for our purposes. Let $X_n=X_{K_n,n}$ and for $N\in \mathbb{N}$ define 
	$$Y_{\le N}=\bigcap_{ n\le N} X_n/\sim,$$
	and 
	$$Y=\bigcap_{N\in \mathbb{N}} Y_{\le N}.$$
	It follows from the construction that $Y_{\le N}$ is nonempty for every $N\in \mathbb{N}$.
	This is because $h_\psi(V(\mathcal{T}))\subseteq X_{K_n,n}/\sim$ for any continuous extension of $\varphi_n$ to some $\psi$ for every $n\le N$.
	In particular, one can take $\psi$ to be any continuous extension of $\varphi_N$.
	As the sets $X_{K_n,n}/\sim$ are compact and $\AutT$-invariant by \cref{pr:CompactInv}, we have that $Y_{\le N}$ and consequently $Y$ is such.
	In particular, $Y\subseteq \freeT/\sim$ by \cref{pr:CompactInv} and, consequently, every element of $Y$ is hyperaperiodic.
	This proves \cref{t:hyperaperiodic}.

	Now we are ready to finish the proof of \cref{t:from01}.
	Let $f$ be a continuous $\Pi$-coloring of $\tzeroone$. Since $Y \subseteq \operatorname{Free}\left(\{0,1\}^{T_\Delta}/\sim\right)$ and $Y$ is compact, there is some $L\in \mathbb{N}$, so that for each half-edge $h$, the value of $f(h)$ is completely determined by $x \restriction B_{T_\Delta}(\mathfrak{r},L)$, where $h$ is incident to $[x] \in Y$.
	Define $\mathbb{B}_{L+1}$ to be the collection of all basic open sets that are determined by $\{0,1\}$-labeled neighborhoods $B_{T_\Delta}(\mathfrak{r},L+1)$ that appear in $Y$.
	In other words, $\mathbb{B}_{L+1}$ is an open cover of $Y$.
	By compactness and the fact that $Y=\bigcap_{N\in \mathbb{N}} Y_{\le N}$, there is $N\in \mathbb{N}$ such that $\mathbb{B}_{L+1}$ is an open cover of $Y_{\le N}$.
	Therefore, there is a continuous extension $f'$ of $f$ to
	$$H(\tzeroone\restriction Y_{\leq N})=H(\tzeroone) \cap \left(Y_{\le N} \times \binom{Y_{\le N}}{2}\right),$$
	using the constant value determined by $x \restriction B_{T_\Delta}(\mathfrak{r},L)$, that is a $\Pi$-coloring in the sense that it satisfies the vertex and edge constraints.
	
	Now, extend $\varphi_N$ to any continuous map $\psi:V(\mathcal{T}) \to \{0,1\}$.
	By \cref{pr:InducedMap} and the construction, we have that $h_\psi$ is a homomorphism and $h_\psi(V(\mathcal{T}))\subseteq X_{n}/\sim$ for every $n\le N$.
	In particular, we have that $h_\psi(V(\mathcal{T}))\subseteq Y_{\leq N}$ and as it is a homomorphism, it naturally induces a continuous map $\overrightarrow{h}_\psi$ from $H(\fT)$ to $H(\mathcal{G}_{\{0,1\},\Delta}\restriction Y_{\leq N})$.
	Define  $c=f' \circ \tilde{h}_\psi$. It is easy to see that $c$ is a continuous $\Pi$-coloring of $\mathcal{T}$.

	\section{$\baire = \local(O(\log n))$}
	\label{sec:baire}
	
	In the last two sections, we show that on $\Delta$-regular forests the classes $\baire$ and $\local(O(\log(n)))$ are the same.
	At first glance, this result looks rather counter-intuitive.
	This is because in finite $\Delta$-regular forests every vertex can see a leaf at distance $O(\log(n))$, while there are no leaves at all in an infinite $\Delta$-regular tree.
	However, there is an intuitive reason why these classes are the same: in both setups there is a technique to decompose an input graph into a hierarchy of subsets, see \cref{subsubsec:baire} and \cref{subsec:LOCAL}. Furthermore, the existence of a solution that is defined inductively with respect to these decompositions can be characterized by the same combinatorial condition of Bernshteyn.

	\paragraph{Combinatorial Condition -- $\ell$-full Set}
	In both structural decompositions we need to extend a partial coloring along paths that have their endpoints colored from the inductive step.
	The precise formulation of the combinatorial condition that captures this demand was extracted by Bernshteyn.
	He proved that it characterizes the class $\baire$ for Cayley graphs of virtually free groups.
	Note that this class contains, e.g., $\Delta$-regular forests with a proper edge $\Delta$-coloring.

	\begin{definition}
		A \emph{spiky $\Delta$-path} is a finite $\Delta$-regular tree $P$, so that $(V(P),E(P))$ is a path. 
	\end{definition}
	Observe that, up to isomorphism there is only one such path for each given length. 
	
	\begin{definition}[Combinatorial condition -- an $\ell$-full set]
		\label{def:ellfull}
		Let $\Pi = (\Sigma, \fV, \fE)$ be an LCL problem and $\ell\geq 2$.
		A set $\emptyset \neq \fV' \subseteq \fV$ is \emph{$\ell$-full} whenever the following is satisfied.
		Take a spiky $\Delta$-path $P$ with at least $\ell$ vertices.
		Take any $c_1, c_2 \in \fV'$ and label arbitrarily the half-edges around the endpoints with $c_1$ and $c_2$, respectively.
		Then there is a way to label $H(P)$ with configurations from $\fV'$ such that all the $\ell-1$ edges on the path have valid edge configuration on them. 
	\end{definition}

		\begin{remark}
			Let us remark that this condition is equivalent to the existence of $\emptyset\neq \Sigma'\subseteq \Sigma$ with the property that every $a\in \Sigma'$ is an element of a multiset from $\mathcal{V}$ that only uses elements from $\Sigma'$ and any spiky path of length at least $\ell$ with any choice of $c_1,c_2\in \Sigma'$ at the starting and end half-edges admits a valid labeling of $H(P)$ only using elements from $\Sigma'$. The equivalence can be seen by taking $\Sigma'=$ all labels used in $\mathcal{V}'$, and taking $\mathcal{V}'$ to be the subset of $\mathcal{V}$ that only uses elements from $\Sigma'$.  
		\end{remark}

	We will say that a spiky path $P$ \emph{starts} (resp. \emph{ends}) with a half edge $h$, if $h$ is incident to the first (last) vertex of the path and $h$ is not virtual. Note that by the definition of an LCL problem, whether a labeling with $c_1$ at the starting and $c_2$ in the ending point extends to a labeling of the whole path, depends only on the values on the start and end half-edge, and not on the values of the virtual half-edges of the starting and ending vertices. 
	
	Now we are ready to formulate the result that combines Bernshteyn's result (equivalence between (1) and (3)) with the main results of this section.
	
	\begin{theorem}\label{thm:MainBaireLog}
		Let $\Pi$ be an LCL problem on regular forests. 
		Then the following are equivalent:
		\begin{enumerate}
			\item \label{b:baire} $\Pi\in \baire$,
			\item \label{b:toast} $\Pi$ admits a $\toast$ algorithm,
			\item \label{b:ellfull} $\Pi$ admits an $\ell$-full set,
			\item \label{b:local} $\Pi\in \local(O(\log(n)))$.
		\end{enumerate}
	\end{theorem}
	
	We start by showing \eqref{b:ellfull} implies \eqref{b:toast} (\cref{pr:toastable}) and \eqref{b:ellfull} implies \eqref{b:local} (\cref{thm:ellfull_to_logn}).
	Then we demonstrate that \eqref{b:baire} implies \eqref{b:ellfull} ((\cref{thm:baireellfull})) and, the most challenging, \eqref{b:local} implies \eqref{b:ellfull} (\cref{thm:logn_to_ellfull}).
	The remaining part \eqref{b:toast} implies \eqref{b:baire} is a consequence of \cref{pr:BaireToast} and the definition of a $\toast$ algorithm.
	
	\begin{remark*}
		Bernshteyn's result described in the introduction about the class $\baire$ was again stated for graphs isomorphic to Cayley graphs of some countable groups, e.g., free groups or free products of several copies of $\mathbb{Z}_2$ with the standard generating set etc.
		Unlike in the $O(\log^* n)$ regime, we note that in this case these ``Cayley graph structure'' itself satisfies the $\ell$-full condition.
		Consequently, there is no difference between investigating problems on $\Delta$-regular forests and, e.g., $\Delta$-regular forests with a proper edge $\Delta$-coloring.
	\end{remark*}
	
	Next we discuss the proof of \cref{thm:MainBaireLog}.
	We include the proof of Bernshteyn's result \cite{Bernshteyn_work_in_progress} for completeness.

	\subsection{Sufficiency}
	
	We start by showing that the combinatorial condition is sufficient for $\baire$ and $\local(O(\log(n)))$.
	Namely, it follows from the next results together with \cref{pr:BaireToast} that (3.) implies all the other conditions in \cref{thm:MainBaireLog}.
	As discussed above the main idea is to color inductively along the decompositions.
	
	\begin{proposition}\label{pr:toastable}
		Let $\Pi=(\Sigma, \fV, \fE)$ be an LCL problem that admits an $\ell$-full set $\fV'\subseteq \fV$ for some $\ell>0$.
		Then $\Pi$ admits a $\toast$ algorithm that produces a $\Pi$-coloring for every $(2\ell+2)$-ctoast $\fD$.
	\end{proposition}
	\begin{proof}
		Our aim is to build a partial extending function.
		Set $q:=2\ell+2$.
		Let $E$ be a piece in a $q$-ctoast $\fD$ and suppose that $D_1,\dots,D_k\in \fD$ are subsets of $E$ such that their boundaries are $q$-separated.
		Suppose, moreover, that we have defined inductively a coloring of half-edges of vertices in $D=\bigcup D_i$ using only vertex configurations from $\fV'$ such that every edge configuration $\fE$ is satisfied for every edge in $D$.
		
		We handle each connected component of $E\setminus D$ separately.
		Let $A$ be one of them.
		Let $u\in A$ be a boundary vertex of $E$.
		Such a vertex exists since every vertex in $E$ has degree $>2$, which follows from the assumption that the pieces of the toast are connected.
		The distance of $u$ and any $D_i$ is at least $2\ell+2$ for every $i\in [k]$.
		We orient all the edges from $A$ towards $u$.
		Moreover if $v_i\in A$ is a boundary vertex of some $D_i$ we assign to $v_i$ a path $V_i$ of length $\ell$ towards $u$.
		Note that $V_i$ and $V_j$ have distance at least $1$, in particular, are disjoint for $i\not=j\in [k]$ .
		Now, until we encounter some path $V_i$, color in any manner half-edges of vertices in $A$ inductively starting at $u$ in such a way that edge configurations $\fE$ are satisfied on every edge and only vertex configurations from $\fV'$ are used, this is possible, as $\mathcal{V}' \neq \emptyset$. 
		Use the definition of $\ell$-full set to find a coloring of any such $V_i$ and continue in a similar manner until the whole $A$ is colored.
	\end{proof}

	\subsubsection{Rake and Compress}
	
	\begin{proposition}[$\ell$-full $\Rightarrow$ $\local(O(\log n))$]\label{thm:ellfull_to_logn}
		Let $\Pi = (\Sigma, \fV, \fE)$ be an LCL problem with an $\ell$-full set $\fV' \subseteq \fV$. Then $\Pi$ can be solved in $O(\log n)$ rounds in  $\local$.
	\end{proposition}
	
	In order to prove this statement we introduce a variant of the so-called rake-and-compress decomposition algorithm.
    
	The rake-and-compress process was first defined in~\cite{MillerR89}. It was later generalized and applied in the study of the $\LOCAL$ model~\cite{chang_pettie2019time_hierarchy_trees_rand_speedup,ChangHLPU20,chang2020n1k_speedups,balliu2021_rooted_trees}. We consider the following version of rake-and-compress.

	\begin{definition}[Rake-and-compress process]\label{def-rake-and-compress} Given two positive integers $\gamma$ and $\ell$, the  rake-and-compress process on a finite forest $T$ is defined as follows. For $i = 1, 2, 3, \ldots$, until all vertices are removed, perform $\gamma$ number of \rake\ and then perform one \compress. 
		We write $\VR{i}$ to denote the set of vertices removed during the  \rake\ operations in the $i$-th iteration and we write  $\VC{i}$ to denote the set of vertices removed during the  \compress\ operation in the $i$-th iteration.
	\end{definition}

	It is clear that $V(T) = \VR{1} \cup \VC{1} \cup \VR{2} \cup \VC{2} \cup   \cdots$  in \cref{def-rake-and-compress}, and this is called a rake-and-compress decomposition. We define the \emph{depth} of the decomposition as the smallest number $d$ such that $V(T) = \VR{1} \cup \VC{1} \cup \VR{2} \cup \VC{2} \cup   \cdots \cup \VR{d}$. That is, $d$ is the smallest iteration number such that  by the time we finish all the \rake\ operations in the $d$-th iteration, all vertices have been removed. The following lemma gives an upper bound $L$ on the depth $d$ of the decomposition.

	\begin{lemma}[\cite{chang_pettie2019time_hierarchy_trees_rand_speedup}]\label{lem:size-rake-compress-1} For every $\ell>0$ there is a constant $C_\ell$ so that in the rake-and-compress decomposition of an $n$-vertex forest $T$ with degrees $\leq \Delta$ using parameters $\ell$ and $\gamma = 1$, we have  $V(T) = \VR{1} \cup \VC{1} \cup \VR{2} \cup \VC{2} \cup   \cdots \cup \VR{L}$ for some $L=C_\ell\log n$.
	\end{lemma}
	
	\paragraph{Local Complexity}
	Once we have an upper bound $L \geq d$ on the depth $d$ of the decomposition, it is clear that the rake-and-compress decomposition of \cref{def-rake-and-compress} can be computed in $O((\ell + \gamma)L)$ rounds. For example, if $\ell = O(1)$ and $\gamma = 1$, then  \cref{lem:size-rake-compress-1} implies that $L = O(\log n)$, so the decomposition algorithm finishes in $O((\ell + \gamma)L) = O(\log n)$ rounds.

	\paragraph{Post-processing} Now we focus on the case of $\gamma = 1$ and $\ell = O(1)$.
	As $\gamma = 1$, each connected component in the subgraph induced by $\VR{i}$ is either a single vertex or an edge. Each connected component in the subgraph induced by $\VC{i}$ is a path of at least $\ell$ vertices. 
	
	For some applications, it is desirable that each $\VR{i}$ is an independent set and each $\VC{i}$ is a collection of paths of $O(1)$ vertices. By doing a post-processing step to modify the given rake-and-compress decomposition $V(T) = \VR{1} \cup \VC{1} \cup \VR{2} \cup \VC{2} \cup \cdots$, it is possible to attain these properties. Specifically, we have the following lemma.
	
	\begin{lemma}[Rake-and-compress decomposition with post-processing~\cite{chang_pettie2019time_hierarchy_trees_rand_speedup}]\label{lem-rake-and-compress-modified}
		Given any constant integer $\ell' \geq 1$, there is an $O(\log n)$-round $\LOCAL$ algorithm that decomposes the vertices of an $n$-vertex forest $T$ into $2L - 1$ layers \[V(T) = \VR{1} \cup \VC{1} \cup \VR{2} \cup \VC{2} \cup \VR{3} \cup \VC{3} \cup \cdots \cup \VR{L},\] 
		with $L= O(\log n)$ satisfying the following requirements.
		\begin{itemize}
			\item For each $1 \leq i \leq L$, $\VR{i}$ is an independent set, and each  $v \in \VR{i}$ has at most one neighbor in $\VR{L} \cup \VC{L-1} \cup \VR{L-1} \cup \cdots \cup \VC{i}$.
			\item For each $1 \leq i \leq L-1$, $\VC{i}$ is a collection of disjoint paths $P_1, P_2, \ldots, P_k$, where each $P_j=(v_1, v_2, \ldots, v_s)$ satisfies the following requirements.
			\begin{itemize}
				\item The number $s$ of vertices in $P_j$ satisfies $\ell' \leq s \leq 2 \ell'$.
				\item There are two vertices $u$ and $w$ in  $\VR{L} \cup \VC{L-1} \cup \VR{L-1} \cup \cdots \cup \VR{i+1}$ such that $u$ is adjacent to $v_1$, $w$ is adjacent to $v_s$, and $u$ and $w$  are the only vertices in  $\VR{L} \cup \VC{L-1} \cup \VR{L-1} \cup \cdots \cup \VR{i+1}$ that are adjacent to $P_j$.
			\end{itemize}
		\end{itemize}
	\end{lemma}

	See \cref{fig:rake_and_compress} for  an example of a decomposition of \cref{lem-rake-and-compress-modified} with $\ell' = 4$. 
    
	\begin{proof}
		We first compute a rake-and-compress decomposition $V(T) = \VR{1} \cup \VC{1} \cup \VR{2} \cup \VC{2} \cup \VR{3} \cup \VC{3} \cup \cdots$ of \cref{def-rake-and-compress} with parameters $\ell = \ell'+2$ and $\gamma = 1$. By \cref{lem:size-rake-compress-1}, the decomposition can be computed in   $O(\log n)$  rounds in the $\LOCAL$ model.
		
		To satisfy all the requirements, we consider the following post-processing step, which takes $O(\log^\ast n)$ rounds to perform. For each $i$, the vertex set $\VR{i}$ is a collection of non-adjacent vertices and edges. For each edge $\{u,v\}$ with $u \in \VR{i}$ and $v \in \VR{i}$, we promote one of $u$ and $v$ to $\VR{i+1}$. 
		For each $i$,  the vertex set  $\VC{i}$ is a collection of disjoint paths of at least $\ell'+2$ vertices. For each of these paths $P$, we compute an independent set   $I\subset V(P)$ satisfying the following conditions.
		\begin{itemize}
			\item  $I$ is an independent set that contains both endpoints of $P$.
			\item Each connected component of the subgraph induced by $V(P) \setminus I$ has at least $\ell'$ vertices and at most $2 \ell'$ vertices.
		\end{itemize}
		Then we promote all the vertices in $I$ to $\VR{i+1}$.
        Note that $I$ remains independent as a subset of $\VR{i+1}$ as no neighbor of $v\in I$ can be an element of $\VR{i+1}$.
        As $\ell' = O(1)$, such an independent set $I$ can be computed in $O(\log^\ast n)$ rounds~\cite{chang_pettie2019time_hierarchy_trees_rand_speedup,cole86}.
		It is straightforward to verify that the new decomposition $V(T) = \VR{1} \cup \VC{1} \cup \VR{2} \cup \VC{2} \cup \VR{3} \cup \VC{3} \cup \cdots$ after the modification satisfies all the requirements.
        By \cref{lem:size-rake-compress-1}, there is $L = O(\log n)$ such that $V(T) = \VR{1} \cup \VC{1} \cup \VR{2} \cup \VC{2} \cup \VR{3} \cup \VC{3} \cup \cdots \cup \VR{L}$.
	\end{proof}

	\begin{figure}[ht]
		\centering
		\includegraphics[width= \textwidth]{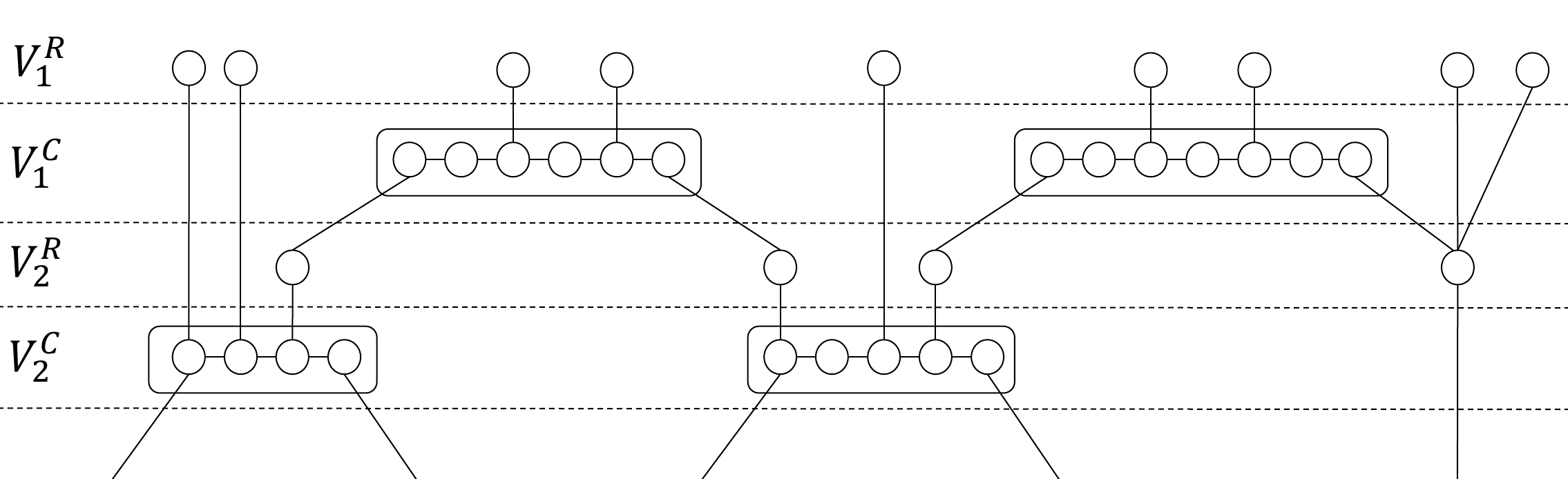}
		\caption{A  decomposition satisfying the requirements in \cref{lem-rake-and-compress-modified}. }
		\label{fig:rake_and_compress}
	\end{figure}
    
    Now we are ready to show \cref{thm:ellfull_to_logn}.
	\begin{proof}[Proof of \cref{thm:ellfull_to_logn}]
		The algorithm first computes a decomposition of \cref{lem-rake-and-compress-modified} with $\ell' = \max\{1, \ell - 2\}$ in $O(\log n)$ rounds.
		Given such a decomposition  \[V(T) = \VR{1} \cup \VC{1} \cup \VR{2} \cup \VC{2} \cup \VR{3} \cup \VC{3} \cup \cdots \cup \VR{L} \ \  \text{with $L = O(\log n)$},\] we  present an algorithm that solves $\Pi$   in $O(\log n)$ rounds by labeling the vertices in this order: $\VR{L}$, $\VC{L-1}$, $\VR{L-1}$, $\ldots$, $\VR{1}$. The algorithm only uses the vertex configurations in the $\ell$-full set $\fV'$.
		
		\paragraph{Labeling $\VR{i}$}
		Suppose  all vertices in $\VR{L}, \VC{L-1}, \VR{L-1}, \ldots, \VC{i}$ have been labeled using $\fV'$.  By \cref{lem-rake-and-compress-modified},   each $v \in \VR{i}$ has at most one neighbor in  $\VR{L} \cup \VC{L-1} \cup \VR{L-1} \cup \cdots \cup \VC{i}$.
		If $v \in \VR{i}$ has no neighbor in $\VR{L} \cup \VC{L-1} \cup \VR{L-1} \cup \cdots \cup \VC{i}$, then we can label the half-edges surrounding $v$ by any $c \in \fV'$ arbitrarily.
		Otherwise, $v \in \VR{i}$ has exactly one neighbor $u$ in $\VR{L} \cup \VC{L-1} \cup \VR{L-1} \cup \cdots \cup \VC{i}$. Suppose the vertex configuration of $u$ is $c$, where the half-edge label on $\{u,v\}$ is $\ta \in c$.  A simple observation from the definition of $\ell$-full sets is that for any $c \in \fV'$ and any $\ta \in c$, there exist $c' \in \fV'$ and $\ta' \in c'$ so that $\{\ta, \ta'\} \in \fE$.
		Hence we can label the half-edges surrounding $v$ by  $c' \in \fV'$  where the half-edge label on $\{u,v\}$ is $\ta' \in c'$.

		\paragraph{Labeling $\VC{i}$} Suppose  all vertices in $\VR{L}, \VC{L-1}, \VR{L-1}, \ldots, \VR{i+1}$ have been labeled using $\fV'$.  By \cref{lem-rake-and-compress-modified}, $\VC{i}$ is a collection of disjoint paths $P_1, P_2, \ldots, P_k$.
		Moreover, the number of nodes $s$ in each   path $P_j =(v_1, v_2, \ldots, v_s)$ satisfies $\ell' \leq s \leq 2 \ell'$, and there are two vertices $u$ and $w$ in $\VR{L} \cup \VC{L-1} \cup \VR{L-1} \cup \cdots \cup \VR{i+1}$ such that  $u$ and $w$ are the only vertices in $\VR{L} \cup \VC{L-1} \cup \VR{L-1} \cup \cdots \cup \VR{i+1}$ that are adjacent to $P_j$,  $u$ is adjacent to $v_1$, and $w$ is adjacent to $v_s$.
		Now consider the path $P_j'= (u, v_1, v_2, \ldots, v_s, w)$. As $u$ and $w$ are in $\VR{L} \cup \VC{L-1} \cup \VR{L-1} \cup \cdots \cup \VR{i+1}$, they have been assigned half-edge labels using $\fV'$. Since  $P'$ contains at least $\ell'+2 \geq \ell$ vertices, the definition of $\ell$-full sets  ensures that we can label $v_1, v_2, \ldots, v_s$ using vertex configurations in $\fV'$ in such a way that the half-edge labels on $\{x, v_1\}, \{v_1, v_2\}, \ldots, \{v_s, y\}$ are all in $\fE$.
	\end{proof}

	\subsection{Necessity}
	
	We start by showing that (3.) in \cref{thm:MainBaireLog} is necessary for $\baire$.
	
	\begin{theorem}[Bernshteyn \cite{Bernshteyn_work_in_progress}]\label{thm:baireellfull}
		Let $\Pi=(\Sigma,\fV,\fE)$ be an LCL problem and suppose that $\Pi\in \baire$.
		Then $\Pi$ admits an $\ell$-full set $\fV'\subseteq\fV$ for some $\ell>0$.
	\end{theorem}
	\begin{proof}
		Suppose that every Borel $\Delta$-regular forest admits a Borel solution on a $\tau$-comeager set for every compatible Polish topology $\tau$.
		In particular, this holds for the Borel graph $\fG$ induced by the standard generators of the group $\Gamma$, the free product of $\Delta$-copies of $\mathbb{Z}_2$, on the free part of the shift action $\Gamma\curvearrowright \free(\{0,1\}^\Gamma)$ endowed with the product topology.
		Let $F$ be such a $\Pi$-coloring of $\fG$.
		Using the canonical edge $\Delta$-coloring $F$ gives rise to a map $\bar{F}:\free(\{0,1\}^\Gamma) \to \Sigma^\Delta$.
		Write $\fV'\subseteq \fV$ for the configurations of half-edge labels around vertices that $\bar{F}$ outputs on a non-meager set.
		Let $C$ be a $\Gamma$-invariant comeager set on which $\bar{F}$ is continuous. Note that we have that $\rng(\bar{F})\subseteq \mathcal{V}'$, and for any given element of $\fV'$ there is a map $s:B_{\operatorname{Cay}(\Gamma)}({\bf 1}_\Gamma,k)\to \{0,1\}$ such that if $x\in C$ satisfies $x\upharpoonright B_{\operatorname{Cay}(\Gamma)}({\bf 1}_\Gamma,k)=s$, then $\bar{F}$ outputs the given element of $\fV'$.
		Since $\fV'$ is finite, we can take $t>0$ to be the maximum of such $k$'s.
		
		We claim that $\fV'$ is $\ell$-full for $\ell=2t+2$.
        Indeed, let $P$ be a spiky path of length $k\ge \ell$ and $c_1,c_2\in \mathcal{V}'$ be an arbitrary labelings of half-edges around the endpoints of $P$.
        Assume additionally that the half-edges around the endpoints of $P$ are colored by the generators of $\Gamma$ in such a way that $c_1,c_2$ can be viewed as elements of $\Sigma^\Delta$.
        This is possible by the definition of $\mathcal{V}'$ and $\bar{F}$.
        Write $\alpha,\beta\in \Delta$ for the edge labels that correspond to the real edges of the endpoints of $P$.
        By the choice of $t$, $\ell$ and $\mathcal{V}'$, we can find a reduced word $w=\gamma_{k-1}\cdot\cdots\gamma_{1} \in \Gamma$ and $s:B_{\operatorname{Cay}(\Gamma)}({\bf 1}_\Gamma, k+t)\to \{0,1\}$ such that $\alpha=\gamma_1$, $\beta=\gamma_{k-1}$, $s \restriction B_{\operatorname{Cay}(\Gamma)}({\bf 1}_\Gamma,t)=c_1\in \Sigma^\Delta$ and $w \cdot s \restriction B_{\operatorname{Cay}(\Gamma)}({\bf 1}_\Gamma,t)=c_2\in \Sigma^\Delta$, where $w \cdot s(w')=s(w^{-1}w')$.
		Since $C$ is comeager, there is some $x \in C$ with $x\restriction B_{\operatorname{Cay}(\Gamma)}({\bf 1}_\Gamma, \ell+k)=s$. Then, the function $F$ evaluated on the half-edges belonging to the spiky path $(x,\gamma_1\cdot x,\gamma_2\gamma_1\cdot x,\dots, \gamma_\ell \cdots \gamma_1 \cdot x)$ witnesses the $\ell$-fullness, by the $\Gamma$-invariance of $C$.
	\end{proof}
	
	To finish the proof of \cref{thm:MainBaireLog} we need to demonstrate the following theorem.  Note that  $\local(n^{o(1)}) =  \local(O(\log n))$ according to the $\omega(\log n)$ -- $n^{o(1)}$ complexity gap~\cite{chang_pettie2019time_hierarchy_trees_rand_speedup}.
	
	\begin{restatable}{theorem}{logcomb}\label{thm:logn_to_ellfull}
		Let $\Pi = (\Sigma, \fV, \fE)$ be an LCL problem solvable in  $\local(n^{o(1)})$ rounds.
		Then there exists an $\ell$-full set  $\fV' \subseteq \fV$ for some $\ell \geq 2$.
	\end{restatable}
	The next section is devoted to showing this theorem\footnote{An earlier version of this paper contained a different argument, which was less self-contained; see \href{https://arxiv.org/abs/2204.09329v1}{here} or \href{https://vidnyanz.elte.hu/szem/top.html}{here} for more on the development of this proof.}.

	\section{$\local(n^{o(1)})$ implies $\ell$-full}
    \label{sec:implieslfull}

	\subsection{Basic setup and observations}
	Recall \cref{def:Deltareg}: for a finite $\Delta$-regular forest $T$, we denote by $H(T)$ the set of half-edges, $\Hr(T)$ the set of real half-edges and ${\Hv} (T)=H(T)\setminus \Hr (T)$ the set of virtual half-edges.
	For $v\in V(T)$, we write $N^H(v)$ for the set of half-edges adjacent to $v$, ${\Nr} (v)=N^H(v)\cap H_r(T)$ and ${\Nv} (v)=N^H(v)\cap \Hv (T)$.
	We write $\deg_T(v)=|\Nr (v)|$.
	Given $e=\{v,w\}\in E(T)$, where $v,w\in V(T)$, there are $g\in \Nr(v)$ and $h\in \Nr(w)$ that belong to $e$.
	We abuse the notation and denote this fact by $e=\{g,h\}$.
	
	\medskip

	Fix an LCL problem $\Pi=(\Sigma,\mathcal{V},\mathcal{E})$.
	
	\begin{definition}[Bipolar tree]
		\label{def:rooted}
		A \emph{(decorated) bipolar tree} is a triple ${\bf T}=(T,(p,q),\mathfrak{d})$, where
		\begin{enumerate}
			\item $T$ is a finite $\Delta$-regular tree,
			\item $(p,q)\in V(T)^2$ with $\deg_T(p)=\deg_T(q)=\Delta-1$ (note that we allow $p=q$),
			\item $\mathfrak{d}$ is a decoration of $T$ consisting of
			\begin{enumerate}
				\item a partial labeling $\varphi:H(T) \rightharpoonup \Sigma$,
				\item a partial map $\alpha:V(T) \rightharpoonup \N$,
				\item a rank function $\rn:V(T) \to \N$,
				\item a set $D \subseteq V(T)$ of distinguished vertices.
			\end{enumerate}
			
		\end{enumerate}
		We will call $p$ the \emph{positive pole}, and $q$ the \emph{negative pole}, and denote by $g^p, g^q$ the incident virtual half-edges.	
	\end{definition}
	
	For a bipolar tree $\mb{T}$, we will use the convention of denoting $\rn_\mb{T}$, $D_{\mb{T}}$, etc. the corresponding object, while we will use just $T$ for the underlying tree.  
	
	We will define the key equivalence relation on bipolar trees.
	
	\begin{definition}[Equivalence]
		\label{def:equiv} Let $\mathbf{T}$ an $\mathbf{T}'$ be a bipolar trees. Say that $\mb{T} \simm \mb{T}'$ if for every $(\pi,\rho) \in \Sigma^2$, the following are equivalent:
		
		\begin{enumerate}
			\item $\varphi_\mb{T}$ extends to a $\Pi$-coloring $\widehat{\varphi}_{\mb{T}}$ of $T$ with $\widehat{\varphi}_{\mb{T}}(g^{p_\mb{T}})=\pi$ and $\widehat{\varphi}_{\mb{T}}(g^{q_\mb{T}})=\rho$
			\item  $\varphi_{\mb{T}'}$ extends to a $\Pi$-coloring $\widehat{\varphi}_{\mb{T}'}$ of $T$ with $\widehat{\varphi}_{\mb{T}'}(g^{p_{\mb{T}'}})=\pi$ and $\widehat{\varphi}_{\mb{T}'}(g^{q_{\mb{T}'}})=\rho$
		\end{enumerate} 
		
	\end{definition}
	We will use $[\mb{T}]$ to denote the $\simm$-equivalence class of $\mb{T}$, and for a set of bipolar trees $W$ denote by $[W]$ the collection of represented equivalence classes. The following is immediate from the definition.
	\begin{claim}
		\label{cl:basic}
		$[\mb{T}]$ does not depend on $(\alpha_{\mb{T}},\rn_{\mb{T}},D_{\mb{T}})$. 
		There are finitely many $[\mb{T}]$ equivalence classes.
	\end{claim}
	
	Note that whether the labeling $\varphi_{\mb{T}}$ can be extended to a $\Pi$-coloring depends only on the equivalence class of $\mb{T}$, so the following definition makes sense. 
	
	\begin{definition}[Valid class]
		An equivalence class $[\mb{T}]$ is called \emph{valid}, if for every (equivalently, for some) $\mb{T}$ in the class, $\varphi_\mb{T}$ can be extended to a $\Pi$-coloring of $T$.  
	\end{definition}

	Let us make an important, but easy observation. Assume that $\mb{S}, \mb{T}$ are bipolar trees. We say that \emph{$\mb{T}$ is attached to the rest of $\mb{S}$ by its poles}, if $T$ is a subtree of $S$ and every path from a vertex in $V(S)\setminus V(T)$ to a vertex in $V(T)$ goes through a pole of $\mb{T}$. Observe that in this case, if $\mb{T} \simm \mb{T}'$, there is a natural way of replacing $\mb{T}$ by $\mb{T}'$ in $\mb{S}$ and create a new bipolar tree $\mb{S'}$, namely, let
	\[H(S')=\left(H(S) \setminus H(T)\right) \cup H(T'), V(S')=\left(V(S) \setminus V(T)\right) \cup V(T'),\]
	and \[E(S')=\left(E(S) \setminus E(T)\right) \cup E(T') \cup \{\{v,p_{\mb{T'}}\},\{w,q_{\mb{T'}}\}\},\] where $v,w \in V(S) \setminus V(T)$ so that $\{v,p_{\mb{T}}\},\{w,q_{\mb{T}}\} \in E(T)$, if there are such vertices. We also replace $\mathfrak{d}_{\mb{T}}$ by  $\mathfrak{d}_{\mb{T}'}$ as well in the obvious way.

	\begin{claim} 
		\label{cl:replace} 
		
		Let $\mb{S}$, $\mb{T}$ be bipolar trees, assume that $\mb{T}$ is attached to the rest of $\mb{S}$ by its poles, $\mb{T} \simm \mb{T}'$ and $\mb{S}'$ arises from replacing $\mb{T}$ by $\mb{T}'$ in $\mb{S}$. Then for any extension $\widehat{\varphi}_{\mb{S}}$ of $\varphi_{\mb{S}}$ to a $\Pi$-coloring of $S$ there is an extension $\widehat{\varphi}_{\mb{S}'}$ of $\varphi_{\mb{S}'}$ to a $\Pi$-coloring, so that $\widehat{\varphi}_{\mb{S}'}\restriction H(S) \cap H(S')=\widehat{\varphi}_{\mb{S}}\restriction H(S) \cap H(S')$.
		
		In particular, $\mb{S} \simm \mb{S}'$.
		
	\end{claim}
	\begin{proof}
		Let $\pi$ and $\rho$ be the value of $\widehat{\varphi}_{\mb{S}}$ on $g^{p_\mb{T}}$ and $g^{q_\mb{T}}$, that is, the virtual half-edges belonging to the poles of $\mb{T}$. Since $\mb{T} \simm \mb{T}'$, there is a $\Pi$-coloring $\widehat{\varphi}_{\mb{T}'}\supseteq \varphi_{\mb{T}'}$ of $\mb{T}'$ extending its partial coloring, so that the value on the virtual half-edges belonging to the poles of $\mb{T}'$ is also $\pi$ and $\rho$, respectively. But then, since $\mb{T}'$ is attached to the rest of $\mb{S}'$ by the poles, and the labeling is correct around them, we have that $\widehat{\varphi}_{\mb{S}'}:=\widehat{\varphi}_{\mb{T}'} \cup \widehat{\varphi}_{\mb{S}} \restriction (H(S) \cap H(S'))$ is a $\Pi$-coloring extending $\varphi_{\mb{S}'}$. 
		
		The last statement follows from the definition of $\simm$. 
	\end{proof}

	\subsection{Gluing and labeling}

	We will create collections of bipolar trees using a simple gluing operation. 
	
	\begin{definition}[Gluing]
		\label{def:kglue}
		Let $\ell >0$ and $(\mathbf{T}^i_j)_{1 \leq i \leq 2\ell+1, 2 \leq j \leq \Delta-1}$ be a collection of bipolar trees, with positive poles $p^i_j(=p_{\mb{T}^i_j})$, and corresponding virtual edges $g^i_j(=g^{p_{\mb{T}^i_j}})$. We define a bipolar tree $\mb{H}$ as follows by gluing $\mb{T}^i_j$ along the positive poles to a path of length $2\ell+1$ as follows. 
		Take $2\ell+1$-many one vertex $\Delta$-regular trees $(r^i)^{2\ell+1}_{i=1}$ with adjacent half edges $(h^i_j)^\Delta_{j=1}$. 
		\begin{enumerate}
			\item $H$ is the disjoint union of the trees $(r^i)_{i \leq 2\ell+1}$ and the trees $(T^i_j)$ to which we add real edges
			$$\{\{h^i_\Delta,h^{i+1}_1\}:1\le i\le 2\ell\}\cup \{\{g^i_j,h^i_j\}:1\le i\le 2\ell+1, \ 2\le j\le \Delta-2\}$$
			by connecting the corresponding half-edges (see \cref{f:glue}),
			\item $p^\mb{H}=r^1$, $q^{\mb{H}}=r^{2\ell+1}$,
			\item 
			\begin{enumerate}
				\item $\varphi_{\mb{H}}$ is the union of $\varphi_{\mathbf{T}^i_j}$,
				\item $\alpha_{\mb{H}}$ is the union of $\alpha_{\mathbf{T}^i_j}$,
				\item let $\rn_{\mb{H}}(v)=\rn_{{\bf T}^i_j}(v)$, when this is defined, and \[\rn_{\mb{H}}(v)=\max\{\rn_{\mb{T^i_j}}(w):w \in \bigcup_{i,j} V(T^i_j)\}+1\] otherwise,
				\item $D_{\mb{H}}=\{r^{\ell+1}\}\cup \bigcup_{i,j}D_{\mb{T}^i_j}$.
			\end{enumerate} 
		\end{enumerate}
		In such a construction, we will refer to $r^{\ell+1}$ as the \emph{middle vertex}, and to the collection $(\mathbf{T}^i_j)_{1 \leq i \leq 2\ell+1, 2 \leq j \leq \Delta-1}$ as an \emph{array of bipolar trees}. 
	\end{definition}
	We will refer to the created bipolar tree as $\mb{H}((\mathbf{T}^i_j)_{1 \leq i \leq 2\ell+1, 2 \leq j \leq \Delta-1})$, or $\mb{H}((\mathbf{T}^i_j)_{i,j})$ or just $\mb{H}$ if the array is clear from the context. 
	\begin{figure}
		\noindent\includegraphics[width=\linewidth]{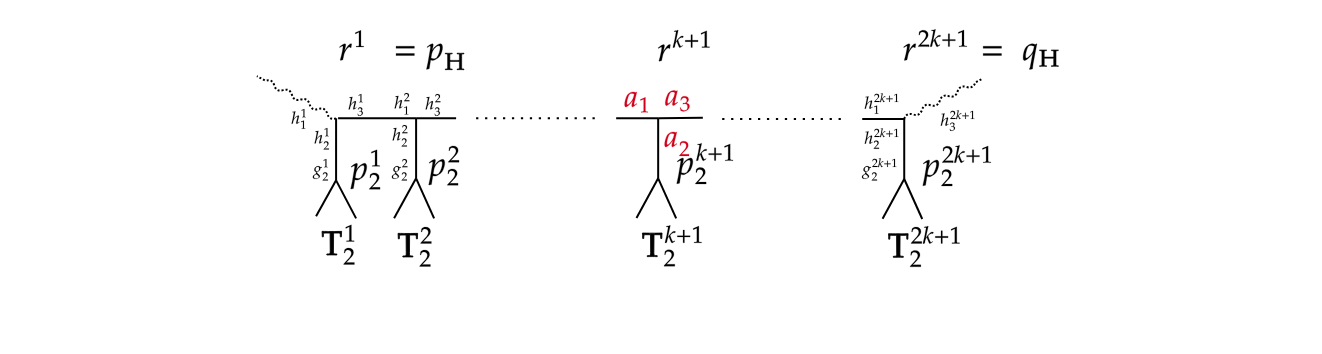}
		\caption{The gluing and labeling procedures for $\Delta=3$}
		\label{f:glue}
	\end{figure}
	Now we extend the partial labeling of $\mb{H}$ around the middle vertex, using a sequence $\mb{a} \in \Sigma^\Delta$. 
	\begin{definition}[$\mb{H}_\mb{a}$]
		If $\mb{a}=(a_1,\dots,a_\Delta) \in \Sigma^\Delta$, and $(\mathbf{T}^i_j)_{1 \leq i \leq 2\ell+1, 2 \leq j \leq \Delta-1}$ is an array, define $\mb{H}_a$ exactly as $\mb{H}$, but with extending the labeling $\varphi_\mb{H}$ to the half-edges $(h^{\ell+1}_j)_{1 \leq j \leq \Delta}$ of $r^{\ell+1}$ with $\varphi_{\mb{H}_{\mb{a}}}(h^{\ell+1}_j)=a_j$. 
		
	\end{definition}

	\begin{claim}
		\label{cl:class}
		The operation $\mb{H}$ and for each $\mb{a}$ the operations $\mb{H}_\mb{a}$ preserve $[\cdot]$-classes, that is, if $({\bf T}^i_j)_{1 \leq i \leq 2\ell+1,2 \leq j \leq \Delta-1}$ and $({\bf T'}^i_j)_{1 \leq i \leq 2\ell+1,2 \leq j \leq \Delta-1}$ are such that ${\bf T}^i_j \simm {\bf T'}^i_j$ for all $i,j$, then for the corresponding trees $\mb{H}\simm \mb{H}'$ and $\mb{H}_\mb{a}\simm \mb{H}'_\mb{a}$ hold.
	\end{claim}
	\begin{proof}
		Apply \cref{cl:replace} and replace ${\bf T}^i_j$ by ${\bf T'}^i_j$ one-by-one. Such an operation does not change the $\classs{\cdot}$-class of $\bf{H}$ and transforms it to ${\bf H'}$. The same argument works for $\mb{H}_\mb{a}$.
	\end{proof}

	The core idea is to repeatedly apply the procedure of creating $\mb{H}_\mb{a}$ to a collection of bipolar trees, using some function $f$ to obtain $\mb{a}$. The whole process will depend only on $\simm$-equivalence classes. As we will see, it will be also necessary to be able to switch the polarity, i.e., to exchange the poles. 
	
	\begin{definition}[$\Gamma_{\ell,f}$] \label{def:Gamma} For a bipolar tree $\mb{T}$, denote by $\bar{\mb{T}}$ the bipolar tree obtained by switching its positive and negative poles. 
		
		Assume that $W$ is a collection of bipolar trees, and $f:[W]^{\sizeofarray} \to \Sigma^\Delta$. 
		Define $\mathbf{T} \in \Gamma_{\ell,f}(W)$ if for some $(\mathbf{T}^i_j)_{i,j} \in W^{(2\ell+1) \times (\Delta-2)}$ we have \[\mathbf{T}=\mathbf{H}_{f(([\mathbf{T}^i_j])_{i,j})}((\mathbf{T}^i_j)_{i,j}) \text{ or }\mathbf{T}=\bar{\mathbf{H}}_{f([\mathbf{T}^i_j]})_{i,j})((\mathbf{T}^i_j)_{i,j}).\]
	\end{definition}
	
	Let us make an observation about transferring this map to a map on collections of equivalence classes.
	
	\begin{claim}
		\label{cl:well-defd}
		Let $\ell>0$, $S$ be the collection of equivalence classes of bipolar trees, $X \subseteq S$ nonempty, and $f:X^{\sizeofarray} \to \Sigma^\Delta$. Define $[\Gamma_{\ell,f}]:\mathcal{P}(X) \setminus \{\emptyset\} \to \mathcal{P}(S) \setminus \{\emptyset\}$ by $[\Gamma_{\ell,f}](Y)=[\Gamma_{\ell,f}(W)]$, where $W$ is arbitrary with $[W]=Y$, and $Y \subseteq X$ non-empty. Then 
		\begin{itemize}
			\item $[\Gamma_{\ell,f}]$ is well-defined and 
			\item monotone, i.e., $Y \subseteq Y'$ implies $[\Gamma_{\ell,f}](Y)\subseteq[\Gamma_{\ell,f}](Y')$. 
		\end{itemize}
	\end{claim}
	\begin{proof}
		The fact that $[\Gamma_{\ell,f}]$ is well-defined follows from \cref{cl:class}, while monotonicity comes from the fact that $Y \subseteq Y'$ implies $Y^{\sizeofarray} \subseteq Y'^{\sizeofarray}$ and the definition of $\Gamma_{\ell,f}$. 
	\end{proof}

	\subsection{Fixed points}

	The crux of the argument is the following fixed point theorem, which we will prove in this section. 
	
	\begin{theorem}\label{thm:MainConstruction}
		There are $\ell>0$, a set $X \neq \emptyset$ of equivalence classes of bipolar trees, and $f:X^{\sizeofarray} \to \Sigma^\Delta$ such that
		$X=[\Gamma_{\ell,f}](X),$
		and every class in $X$ is valid. 
	\end{theorem}
	The number $\ell$ will be found in \cref{lm:pump}, $X$ will be constructed in Section \ref{s:finalconstruct}, while $f$ will come from \cref{cor:existscertif} below.

	\subsubsection{Tools: pumping and fixed points}
	Let us first establish our basic tool, providing us the number $\ell$. 
	\begin{lemma}
		\label{lm:pump}
		There exists a number $\ell>0$ so that for any array $(\mb{T}^i_j)_{1 \leq i \leq 2\ell+1, 2 \leq j \leq \Delta-1}$ of bipolar trees and any $t \ge \ell$, there exists an $m \in \{t,\dots,t+\ell\}$ and a function $\Lambda:\{1,\dots,2m+1\} \to \{1,\dots,2\ell+1\}$ so that $\mb{H}_{\mb{a}}((\mb{T}^i_j)_{1 \leq i \leq 2\ell+1, 2 \leq j \leq \Delta-1}) \simm \mb{H}_{\mb{a}}((\mb{T}^{\Lambda(i)}_j)_{1 \leq i \leq 2m+1,2\leq j\leq \Delta-1})$ for every $\mb{a}\in \Sigma^\Delta$. 
	\end{lemma}
	
	The idea is that for every $k \geq 1$, we look at the equivalence classes of the trees obtained by cutting off the trees farther than $k$ from the middle vertex. If $\ell$ is large enough, some class will be repeated. Then, replacing the smaller tree with the larger one will not change the class of the original tree. This allows us to find arbitrary large $m$'s with the above property.
	\begin{proof}
		Let $N$ be the number of $\simm$ equivalence classes of bipolar trees and let $\ell>N^{|\Sigma^\Delta|}$. To ease the notation, for $m \ge 1$ and a map $\Lambda:\{1,\dots,2m+1\} \to \{1,\dots,2\ell+1\}$, we will denote by $\mb{H}_\mb{a}(\Lambda)$ the tree $\mb{H}_{\mb{a}}((\mb{T}^{\Lambda(i)}_j)_{1 \leq i \leq 2m+1,2\leq j\leq \Delta-1})$. Note that for every $t\ge \ell$ we have to find $m \in \{t,\dots,t+\ell\}$ and $\Lambda$ so that $\mb{H}_\mb{a}(\Lambda)\simm \mb{H}_\mb{a}(\id_{2\ell+1})$ holds for every $\mb{a}\in \Sigma^\Delta$.
		
		For each $k \leq \ell$, consider the map $\Lambda_k: \{1,\dots,2k+1\} \to \{1,\dots,2\ell+1\}$ defined by $\Lambda_k(i)=\ell-k+i$. To each $\Lambda_k$, we can associate the sequence of $\simm$ equivalence classes of the trees $([\mb{H}_{\mb{a}}(\Lambda_k)])_{\mb{a} \in \Sigma^\Delta}$. By the pigeonhole principle, there are $k<k' \leq \ell$, so that $([\mb{H}_{\mb{a}}(\Lambda_k)])_{\mb{a} \in \Sigma^\Delta}=([\mb{H}_{\mb{a}}(\Lambda_{k'})])_{\mb{a} \in \Sigma^\Delta}$, i.e., for each $\mb{a}$, $\mb{H}_{\mb{a}}(\Lambda_k) \simm \mb{H}_{\mb{a}}(\Lambda_{k'})$.
		
		For $l\ge \ell$, assume that $\Lambda^l: \{1,\dots,2l+1\} \to \{1,\dots,2\ell+1\}$ is a map, so that the interval $[l-k+1,l+k+1]$ is monotonly bijected onto $[\ell-k+1,\ell+k+1]$ and for all $\mb{a}$ we have $\mb{H}_{\mb{a}}(\Lambda^l) \simm \mb{H}_{\mb{a}}(\id_{2\ell+1}).$ Observe that for the map $\Lambda^{l+k'-k}:\{1,\dots,2(l+k'-k)+1\} 
		\to \{1,\dots,2\ell+1\}$ defined by 
		\[\Lambda^{l+k'-k}(i)=\begin{cases}
		\Lambda^l(i), \text{ if } i<l-k+1\\
		i-(l-k+1)+(\ell-k'+1), \text{ if } l-k+1\le i \le l-k+1+2k'\\
		\Lambda^l(i-2k'+2k), \text{ if } l-k+1+2k'<i,
		\end{cases}\]
		we have $\mb{H}_{\mb{a}}(\Lambda^{l+k'-k}) \simm \mb{H}_{\mb{a}}(\id_{2\ell+1})$ for each $\mb{a}$: indeed, the tree $\mb{H}_{\mb{a}}(\Lambda^{l+k'-k})$ arises from $\mb{H}_{\mb{a}}(\Lambda^{l})$ by replacing a copy of $\mb{H}_{\mb{a}}(\Lambda_k)$ with $\mb{H}_{\mb{a}}(\Lambda_{k'})$, hence by tree replacement, i.e., \cref{cl:replace}, the equivalence class is not changed. 
		
		Using this observation repeatedly, starting from $l=\ell$, for all $j\in \N$, we can find maps $\Lambda^{\ell+j(k'-k)}:\{1,\dots,2(\ell+j(k'-k))+1)\} 
		\to \{1,\dots,2\ell+1\}$ with $\mb{H}_{\mb{a}}(\Lambda^{\ell+j(k'-k)}) \simm \mb{H}_{\mb{a}}(\id_{2\ell+1})$ for all $\mb{a}$. In particular, for all $t \geq \ell$, there is some $m \in \{t,\dots,t+\ell\}$ with $m=\ell+j(k'-k)$, so $\Lambda=\Lambda^m$ works as desired.  
	\end{proof}
	Now let us fix such an $\ell$ for the rest of the proof. 
	
	We will use the following version of Tarski's fixed point theorem. 
	\begin{proposition}
		\label{pr:UniformBound0}
		Let $S$ be a finite set, $\Gamma:\mathcal{P}(S)\setminus\{\emptyset\} \to \mathcal{P}(S)\setminus \{\emptyset\}$ be a monotone map. Then there is an $X$ with $\Gamma(X)=X$.  
	\end{proposition}
	\begin{proof}
		As $\mathcal{P}(S)$ is finite, we have that $\Gamma^{i+k}(S)=\Gamma^{i}(S)$ for some $i \in \N$ and $k>0$. Note that $\Gamma(S) \subseteq S$, hence by monotonicity, for all $i$ we have $\Gamma^{i+1}(S) \subseteq \Gamma^i(S)$. But then $\Gamma^{i+k'}(S)=\Gamma^{i}(S)$ for all $k' \leq k$, so $X=\Gamma^{i}(S)$ is a fixed point. 
	\end{proof}

	\subsubsection{Certificates}
	
	Now we identify particularly nice elements of equivalence classes of trees. Recall that our assumption is that $\Pi\in \local(n^{o(1)})$. Let $\mathcal{A}=(\mathcal{A}_n)_n$ be the corresponding local algorithm of complexity $(t_n)_n\in n^{o(1)}$.

	\begin{definition}[Certificates]\label{def:Certificate}
		Let $[{\bf T}]$ be an equivalence class of bipolar trees, $n,k \in \N$.
		An \emph{$(n,k)$-certificate for $[{\bf T}]$} is an element $\mb{U} \in [\mb{T}]$ such that:
		\begin{enumerate}
			\item $|V(U)|\leq k$,
			\item $\alpha_{\mb{U}}:V(U)\rightharpoonup \{1,2,\dots,n\}$ is injective in every $(t_n+1)$-neighborhood (i.e., every set of the form $\mathcal{B}_U(v,t_n+1)$, $v \in V(U)$),
			\item $\mathcal{B}_U(p_\mb{U},5t_n)\cap \dom(\alpha_{\mb{U}})=\mathcal{B}_U(q_\mb{U},5t_n)\cap \dom(\alpha_{\mb{U}})=\emptyset$,
			\item if we evaluate $\mathcal{A}_n$ on all the half-edges incident to vertices $v$ with $\dom \alpha_{\mb{U}} \supseteq B_U(v,t_n)$ using $\alpha_\mb{U}$ as ID's, then we get a partial labeling extending $\varphi_{\mb{U}}$. 
		\end{enumerate} 
	\end{definition}

	The next lemma explains the name ``certificate". Recall that an equivalence class is valid, if the labelings of its elements extend to full $\Pi$-colorings.
	\begin{lemma}\label{lem:Certificate}
		Assume that $[\mb{T}]$ has an $(n,k)$-certificate with $1 \leq k\leq n$, then $[\mb{T}]$ is valid.
	\end{lemma}
	\begin{proof}
		Let $\mb{U} \in [\mb{T}]$ be a certificate, we have to show that $\varphi_{\mb{U}}$ extends to a $\Pi$-coloring. Let $l=|V(U)\setminus \dom(\alpha_{\mb{U}})|$ and observe that by \cref{def:Certificate} (1) and $k \le n$, we have that $l\le n-|\rng(\alpha_{\mb{U}})|$. Hence, there is an $\widehat\alpha_{\mb{U}}:V(U)\to \{1,2,\dots,n\}$ extending $\alpha_{\mb{U}}$ that is injective on $V(U)\setminus \dom(\alpha_{\mb{U}})$ and satisfies
		$$\rng(\widehat\alpha_{\mb{U}} \restriction (V(U)\setminus \dom(\alpha_{\mb{U}})))\cap \rng(\alpha_{\mb{U}})=\emptyset.$$
		Now, using (2) it follows that $\widehat \alpha_{\mb{U}}$ is injective in every $(t_n+1)$-neighborhood. Therefore, $\mathcal{A}_n$ must produce a $\Pi$-coloring $\widehat{\varphi}$ when evaluated using $\widehat{\alpha}$ values as ID's. By (4) we have $\widehat{\varphi} \supseteq \varphi_{\mb{U}}$. 
	\end{proof}
	
	The following claim is immediate from the definition. 
	\begin{claim}
		\label{cl:sym}
		\begin{itemize}
			\item If $\mb{U}$ is an $(n,k)$-certificate for $[\mb{T}]$, then $\bar{\mb{U}}$ is an $(n,k)$-certificate for $[\bar{\mb{T}}]$.
			\item If $\mb{U}$ is a single vertex tree with $\alpha_{\mb{U}}=\varphi_{\mb{U}}=\emptyset$ and $p_\mb{U}=q_{\mb{U}}$, then $\mb{U}$ is an $(n,1)$-certificate for $[\mb{U}]$ for each $n \geq 1$. 
		\end{itemize}
	\end{claim}

	\subsubsection{Existence of certificates}
	Now we show that assuming the trees used in $\mb{H}$'s construction have certificates, we can create certificates for $\mb{H}$ of somewhat bigger size. The naive idea is to just glue the certificates themselves as in the construction of $\mb{H}$. However, this would typically contradict (3) in \cref{def:Certificate}. The trick is to use the pumping lemma \cref{lm:pump} to make a sufficient amount of space.
	
	Recall that $\ell$ is chosen using \cref{lm:pump}, and $(t_n) \in n^{o(1)}$. 
	\begin{theorem} 
		\label{thm:existscertif}
	
		Assume that $k\geq 1$, and $n$ is such that $10t_n \geq 2\ell+1$ and $n\ge 30\Delta t_nk$. Moreover, assume that $(\mb{T}^i_j)_{1\leq i \leq 2\ell+1,2 \leq j\leq \Delta-1}$ is an array of bipolar trees, so that every $[\mb{T}^i_j]$ admits an $(n,k)$-certificate. Then there exists an $\mb{a} \in \Sigma^\Delta$ so that $[\mb{H}_\mb{a}((\mb{T}^i_j)_{1\leq i \leq 2\ell+1,2 \leq j\leq \Delta-1})]$ admits an $(n,30\Delta t_nk)$-certificate.
	\end{theorem}
	
	\begin{proof}
		Fix $(n,k)$-certificates $\mb{U}^i_j$ for $[\mb{T}^i_j]$ for every $1\le i\le 2\ell+1$ and $2\le j\le\Delta-1$. Now, we can use \cref{lm:pump} to find some $m \in \{10t_n,\dots,10t_n+\ell\}$ and
		$$\Lambda:\{1,\dots,2m+1\}\to \{1,\dots,2\ell+1\}$$
		so that for all $\mb{a} \in \Sigma^\Delta$ we have $\mb{H}_{\mb{a}}((\mb{U}^i_j)_{1 \leq i \leq 2\ell+1,2\leq j\leq \Delta-1}) \simm \mb{H}_{\mb{a}}((\mb{U}^{\Lambda(i)}_j)_{1 \leq i \leq 2m+1,2\leq j\leq \Delta-1})$. 
		
		Let $\mb{U}_0$ be the bipolar tree $\mb{H}((\mb{U}^{\Lambda(i)}_j)_{1 \leq i \leq 2m+1,2\leq j\leq \Delta-1})$. 
		Observe that by \cref{def:Certificate} (1) and the choice of $m$ and $t_n$ we have
		\begin{equation}\label{eq:Size1}\tag{$\dagger$}
		|V(U_0)|\le (\Delta-2)(2m+1)k+2m+1\leq (20t_n+2\ell+1)\Delta k \leq 30t_n\Delta k.
		\end{equation}
		
		We will find an $\mb{a}\in \Sigma^\Delta$ and extensions $\widehat{\varphi}_{\mb{U}_0} \supseteq \varphi_{\mb{U}_0}$, $\widehat{\alpha}_{\mb{U}_0} \supseteq \alpha_{\mb{U}_0}$ so that \[\mb{U}=(U_0,(p^{\mb{U}_0},q^{\mb{U}_0}),(\widehat{\varphi}_{\mb{U}_0},\widehat{\alpha}_{\mb{U}_0},\rn_{\mb{U}_0},D_{\mb{U}_0}))\] is a  $(n,30\Delta t_{n}k)$-certificate and $\mb{U} \simm \mb{H}_\mb{a}((\mb{U}^i_j)_{i,j})$, which is sufficient to show the theorem. 
		
		Let $(r^1,\dots,r^{2m+1})$ be the vertices added to $\bigcup_{1\leq i\leq 2m+1, 2\leq j\leq \Delta-2}V(U^i_j)$ in $V(U_0)$. Now observe that by \cref{def:Certificate} (2) and (3),  $\alpha_{\mb{U}_0}$ is injective on $t_{n}+1$-neighborhoods and $\dom(\alpha_{\mb{U}_0}) \cap B_{U_0}(r^{m+1},3t_n+2)=\emptyset$. Hence there is an extension $\widehat{\alpha}_{\mb{U}_0}:\dom(\alpha_{\mb{U}_0}) \cup  B_{U_0}(r^{m+1},t_n) \to \{1,2,\dots,n\}$, that satisfies (2) (note that here we use that $n \geq 30\Delta t_nk \geq |V(U_0)|$).
		     
		Now evaluate $\mathcal{A}_n$ on $r^{m+1}$ with labeling $\widehat{\alpha}_{\mb{U}_0}$, and let $\mb{a}$ be the labeling obtained on $(h^{m+1}_j)_{1\leq j \leq \Delta}$, the half-edges incident to $r^{m+1}$. Let $\widehat{\varphi}_{\mb{U}_0}$ be the extension of $\varphi_{\mb{U}_0}$ by these labels.

		\begin{claim}
			$\mb{U}$ is an $(n,30\Delta t_{n}k)$-certificate for $[\mb{H}_\mb{a}((\mb{U}^i_j)_{i,j})]$. 
		\end{claim}
		\begin{proof}
			First we check that it satisfies (1)-(4) from \cref{def:Certificate}, and then we show $\mb{U} \in [\mb{H}_\mb{a}((\mb{U}^i_j)_{i,j})]$.
			
			(1) holds by the inequality \eqref{eq:Size1}.
			
			(2) and (4) hold by the definition of $\alpha_{\mb{U}}=\widehat{\alpha}_{\mb{U}_0}$ and $\varphi_{\mb{U}}=\widehat{\varphi}_{\mb{U}_0}$.
			
			(3) can be shown as follows. Observe that as (3) holds for all $\mb{U}^i_j$, we have $B_U(r^i,5t_n) \cap \dom(\alpha_{\mb{U}^i_j})\subseteq B_{U^i_j}(p_{\mb{U}^i_j},5t_n) \cap \dom(\alpha_{\mb{U}^i_j}) =\emptyset$. Using this, and the fact that by $m\geq 10t_n$ we have $$\mathcal{B}_{U}(r^1,5t_n)\cap \mathcal{B}_{U}(r^{m+1},t_n)=\mathcal{B}_{U}(r^{2m+1},5t_n)\cap \mathcal{B}_{U}(r^{m+1},t_n)=\emptyset,$$ we get $B_U(p_{\mb{U}},5t_n) \cap \dom(\alpha_{\mb{U}})=B_U(q_{\mb{U}},5t_n) \cap \dom(\alpha_{\mb{U}})=\emptyset$.
			
			Finally, \[\mb{H}_\mb{a}((\mb{U}^i_j)_{i,j}) \simm \mb{H}_\mb{a}((\mb{U}^{\Lambda(i)}_j)_{i,j}) \simm \mb{U},\]
			where the first equivalence holds by the choice of $\Lambda$, while the second holds by \cref{cl:basic} and the fact that $\varphi_{\mb{H}_\mb{a}((\mb{U}^{\Lambda(i)}_j)_{i,j})} =\varphi_\mb{U}$. 
		\end{proof}
		This finishes the proof of the theorem. 
	\end{proof}
	\begin{corollary} 
		\label{cor:existscertif}
		Assume $k \geq 1$ and $n$ is such that $10t_n \geq 2\ell+1$ and $n \geq 30\Delta t_nk$. Let $S \neq \emptyset$ be a collection of equivalence classes of bipolar trees and assume that every $[\mb{T}] \in S$ admits an $(n,k)$-certificate. Then there exists an $f:S^{\sizeofarray} \to \Sigma^\Delta$ so that every element of $[\Gamma_{\ell,f}](S)$ admits an $(n,30\Delta t_nk)$-certificate.
	\end{corollary}
	\begin{proof}
		Let $[\mb{T}^i_j]_{i,j} \in S^{\sizeofarray}$ be any array. We define $f([\mb{T}^i_j]_{i,j})=\mb{a}$, where $\mb{a} \in \Sigma^\Delta$ is provided by \cref{thm:existscertif}. By the definition of $f$, we have that  $[\mb{H}_\mb{a}((\mb{T}^i_j)_{i,j})]=[\mb{H}_f((\mb{T}^i_j)_{i,j})]$ admits an $(n,30\Delta t_nk)$-certificate. Now, combining the facts that every $[\mb{T}] \in [\Gamma_{\ell,f}](S)$, arises in a form of $\mb{H}_f$ or $\bar{\mb{H}}_f$, and that if $\mb{U}$ is a certificate for $[\mb{T}]$ then $\bar{\mb{U}}$ is one for $[\bar{\mb{T}}]$, we are done.
	\end{proof}
	\subsubsection{Proof of \cref{thm:MainConstruction}}
	\label{s:finalconstruct}
	
	We are ready to finish the proof of our fixed point theorem. Let $N$ be the number of $\simm$ equivalence classes, and $n$ be a number with $10t_n \geq 2\ell+1$ and $(30\Delta t_n)^N<n$. Such a number exists by the assumption $(t_n) \in n^{o(1)}$. Define 	
	\[X_{m}=\{[\mb{T}]:[\mb{T}] \text{ admits an $(n,(30\Delta t_n)^m)$-certificate}\}.\] Clearly, the sets $(X_{m})_{m \in \N}$ are increasing and $X_0 \neq \emptyset$ by \cref{cl:sym}. Since $(30\Delta t_n)^N<n$, there is some $m_0 < N$ so that $X_{m_0}=X_{m_0+1}$. Let $X'=X_{m_0}$ and use \cref{cor:existscertif} to find a function $f:X'^{\sizeofarray} \to \Sigma^\Delta$, so that every class in $[\Gamma_{\ell,f}](X')$ admits a $(n,({30\Delta t_n})^{m_0+1})$-certificate. Then $[\Gamma_{\ell,f}](X') \subseteq X_{m_0+1}=X_{m_0}=X'$. Using this, and  using \cref{cl:well-defd} we get that $[\Gamma_{\ell,f}]: \mathcal{P}(X') \setminus \emptyset \to  \mathcal{P}(X') \setminus \emptyset$ is monotone, hence by \cref{pr:UniformBound0}, there is an $X \subseteq X'$ non-empty fixed point. 
	
	In order to see that every equivalence class in $X$ is valid, note that each of those has a $(n,(30\Delta t_n)^{m_0})$-certificate, and by the choice of $n$ we have $n \geq (30\Delta t_n)^{m_0}$, which shows validity by \cref{lem:Certificate}.

	\subsection{The $\ell_\Pi$-full condition from \cref{thm:MainConstruction}}
	
	In this section we show that the $\ell_\Pi$-full condition holds for $\Pi$.

	\subsubsection{Observations on $\ell$-fullness}
	\label{ss:observ}
	Let us start with some easy observations about $\ell$-fullness, which will point towards the general direction of our argument. If $\varphi$ is a function with finite range, we will denote by $\overline{ran}(\varphi)$ the multiset determined by $\varphi$, that is, every element is enumerated as many times as the size of its $\varphi$ preimage. If $\mb{a}\in \Sigma^\Delta$, we will denote by $\overline{\mb{a}}$ the multiset $\overline{ran}(\mb{a})$.

	\begin{definition}
		Let $\mb{a}=(a_1,\dots,a_\Delta),\mb{b}=(b_1,\dots,b_\Delta) \in \Sigma^\Delta$, $G, H \subseteq \{1,\dots,\Delta\}$ and $S \subseteq \N$. We say that \emph{$(\mb{a},\mb{b})$ is realizable with length in $S$, using $G$ and $H$}, if for all $i \in G$, $j \in H$ and $l \in S$ there is a $\Pi$-coloring $\varphi$ of the spiky path $P$ of length $l$, so that
		
		\begin{itemize}
			\item if $g$ and $h$ are the starting and end half-edges of $P$, then $\varphi(g)=a_i$, $\varphi(h)=b_j$,
			\item if $v$ and $w$ are the start and end vertices of $P$, then $\overline{\bf a}=\overline{ran}(\varphi\upharpoonright N(v))$ and $\overline{\bf b}=\overline{ran}(\varphi\upharpoonright N(w))$.
		\end{itemize}
		In symbols, \realized{\mb{a}}{\mb{b}}{S}{G}{H}. 
		
		If $\mathcal{V}'$ is a collection of multisets, and in addition to the above, for all vertices $v$ of $P$ we have $\overline{ran}(\varphi\upharpoonright N(v)) \in \mathcal{V}'$, we say that this realization is possible in $\mathcal{V}'$, in symbols \realized{\mb{a}}{\mb{b}}{S}{G}{H} in $\mathcal{V}'$.
		
	\end{definition}
	Note that \realized{\mb{a}}{\mb{b}}{S}{G}{H} in $\mathcal{V}'$ implies that $\mb{a},\mb{b} \in \mathcal{V}'$. When using this notation, if it is clear from the context, we will write $\geq n$ for $\{l \in \N: l \geq n\}$. 
	
	Now, in the above terminology a set $\mathcal{V}'$ being $n$-full is equivalent to $\forall \mb{a},\mb{b} \in \Sigma^\Delta$ with $\overline{\mb{a}},\overline{\mb{b}} \in \mathcal{V}'$ we have \realized{\mb{a}}{\mb{b}}{\geq n}{\{1,\dots,\Delta\}}{\{1,\dots,\Delta\}} in $\mathcal{V}'$. Our goal is to find sufficient conditions for this to hold.
	\begin{claim}
		\label{cl:pathglue}
		Assume that $l,k>1$, $i \neq j \in \{1,\dots,\Delta\}$, \realized{\mb{a}}{\mb{b}}{l}{G}{\{i\}} in $\mathcal{V}'$ and \realized{\mb{b}}{\mb{c}}{k}{\{j\}}{H} in $\mathcal{V}'$. Then \realized{\mb{a}}{\mb{c}}{l+k-1}{G}{H} in $\mathcal{V}'$. 
	\end{claim}
	\begin{proof}
		From $\Pi$-colorings $\varphi$ and $\psi$ of paths of length $l$ and $k$ we can create a $\Pi$-coloring of the path of length $l+k-1$, by gluing together the paths' end and starting vertices colored by $\mb{b}$: note that by the assumption $i \neq j$ this is well defined. 
	\end{proof}
	
	\begin{claim}
		\label{cl:pathglue2}
		Assume that $n_0>1$ and \realized{\mb{a}}{\mb{b}}{n_0,n_0+1}{\{1,\Delta\}}{\{1,\Delta\}} in $\mathcal{V}'$. Then 
		
		\begin{enumerate}
			\item \realized{\mb{b}}{\mb{b}}{2n_0-1,2n_0}{\{1,\Delta\}}{\{1,\Delta\}} in $\mathcal{V}'$. 
			\item \realized{\mb{b}}{\mb{b}}{\geq 4n^2_0}{\{1,\Delta\}}{\{1,\Delta\}} in $\mathcal{V}'$.
			\item \realized{\mb{a}}{\mb{b}}{\geq 5n^2_0}{\{1,\Delta\}}{\{1,\Delta\}} in $\mathcal{V}'$.
		\end{enumerate} 
	\end{claim}
	\begin{proof}
		\begin{enumerate}
			\item Observe that, by $P$ being symmetric, \realized{\mb{a}}{\mb{b}}{n_0}{\{1,\Delta\}}{\{1,\Delta\}} implies \realized{\mb{b}}{\mb{a}}{n_0}{\{1,\Delta\}}{\{1,\Delta\}}. Now the claim follows form \cref{cl:pathglue}.
			\item Inductively on $k',k'' \in \N$ we show that \realized{\mb{b}}{\mb{b}}{k'(2n_0-2)+k''(2n_0-1)+1}{\{1,\Delta\}}{\{1,\Delta\}}. Indeed, using the first statement and \cref{cl:pathglue}, for $k'$ it is enough to notice that 
			\[k'(2n_0-2)+k''(2n_0-1)+1+2n_0-1-1=(k'+1)(2n_0-2)+k''(2n_0-1)+1,\]
			and the analogous calculation allows to increase $k''$ by $1$ as well\footnote{Note that just for this calculation, it would me more convenient to use the number of edges for the length of a path, but for other calculations the number of vertices works better.}. 
			
			Now, by classical estimates on the Frobenius number \cite{sylvester1882subvariants}, every number \[\ge (2n_0-2)(2n_0-1)-(2n_0-2)-(2n_0-1)+1\] can be expressed in the form $k'(2n_0-2)+k''(2n_0-1)+1$, with $k', k'' \in \N$, showing in particular our estimate. 
			\item Now, using \cref{cl:pathglue} again and $4n_0^2+n_0 \le 5n_0^2$, yields our estimate.
		\end{enumerate}
	\end{proof}
	
	\begin{proposition}
		\label{pr:pathgluefinal}
		Assume that there are some $\emptyset \neq\mathcal{V}^* \subseteq \mathcal{V}'$ and $n_0,n_1$ such that
		\begin{enumerate}
			\item for all $\mb{a},\mb{b}$ with $\overline{\mb{a}},\overline{\mb{b}} \in \mathcal{V}^*$ we have \realized{\mb{a}}{\mb{b}}{n_0,n_0+1}{\{1,\Delta\}}{\{1,\Delta\}} in $\mathcal{V}'$,
			\item for all $\mb{c}$ with $\overline{\mb{c}} \in \mathcal{V}'$ and $i \in \{1,\dots,\Delta\}$ there are $\mb{a}$ with $\overline{\mb{a}} \in \mathcal{V}^*$, and $i' \in \{1,\dots,\Delta\}$ with  \realized{\mb{c}}{\mb{a}}{n'}{\{i\}}{\{i'\}} in $\mathcal{V}'$ for some $n' \leq n_1$.
			
		\end{enumerate}
		Then $\mathcal{V}'$ is $(5n^2_0+2n_1)$-full.
		
	\end{proposition}
	\begin{proof}
		As mentioned before \cref{cl:pathglue}, we have to show that for all $\mb{c}$, $\mb{d}$ with $\overline{\mb{c}},\overline{\mb{d}} \in \mathcal{V}'$ we have \realized{\mb{c}}{\mb{d}}{\geq \ell}{\{1,\dots,\Delta\}}{\{1,\dots,\Delta\}} in $\mathcal{V}'$, where $\ell=5n^2_0+2n_1$. To show this, pick any $i,j \in \{1,\dots,\Delta\}$. Then, by our assumptions there are $\mb{a},\mb{b}$ with $\overline{\mb{a}},\overline{\mb{b}} \in \mathcal{V}^*$, $i',j' \in \{1,\dots,\Delta\}$ with \realized{\mb{c}}{\mb{a}}{n'}{\{i\}}{\{i'\}} and \realized{\mb{b}}{\mb{d}}{n''}{\{j'\}}{\{j\}} in $\mathcal{V}'$, where $n',n'' \leq n_1$. By the first assumption and \cref{cl:pathglue2} we have \realized{\mb{a}}{\mb{b}}{\geq 5n^2_0}{\{1,\Delta\}}{\{1,\Delta\}} in $\mathcal{V}'$. 
		
		Finally, we have \realized{\mb{a}}{\mb{b}}{\ge\ell-n'-n''}{\{1,\Delta\}}{\{1,\Delta\}} in $\mathcal{V}'$, and in particular \realized{\mb{a}}{\mb{b}}{\ge\ell-n'-n''+2}{\{i''\}}{\{j''\}} in $\mathcal{V}'$  with some $i'' \in \{1,\Delta\}\setminus \{i'\}$ and $j'' \in \{1,\Delta\}\setminus \{j'\}$. Putting these relations together, by \cref{cl:pathglue} we have \realized{\mb{c}}{\mb{d}}{\ge\ell}{\{i\}}{\{j\}} in $\mathcal{V}'$, as desired.
	\end{proof}

	\subsubsection{Finding the $\ell_\Pi$-full set}

	In this section, we use \cref{thm:MainConstruction} to show that $\Pi$ admits an $\ell_\Pi$-full set $\mathcal{V}'\subseteq \mathcal{V}$ for some $\ell_\Pi>0$. In order to do that, we fix $\ell>0$, $X$ and $f$ as in \cref{thm:MainConstruction}. 
	
	\begin{definition}[The family $W_i$]
		\label{def:famW}
		Set $W_0=\{\mb{T}: [{\bf T}] \in X, D_{\bf{T}}=\emptyset, \rn_{{\bf T}}\equiv 0\}$ and define $W_{m}=\Gamma_{\ell,f}^m(W_0)$.

	\end{definition}
	
	Note that this is well defined, as $\Gamma_{\ell,f}(W)$ makes sense, whenever $[W]^{\sizeofarray} \subseteq \dom(f)$, and here $[W_i]=[W_0] \subseteq \dom(f)$ by \cref{thm:MainConstruction}. Note also that the set $W_0$ is nonempty, since $\simm$ does not depend on $\rn_{\mb{T}}$ and $D_{\mb{T}}$ and $X \neq \emptyset$. 
	
	\medskip

	We start with the following immediate observations that follow directly from the construction and inductive definition of rank.
	
	\begin{claim}\label{cl:ObservationOne}
		Let $m\in \mathbb{N}$, $\mb{T}\in W_m$ and $1\le k\le m$.
		
		\begin{enumerate}
			\item The set of vertices of rank $k$ form a disjoint collection of paths each of length exactly $2\ell+1$ and each containing a unique vertex from $D_{\bf T}$ as one of its middle point.
			\item If $\{v,w\}\in E(T)$ and $\rn_{\bf T}(v)=k$, then $\rn_{\bf T}(w)\in \{k-1,k,k+1\}$.
			\item If $\rn_{\bf T}(v)=k>1$ then there is a $w$ with $\{v,w\}\in E(T)$ and $\rn_{\bf T}(w)=k-1$. 
		\end{enumerate}
		
	\end{claim}
	
	\begin{definition}[$P_{\mb{T}}$]
		An edge $\{v,w\}$ of a bipolar tree $\mb{T}$ is called \emph{negative}, if $v\in D_{\mb{T}}$, $\rn_{\mb{T}}(v)=\rn_{\mb{T}}(w)$, and $w$ is not on the path connecting $v$ to the positive pole $p_{\mb{T}}$. Let $P_{\mb{T}}$ be the collection of vertices $v$ so that the path from $v$ to $p_{\mb{T}}$ contains no negative edges. 
	\end{definition}
	
	The next lemma shows the importance of this collection of vertices, they, in particular avoid vertices with virtual edges.
	
	\begin{lemma}\label{lm:NoVirtue}
		Let $m \in \N$, $\mb{T}\in W_m$. Then for any $v \in P_{\mb{T}}$ with $\rn_{\mb{T}}(v)>0$ we have
		\begin{itemize}
			\item  $\Nv(v)=\emptyset$ or
			\item $v=p_{\mb{T}}$.
		\end{itemize} 
		In particular, if $0<\rn_{\mb{T}}(v)<m$, then $\Nv(v)=\emptyset$.
	\end{lemma}
	\begin{proof}	
		
		We show this by induction on $m$. If $m=0$, then there is nothing to prove. Now assume that statement for $m$ and that $\mb{T}\in W_{m+1}$. If $0<\rn_{\mb{T}}(v)<m+1$, then $v$ is a vertex of some ${\bf T}^i_j\in W_{m}$ with rank $<m+1$. Observe that the path from $v$ to the positive pole of $\mb{T}$ must go through the positive pole of $\mb{T}^i_j$ and $D_{\mb{T}^i_j}\subset D_{\mb{T}}$, so $v \in P_{\mb{T}^i_j}$ holds as well. Thus, by the inductive hypothesis, it must satisfy one of the two options above in ${\bf T}^i_j$. Now, in the first case we are done, while in the second case, i.e., when $v=p_{\mb{T}^i_j}$, then $\Nv(v)=\emptyset$ will hold in $\mb{T}$ by the gluing construction \cref{def:kglue}.
		
		Finally, if $\rn_{\mb{T}}(v)=m+1$, $v \neq p_{\mb{T}}$, and $\Nv(v) \neq \emptyset$, then $v=q_{\mb{T}}$. But then the path from $v$ to $p_{\mb{T}}$ contains the middle $r^{\ell+1} \in D_{\mb{T}}$ and the negative edge belonging to it.
	\end{proof}

	The next lemma says that we can reach a point in $D_\mb{T}$ from essentially any point using a path that only goes through vertices of essentially the same rank.
	
	\begin{lemma}\label{lm:PathToDesign}
		Let $m \in \N$, ${\bf T}\in W_m$. Assume that $2\le k< m$, $v \in P_{\mb{T}}$ with $\rn_{\bf T}(v)=k$, $h\in N^H(v)$ and 
		\begin{enumerate}
			\item $v \not \in D_{\mb{T}}$ or
			\item $v \in D_{\mb{T}}$ and $h$ is not contained in a negative edge.
		\end{enumerate}
		Then there is a $w\in D_{\bf T}$, $2 \le l\le \ell+3$ and a path
		$P=(v=v_1,\dots,v_l=w)$ in $\mb{T}$ of vertices in $P_{\mb{T}}$ that starts at $h$ and $\rn_{\bf T}(v_j)\in \{k-1,k,k+1\}$ for every $1\le j\le l$.
	\end{lemma}
	\begin{proof}
		By \cref{lm:NoVirtue}, $\Nv(v)=\emptyset$, so there is some $v_2\in V(T)$ such that $h$ belongs to $\{v_1,v_2\}\in E(T)$. Observe that for any $v \in P_{\mb{T}} \setminus D_{\mb{T}}$ we have $N^{H}(v) \subseteq P_{\mb{T}}$, therefore, in either case of the lemma, $v_2 \in P_{\mb{T}}$ holds as well.
		
		Now, if $v_2 \in D_{\mb{T}}$ then we are done by \cref{cl:ObservationOne} (2), so assume that this is not the case. 
		
		By \cref{cl:ObservationOne} every vertex $u \in V(T)$ with $\rn_{\mb{T}}(u) \geq 1$ belongs to a path of length $2\ell+1$ consisting of vertices with rank equal to $\rn_{\mb{T}}(u)$, and 
		having a vertex from $D_\mb{T}$. Let us denote the subpath of this path connecting $u$ to the element of $D_{\mb{T}}$ by $P_u$. Note that if $u \in P_{\mb{T}}$, then every element of $P_u$ is in $P_{\mb{T}}$ as well. 
		
		To see the statement, note that by \cref{cl:ObservationOne} we have $|\rn_{\mb{T}}(v_2) -\rn_{\mb{T}}(v_1)| \leq 1$, hence $\rn_{\mb{T}}(v_2) \geq 1$. Thus, if $\rn_{\mb{T}}(v_2) \neq \rn_{\mb{T}}(v_1)$, we can let $P$ be the concatenation of $v_1$ and $P_{v_2}$. Finally, if $\rn_{\mb{T}}(v_2)=\rn_{\mb{T}}(v_1)$, by \cref{cl:ObservationOne} and our assumption on $v_2$ (i.e., that $v_2 \not \in D_\mb{T}$) we can find a $v_3 \in P_{\mb{T}}$ with $(v_2,v_3) \in E(T)$ and $1 \leq \rn_{\mb{T}}(v_3)=\rn_{\mb{T}}(v_2)-1$, so letting $P$ be the concatenation of $(v_1,v_2)$ and $P_{v_3}$ works by $v_3 \in P_{\mb{T}}$. 
	\end{proof}

	\begin{definition}
		For $i \leq m$, let \[\mathcal{V}_{i,m}=\{\alpha \in \mathcal{V}:\exists \mb{T} \in W_m \ \exists v \in P_{\mb{T}} \ \exists \widehat{\varphi}_\mb{T}\supseteq \varphi_\mb{T} \] \[ \widehat{\varphi}_\mb{T} \text{ is a $\Pi$-coloring, $\rn_{\mb{T}}(v)=i$, and $\alpha=\overline{ran}(\varphi_{\mb{T}}\restriction N^{H}(v))$} \},\]
		let \[\mathbb{V}=ran(f\restriction [W_0]^{\sizeofarray}),\]
		and \[\mathcal{V}^*=\{\overline{\alpha}:\alpha \in \mathbb{V}\}.\]
	\end{definition}
	
	\begin{claim}
		\label{cl:subset}
		Assume that $0 <i \le m$, $\mb{a}=(a_1,\dots,a_\Delta) \in \mathbb{V}$ and $1 \leq j<\Delta$. Then there exists a tree $\mb{T} \in W_m$, a vertex $v \in D_{\mb{T}} \cap P_{\mb{T}}$ with $\rn_{\mb{T}}(v)=i$ so that $\varphi_{\mb{T}}(h)=a_j$ holds for some $h \in \Nr(v)$ not belonging to a negative edge and $\overline{ran}(\varphi_\mb{T} \restriction N^H(v))=\bar{a}$. 
		
		In particular, $\mathcal{V}^* \subseteq \mathcal{V}_{i,m}$ for all $0<i \leq m$. 
	\end{claim}
	\begin{proof}
		
		Since $[W_0]=[W_m]$ for all $m$, we have $ran(f\restriction [W_0]^{\sizeofarray})=ran(f\restriction [W_m]^{\sizeofarray})$ as well. Therefore, for $i=m$, the statement follows from the construction of the elements of $W_m$.  
		
		Now we show the rest of the statement by induction on $m$. The case $m=i=1$ has been already handled. Now, assume that we have shown the claim for $m$, $i<m+1$, and $\mb{a}$ is such that $\mb{a}\in \mathbb{V}$. By the inductive hypothesis, we can find a tree $\mb{T}' \in W_m$ together with a vertex $v \in D_{\mb{T}'} \cap P_{\mb{T}'}$ as above. Now, let $(\mb{T}^i_j)_{i,j} \in W^{\sizeofarray}_m$ be an arbitrary array with $\mb{T}^1_2=\mb{T}'$, and $\mb{T}=\mb{H}_{f}((\mb{T}^i_j)_{i,j})$. Then $v \in P_{\mb{T}} \cap D_{\mb{T}}$. Note that the negative edge in $\mb{T}$ adjacent to $v$ is also negative in $\mb{T}'$ by the fact the path from $v$ to $p_{\mb{T}}$ goes through $p_{\mb{T}'}$. Hence, $\varphi_{\mb{T}}(h)=a_j$ must hold on some $h$ not belonging to such an edge.
	\end{proof}

	\begin{lemma}
		\label{lm:DesignFlex0}
		Let ${\bf a}=(a_1,\dots,a_\Delta)\in \mathbb{V}$, $p \in \{1,\Delta\}$ and $0<k$. There are a tree ${\bf T}_{{\bf a},p} \in W_k$ and a vertex $v \in D_{\mb{T}_{\mb{a},p}} \cap P_{{\mb{T}_{\mb{a},p}}}$ such that
		\begin{enumerate}
			\item $\rn_{{\bf T}_{{\bf a},p}}(v)=k$,
			\item $\overline{\bf a}=\overline{ran}(\varphi_{\mb{T}_{{\bf a},p}}\upharpoonright N^H(v))$,
			\item the path $Q_v$ from $v$ to the positive pole of ${\bf T}_{{\bf a},p}$ has length $\ell+1$, and contains vertices of rank $k$ (necessarily from $P_{{\mb{T}_{\mb{a},p}}}$),
			\item if $Q_v$ starts with $h$, then $\varphi(h)=a_p$. 
		\end{enumerate}    
		
	\end{lemma}
	
	\begin{proof}
		By definition, there is a sequence $(\mb{T}^i_j)_{i,j}\in W_{k-1}^{\sizeofarray}$, such that the output of $f$ on $({\bf T}^i_j)_{i,j}$ is $\mb{a}$. 
		
		For $p=1$, let $\mb{T}_{\mb{a},p}=\mb{H}_f((\mb{T}^i_j)_{i,j})$, and for $p=\Delta$, let $\mb{T}_{\mb{a},p}=\bar{\mb{H}}_f((\mb{T}^i_j)_{i,j})$. In both cases, let $v=r^{\ell+1}$, that is, the middle vertex. Now, all the conditions of the lemma hold by the definition of $\mb{H}_f((\mb{T}^i_j)_{i,j})$. 
	\end{proof}
	The next lemma will be a partial step towards establishing $\ell$-fullness, saying that certain distances between configurations around distinguished vertices can be realized. 
	
	\begin{lemma}\label{lm:DesignIsFlexible}
		There exists $n_0>0$ so that for all ${\bf a}=(a_1,\dots,a_\Delta)\in \mathbb{V}$, ${\bf b}=(b_1,\dots,b_\Delta)\in \mathbb{V}$, $0<k$, $p,q\in \{1,\Delta\}$,
		and for each $l \in \{0,1\}$ there are a tree ${\bf T}^l_{{\bf a},{\bf b},p,q} \in W_{k+1}$, and $v,w\in D_{{\bf T}^l_{{\bf a},{\bf b},p,q}} \cap P_{{\bf T}^l_{{\bf a},{\bf b},p,q}}$ such that
		\begin{enumerate}
			\item $\rn_{{\bf T}^l_{{\bf a},{\bf b},p,q}}(v)=\rn_{{\bf T}^l_{{\bf a},{\bf b},p,q}}(w)=k$,
			\item $\overline{\bf a}=\overline{ran}(\varphi_{{\bf T}^l_{{\bf a},{\bf b},p,q}}\upharpoonright N^H(v))$ and $\overline{\bf b}=\overline{ran}(\varphi_{{\bf T}^l_{{\bf a},{\bf b},p,q}}\upharpoonright N^H(w))$,
			\item the (unique) path $P=(v=v_1,\dots,v_j=w)$ has length exactly $n_0+l$, $\rn_{{\bf T}^l_{{\bf a},{\bf b},p,q}}(v_i)\in \{k,k+1\}$ and $v_i \in P_{{\bf T}^l_{{\bf a},{\bf b},p,q}}$ for every $1\le i\le j$,
			\item if $h\in N^H(v)$ and $g\in N^H(w)$ are such that $P$ starts at $h$ and ends at $g$, then $\varphi(h)=a_p$ and $\varphi(g)=b_q$. 
		\end{enumerate}    
	\end{lemma}
	\begin{proof}
		
		Let us fix trees for $\mb{a},\mb{b}$ as in \cref{lm:DesignFlex0}, ${\bf T}_{{\bf a},p}, {\bf T}_{{\bf b},q} \in W_k$ together with vertices $v,w$. Let $n_0=2\ell+4$. We construct an element ${\bf T}^l_{{\bf a},{\bf b},p,q}$ to be of the form ${\bf H}_f$ by considering any array $(\mb{T}^i_j)_{i,j} \in W^{\sizeofarray}_{k}$ with the property that
		$$\mb{T}^1_2={\bf T}_{{\bf a},p} \text{ and }
		\mb{T}^{l+2}_2={\bf T}_{{\bf b},q},$$
		and let $v$ and $w$ be the corresponding copies from  ${\bf T}_{{\bf a},p}$ and ${\bf T}_{{\bf b},q}$. 
		Let us check that (1)-(4) holds for ${\bf T}^l_{{\bf a},{\bf b},p,q}$ and $v,w$. Indeed, the first three conditions are clear from its definition and the corresponding clauses of \cref{lm:DesignFlex0}. Observe also that the path $P=(v=v_1,\dots,v_{n_0+l}=w)$ that connects $v$ and $w$ in ${\bf T}'$ has the form $Q_v$ concatenated with $(r^1,\dots,r^{l+2})$ concatenated with the reverse of $Q_w$, hence (4) holds as well by \cref{lm:DesignFlex0}.
	\end{proof}
	
	This construction allows us to deduce the following. Recall the notation from \cref{ss:observ}.
	
	\begin{proposition}
		\label{pr:vworks}
		There exists an $n_0>0$ such that for all $1\leq k<m$, $\mb{a},\mb{b} \in \mathbb{V}$, $l \in \{0,1\}$, and $p,q \in \{1,\Delta\}$ there are a tree ${\bf T}^l_{{\bf a},{\bf b},p,q}$, and $v,w\in D_{{\bf T}^l_{{\bf a},{\bf b},p,q}} \cap P_{{\bf T}^l_{{\bf a},{\bf b},p,q}}$ satisfying (1)-(4) from \cref{lm:DesignIsFlexible}, but with ${\bf T}^l_{{\bf a},{\bf b},p,q} \in W_{m}$. Consequently, for all $\mb{a}, \mb{b} \in \mathbb{V}$ we have \realized{\mb{a}}{\mb{b}}{n_0,n_{0}+1 }{\{1,\Delta\}}{\{1,\Delta\}} in $\mathcal{V}_{k,m}\cup \mathcal{V}_{k+1,m}$.
	\end{proposition}
	
	\begin{proof}
		Note that for $m=k+1$, we have already shown the first statement, form which the ``consequently" part follows: indeed, by \cref{lm:DesignIsFlexible}, and the fact that the partial labeling of a tree in $W_{k+1}$ extend to a full $\Pi$-coloring, using the paths (with half-edges) guaranteed by (3) of \cref{lm:DesignIsFlexible} we get \realized{\mb{a}}{\mb{b}}{n_0,n_{0}+1 }{\{1,\Delta\}}{\{1,\Delta\}}. To see that the coloring is actually in $\mathcal{V}_{k,m}\cup \mathcal{V}_{k+1,m}$, note that the rank of the intermediate vertices is in $\{k,k+1\}$ as well, so this follows from the definition of $\mathcal{V}_{k,m}$. 
		
		We show the remaining cases by induction on $m$, the $m=2$ case is done. Now, given $\mb{T}^l_{\mb{a},\mb{b},p,q} \in W_m$, we construct $\mb{T}^{l'}_{\mb{a},\mb{b},p,q} \in W_{m+1}$, using any array $(\mb{T}^i_j)_{i,j}$, with $\mb{T}^1_2=\mb{T}^l_{\mb{a},\mb{b},p,q}$. This clearly works. 
	\end{proof}
	
	\begin{proposition}
		\label{pr:v'works}
		Assume that $1\le k<m-1$. There exists a $n_1>0$ so that for all $\mb{c}=(c_1,\dots,c_\Delta)$ with $\overline{\mb{c}} \in \mathcal{V}_{k+1,m}$ and for all $i \in \{1,\dots,\Delta\}$ there is an $\mb{a}$ with $\overline{\mb{a}} \in \mathcal{V}^*$ and an $i' \in \{1,\dots,\Delta\}$ 
		so that \realized{\mb{c}}{\mb{a}}{n'}{\{i\}}{\{i'\}} in $\mathcal{V}_{k,m}\cup \mathcal{V}_{k+1,m} \cup \mathcal{V}_{k+2,m}$ for some $n' \leq n_1$. 
	\end{proposition}
	\begin{proof}
		
		Let $n_1=\max \{n_0,\ell+3\}$. If $\mb{c} \in \mathbb{V}$ and $i=\Delta$ then the claim follows from \cref{pr:vworks}, so assume that this is not the case. We claim that there exists some tree $\mb{T} \in W_m$, $v \in P_{\mb{T}}$ with $\rn_{\mb{T}}(v)=k+1$ and a $\Pi$-coloring $\widehat{\varphi}_{\mb{T}}\supseteq \varphi_\mb{T}$ so that $\bar{\mb{c}}=\overline{ran}(\widehat{\varphi}_\mb{T} \restriction N^H(v))$ and $c_i=\widehat{\varphi}_\mb{T}(h)$, where $h \in \Nr(v)$ is a half-edge not belonging to a negative edge: indeed, for $\mb{c} \in \mathbb{V}$ this follows from \cref{cl:subset}, and otherwise, this follows from the definition of $\mathcal{V}_{k+1,m}$  and the fact that negative edges are not adjacent to vertices $\not \in D_{\mb{T}}$. 
		
		By \cref{lm:PathToDesign} there is a path of length $n'\leq \ell+3$ that connects $v$ with $w\in D_{\bf T}$, only using vertices from $P_{\mb{T}}$ of rank in the set $\{k,k+1,k+2\}$. Note also that by definition $\overline{ran}(\widehat{\varphi}_\mb{T} \restriction N^H(w)) \in \mathcal{V}^*$. Hence, by \cref{pr:LargeM}, considering $P$ as spiky path together with $\widehat{\varphi}_\mb{T}$, it witnesses  \realized{\mb{c}}{\mb{a}}{n'}{\{i\}}{\{i'\}} in $\mathcal{V}_{k,m}\cup \mathcal{V}_{k+1,m} \cup \mathcal{V}_{k+2,m}$ for some $i' \in \{1,\dots,\Delta\}$.
	\end{proof}

	\begin{proposition}\label{pr:LargeM}
		Let $0<m\in \mathbb{N}$ and $1\le k< m$.
		Then we have $\mathcal{V}_{k,m}\subseteq \mathcal{V}_{k+1,m}$. In particular, there exist $M\ge 3$ and $1\le k_0\le M-2$ so that $\emptyset \neq \mathcal{V}_{k_0,M}=\mathcal{V}_{k_0+1,M}=\mathcal{V}_{k_0+2,M}$.
	\end{proposition}
	\begin{proof}
		
		For the first assertion, if $m=1$ then there is nothing to prove, so we assume $m \ge 2$. Assume that $\bar{a}\in \mathcal{V}_{k,m}$ is witnessed by an extension $\widehat{\varphi}_{\mb{T}}$ of the partial labeling $\varphi_\mb{T}$ and $v\in V(T)$ such that $\rn_{\bf T}(v)=k$, where ${\bf T}\in W_m$. Our task is to find another tree ${\bf S}\in W_{m}$ and a vertex $u \in P_{\mb{S}}$ with $\rn_{\mb{S}}(u)=k+1$, such that there is an extension of $\varphi_{\mb{S}}$ so that around $u$ we have the label multi-set $\bar{a}$.

		We modify ${\bf T}$ to ${\bf S}$ as follows. By definition $\mb{T}$ arose from an array $(\mb{T}^i_j) \in W_{m-1}^{\sizeofarray}$. Then $v$ is contained in exactly one of these trees, say, ${\bf T'}$.  By \cref{def:famW}, each connected component of $\mb{T}'$ formed by half-edges incident to rank $0$ vertices is a tree from $W_0$. To each such tree $\mb{U}$, fix some $\mb{U}' \in W_1$ with $\mb{U}' \simm \mb{U}$: note that this is possible, as by \cref{thm:MainConstruction} $[\Gamma_{\ell,f}]([W_0])=[W_1]$. To obtain $\mb{S}$, replace each $\mb{U}$ by $\mb{U}'$ in $\mb{T}'$ using the corresponding decorations $\mathfrak{d}_{\mb{U}'}$, and let $\rn_{\mb{S}}(w)=\rn_{\mb{T}'}(w)+1$, for every $w \in V(T')$ with $\rn_{\mb{T}'}(w)>0$. Finally, let $D_{\mb{S}}=D_{\mb{T'}} \cup \bigcup_{\mb{U}} D_{\mb{U}'}$. By \cref{cl:replace} for any extension of $\varphi_{\mb{T}'}$ to a $\Pi$-coloring, there exists a $\Pi$-coloring of $\mb{S}$ so that the two maps coincide on each $w$ with $\rn_{\mb{T}'}(w)>0$. In particular, this holds for $\widehat{\varphi}_{\mb{T}} \restriction \mb{T}'$, and, since $\rn_{\mb{S}}(v)=k+1$ and $v \in P_{\mb{T}'}$ by construction, this witnesses $\bar{a} \in \mathcal{V}_{k+1,m}$.

		To see the second statement, note that the sequence $\{\mathcal{V}_{k,m}\}_{k=1}^m$ is increasing and $\mathcal{V}$ is finite, hence, it follows from the pigeonhole principle that for large enough $M\ge 3$, we can find $1\le k_0\le M-2$ that works as required. Finally, the fact that all the sets $\mathcal{V}_{k,m}$ are non-empty is clear, as the partial labeling of each $\mb{T}$ was valid, by \cref{thm:MainConstruction}. 
	\end{proof}

	Now we are left to put together these ingredients and find the $\ell_{\Pi}$-full set.
	Let $M>0$ and $1\le k_0\le M-2$ be as in Proposition~\ref{pr:LargeM} and define $\mathcal{V}'=\mathcal{V}_{k_0,M}$.

	\begin{theorem}
		The set $\mathcal{V}'$ is $\ell_{\Pi}$-full for some $\ell$.
	\end{theorem}
	\begin{proof}
		Let $n_0$ and $n_1$ be the constants coming from \cref{pr:vworks} and \cref{pr:v'works}. Let $\ell_\Pi=5n^2_0+2n_1$. Now the theorem follows from \cref{pr:pathgluefinal}: indeed, observe that by \cref{pr:LargeM} we have $\mathcal{V}'=\mathcal{V}_{k_0,M}=\mathcal{V}_{k_0+1,M}=\mathcal{V}_{k_0+2,M}$. Therefore, \cref{pr:vworks} guarantees the first condition in \cref{pr:pathgluefinal}, while \cref{pr:v'works} guarantees the second condition, so we are done.
	\end{proof}

	\subsection*{Acknowledgements} 
	We would like to thank Anton Bernshteyn, Endre Csóka, Mohsen Ghaffari, Jan Hladký, Steve Jackson, Alexander Kechris, Edward Krohne, Oleg Pikhurko, Brandon Seward, Jukka Suomela, and Yufan Zheng for insightful discussions.

	YC was supported by  Dr.~Max R\"{o}ssler, by the Walter Haefner Foundation, and by the ETH Z\"{u}rich Foundation.
    JG was supported by Leverhulme Research Project Grant RPG-2018-424, by MSCA Postdoctoral Fellowships 2022 HORIZON-MSCA-2022-PF01-01 project BORCA grant agreement number 101105722, and by the Alexander von Humboldt Foundation in the framework of the Alexander von Humboldt Professorship of Daniel Kráľ endowed by the Federal Ministry of Education and Research.
    CG and VR were supported by the European Research Council (ERC) under the European Unions Horizon 2020 research and innovation programme (grant agreement No.~853109). VR was also supported by the Czech Science Foundation (GA~ČR), Junior Star project No.~26-23599M.
	ZV was partially supported by the National Research, Development and Innovation Office	-- NKFIH, grants no.~113047, no.~129211 and FWF Grant M2779.

	\bibliographystyle{alpha}
	\bibliography{ref}

\end{document}

%% file: macros.tex
\usepackage{amscd}
\usepackage{mathrsfs}
\usepackage[IL2]{fontenc}

\usepackage{comment} 

\usepackage{color}
\usepackage{amssymb,amsmath,amsthm,amsfonts,latexsym}
\usepackage[utf8x]{inputenc}
\usepackage{bbm} 

\usepackage{graphicx}
\usepackage[textsize=tiny]{todonotes}

\usepackage{amsmath,amsfonts,amssymb,amsthm} 
\usepackage{tikz}
\usepackage{hyperref}
\usepackage{thmtools}
\usepackage{thm-restate}

\newtheorem{theorem}{Theorem}[section]

\newtheorem*{claim*}{Claim}
\newtheorem*{theorem*}{Theorem}
\newtheorem*{definition*}{Definition}
\newtheorem*{remark*}{Remark}

\newtheorem{corollary}[theorem]{Corollary}
\newtheorem{lemma}[theorem]{Lemma}
\newtheorem{remark}[theorem]{Remark}
\newtheorem{claim}[theorem]{Claim}

\newtheorem{proposition}[theorem]{Proposition}

\newtheorem{definition}[theorem]{Definition}

\newcommand{\comm}[1]{}
\usepackage[capitalize, nameinlink]{cleveref}
\crefname{theorem}{Theorem}{Theorems}
\crefname{proposition}{Proposition}{Propositions}
\crefname{observation}{Observation}{Observations}
\crefname{lemma}{Lemma}{Lemmas}
\crefname{claim}{Claim}{Claims}
\crefname{problem}{Problem}{Problems}
\crefname{conjecture}{Conjecture}{Conjectures}
\crefname{question}{Question}{Questions}
\crefname{example}{Example}{Examples}
\crefname{fact}{Fact}{Facts}

\newcommand{\rake}{{\sf Rake}}
\newcommand{\compress}{{\sf Compress}}

\newcommand{\VC}[1]{{V}_{#1}^\mathsf{C}}
\newcommand{\VR}[1]{{V}_{#1}^\mathsf{R}}

\newcommand{\simm}{\overset{\star}{\sim}}

\renewcommand{\deg}{\operatorname{deg}}
\newcommand{\dom}{\operatorname{dom}}

\newcommand{\free}{\operatorname{Free}}

\newcommand{\local}{\mathsf{LOCAL}}

\newcommand{\LOCAL}{\mathsf{LOCAL}}

\newcommand{\id}{\operatorname{id}}	
\newcommand{\Hv}{H_{\operatorname{v}}}
\newcommand{\Hr}{H_{\operatorname{r}}}
\newcommand{\Nv}{N_{\operatorname{v}}}
\newcommand{\Nr}{N_{\operatorname{r}}}
\newcommand{\rng}{ran}

\newcommand{\CONT}{\mathsf{CONTINUOUS}}
\newcommand{\cont}{\mathsf{CONTINUOUS}}

\newcommand{\TOAST}{\mathsf{TOAST}}

\newcommand{\toast}{\TOAST}

\newcommand{\borel}{\mathsf{BOREL}}

\newcommand{\continuous}{\mathsf{CONTINUOUS}}
\newcommand{\baire}{\mathsf{BAIRE}}

\newcommand{\fA}{\mathcal{A}}

\newcommand{\fD}{\mathcal{D}}
\newcommand{\fE}{\mathcal{E}}

\newcommand{\fG}{\mathcal{G}}

\newcommand{\fT}{\mathcal{T}}

\newcommand{\fV}{\mathcal{V}}

\newcommand{\fZ}{\mathcal{Z}}

\newcommand{\ta}{\mathtt{a}}

\newcommand{\N}{\mathbb{N}}

\newcommand{\dist}{\fT^{\operatorname{dist}}}
\newcommand{\distance}{dist}
\newcommand{\lab}{\fT^{\operatorname{\{0,1\}}}}
\newcommand{\AutR}{\operatorname{Aut}^o(T_\Delta)}
\newcommand{\AutT}{\operatorname{Aut}(T_\Delta)}

\newcommand{\ord}{ord}
\newcommand{\vdeg}{vdeg}
\newcommand{\tzeroone}{\mathcal{T}^{\{0,1\}}_\Delta}
\newcommand{\rn}{\operatorname{rank}}
\newcommand{\mb}{\mathbf}

\def\leukfrac#1/#2{\leavevmode
               \kern.1em
                \raise.9ex\hbox{\the\scriptfont0 ${}_#1$}
                \hskip -1pt\kern-.1em
                /\kern-.15em\lower.10ex\hbox{\the\scriptfont0 ${}_#2$}}

\makeatletter
\def\diam{\mathop{\operator@font diam}\nolimits}
\makeatother